\documentclass[11pt,english]{smfart}

\usepackage{amsmath,amsthm,sfheaders}
\usepackage{xspace,smfthm}
\usepackage[psamsfonts]{amssymb}
\usepackage[latin9]{inputenc}
\usepackage{graphicx,color}
\usepackage{hyperref,fancyhdr}

\usepackage{xcolor}

\theoremstyle{plain} 
\newtheorem{theorem}{\bf Theorem}[section]
\newtheorem{proposition}[theorem]{\bf Proposition}
\newtheorem{Theorem}{\bf Theorem}

\newtheorem{lemma}[theorem]{\bf Lemma}
\newtheorem{corollary}[theorem]{\bf Corollary}
\newtheorem{Corollary}[Theorem]{\bf Corollary}

\newtheorem*{remark}{\bf Remark}

\newtheorem*{notation}{\bf Notations}
\newtheorem{definition}[theorem]{\bf Definition}

\def\C{{\mathbb C}}
\def\N{{\mathbb N}}
\def\R{{\mathbb R}}

\def\Q{{\mathbb Q}}

\def\D{\mathbb{D}}

\def\J{\mathcal{J}}

\def\P{\mathbb{P}}

\def\supp{\textup{supp}}
\def\id{\mathrm{id}}
\def\pe{\textup{ := }}
\def\rat{\textup{Rat}}
\def\Per{\textup{Per}}
\def\bif{\textup{bif}}

\def\J{\mathcal{J}}

\def\Fix{\textup{Fix}}
\def\un{{\underline{n}}}
\def\uw{{\underline{w}}}
\def\cm{\textup{cm}}
\def\fm{\textup{fm}}
\def\tm{\textup{tm}}

\def\cM{\mathcal{M}} 
\def\PSL{\mathrm{PSL}}

\def\and{{\quad\text{and}\quad}}

\title{Hyperbolic components of rational maps: Quantitative equidistribution and counting}
\author{Thomas Gauthier}
\address{LAMFA, UPJV, 33 rue Saint-Leu, 80039 AMIENS Cedex 1, FRANCE}
\address{CMLS, Ecole Polytechnique, 91128 PALAISEAU Cedex, FRANCE}
\email{thomas.gauthier@u-picardie.fr}
\author{Y\^usuke Okuyama}
\address{Division of Mathematics, Kyoto Institute of Technology, Sakyo-ku, Kyoto 606-8585 JAPAN}
\email{okuyama@kit.ac.jp}
\author{Gabriel Vigny}
\address{LAMFA, UPJV, 33 rue Saint-Leu, 80039 AMIENS Cedex 1, FRANCE}
\email{gabriel.vigny@u-picardie.fr}
\thanks{The first and third authors' research is partially supported by the ANR grant Lambda ANR-13-BS01-0002.}
\thanks{The second author's research is partially supported by JSPS Grant-in-Aid 
for Scientific Research (C), 15K04924}
\thanks{The second author would also like to thank the LAMFA and the Universit\'e de Picardie Jules Verne, which he was visiting in spring 2016 and where this work grew up}

\begin{document}

\begin{abstract}
Let $\Lambda$ be a quasi-projective variety and assume that, either $\Lambda$ is a subvariety of the moduli space $\cM_d$ of degree $d$ rational maps, or $\Lambda$ parametrizes an algebraic family $(f_\lambda)_{\lambda\in\Lambda}$ of degree $d$ rational maps on $\P^1$. We prove the equidistribution of parameters having $p$ distinct 
neutral cycles towards the bifurcation current $T_\bif^p$ letting the periods of the cycles go to $\infty$, with an exponential speed of convergence. Several consequences of this result are:
\begin{itemize}
\item a precise asymptotic of the number of hyperbolic components of parameters admitting $2d-2$ distinct attracting cycles of exact periods $n_1,\dots, n_{2d-2}$ as $\min_j n_j \to \infty$ in term of the mass of the bifurcation measure and compute that mass in the case where $d=2$. In particular, in $\mathcal{M}_d$, the number of such components 
is $\asymp d^{n_1+\cdots+n_{2d-2}}$, provided that $\min_j n_j$ is large enough.
\item in the moduli space $\mathcal{P}_d$ of polynomials of degree $d$, 
among hyperbolic components such that all (finite) critical points are in {the immediate basins of (not necessarily distinct) attracting cycles} of respective exact periods $n_1,\ldots,n_{d-1}$, the proportion of those components, counted with multiplicity, having at least two
critical points are in the same basin of attraction is exponentially small.
\item in $\cM_d$, we prove the equidistribution of the centers of the hyperbolic components  admitting $2d-2$ distinct attracting cycles of exact periods 
$n_1,\dots, n_{2d-2}$ towards the bifurcation measure $\mu_\bif$
with an exponential speed of convergence. 
\item we have equidistribution, up to extraction, of the parameters having $p$ distinct cycles of given multipliers towards the bifurcation current $T_\bif^p$ outside a pluripolar set of multipliers as the minimum of the periods of the cycles goes to $\infty$. 
\end{itemize} 
As a by-product, we also get the weak genericity of hyperbolic postcritically finiteness
in the moduli space of rational maps.
A key step of the proof is a {locally uniform} version of the quantitative approximation of the Lyapunov exponent of a rational map by the $\log^+$ of the modulus of the multipliers of periodic points. 
\end{abstract}

\maketitle

\tableofcontents

\section{Introduction}
For a holomorphic family $(f_{\lambda})_{\lambda\in \Lambda}$ 
of degree $d>1$ rational maps on the Riemann sphere $\P^1$
parametrized by a quasi-projective variety $\Lambda$, 
the \emph{bifurcation locus} of $(f_{\lambda})_{\lambda\in \Lambda}$
on $\Lambda$ is the $J$-unstability locus in the sense of Ma\~ne-Sad-Sullivan, i.e., the closure of 
the set of all parameters in $\Lambda$ at which
the Julia set $J_{\lambda}$ of $f_{\lambda}$ does not move
continuously. 
It is now classical that  this set is nowhere dense in $\Lambda$ and admits several distinct topological descriptions, such as the closure of the set of parameters for which $f_{\lambda}$ admits a non-persistent neutral cycle or the existence of an unstable critical dynamics (see e.g.~\cite{MSS,lyubich,McMullen}).
From now on, pick any integer $d>1$.

On the other hand, any (individual) rational map $f$ of degree $d$ on 
$\P^1$ admits a unique maximal entropy measure $\mu_f$, whose support coincides with the Julia set $\J_f$ of $f$, and the Lyapunov exponent of $f$ with respect to $\mu_f$ is defined by $L(f):=\int_{\P^1}\log|f'|\mu_f$ and satisfies $L(f)\geq\frac{1}{2}\log d>0$. For a family $(f_\lambda)_{\lambda\in\Lambda}$,
the induced Lyapunov function $L:\lambda\in\Lambda\longrightarrow L(f_\lambda)\in\R$ is \emph{p.s.h} and continuous on the parameter space $\Lambda$. We can define the \emph{bifurcation current} of $(f_{\lambda})_{\lambda\in \Lambda}$ on $\Lambda$ as the closed positive $(1,1)$-current
\[
T_\bif:=dd^cL.
\]
By DeMarco \cite{DeMarco2}, the support of $dd^cL$ coincides with the bifurcation locus of the family $(f_\lambda)_{\lambda\in\Lambda}$. 
For any integer $1\leq p\leq \dim\Lambda$, Bassanelli and Berteloot
also defined the $p$-\emph{bifurcation currents} $T_\bif^p$ as the $p$-th exterior product of $T_\bif$.
It is a positive closed current of bidegree $(p,p)$ so 
the \emph{bifurcation measure} $\mu_\bif:=(dd^cL)^{\dim\Lambda}$
is a positive measure
on $\Lambda$. 
{If $p>1$,} the current $T_\bif^p$ detects, in a certain sense, stronger bifurcations than $T_\bif=T_\bif^1$ \cite{BB1}. 
Indeed, its topological support admits several dynamical characterizations 
 {similar to} that of the bifurcation locus: for example, it is the closure of parameters admitting $p$ distinct neutral cycles or $p$ critical points preperiodic to repelling cycles (see~\cite{Dujardin2012,gauthier:Indiana}).

The group $\PSL_2(\C)$ of M\"obius transformations acts on 
the space $\rat_d$ of degree $d$ rational maps on $\P^1$,
which is itself a holomorphic family of rational maps,
by conjugacy. The \emph{moduli space $\mathcal{M}_d$} 
of degree $d$ rational maps on $\P^1$
is the {\itshape orbit space} of $\mathrm{PSL}_2(\C)$ in $\rat_d$, that is,
the quotient of $\rat_d$ resulting from this action of $\PSL_2(\C)$.
It is an irreducible affine variety of dimension $2d-2$,
and is singular if and only if $d\geq3$ (Silverman \cite{silverman-spacerat}).
The Lyapunov function $f\mapsto L(f)$ on $\rat_d$
{descends to}
a continuous and psh function $\mathcal{L}:\mathcal{M}_d\to\R$.
For any integer $1\le p\le 2d-2$, 
the $p$-bifurcation current on $\mathcal{M}_d$ is thus given by $T_\bif^p:=(dd^c\mathcal{L})^p$, and the \emph{bifurcation measure} on $\mathcal{M}_d$ is by
\[\mu_\bif:=T_\bif^{2d-2}=(dd^c\mathcal{L})^{2d-2}, \]
which is a finite positive measure on $\mathcal{M}_d$ of
strictly positive total mass (see \cite{BB1}). 

One of the features of the bifurcation currents is to allow measurable statements of the above density, or in general, accumulation properties. 
Let us be more precise. 
Let $\Lambda$ be a quasi-projective variety such that, either $\Lambda \subset \cM_d$, or parametrizing an algebraic family $(f_\lambda)_{\lambda\in\Lambda}$ of degree $d$ rational maps on $\P^1$.
For any $n\in\N^*$ 
and any $w\in\C\setminus\{1\}$, let $\Per_n(w)$ be the analytic hypersurface
\begin{gather*}
\Per_n(w):=\{\lambda\in\Lambda:f_\lambda\text{ has a cycle of multiplier }w\text{ and the exact period }n\}
\end{gather*}
in $\Lambda$
 and denote by $[\Per_n(w)]$ the current of integration over $\Per_n(w)$
on $\Lambda$. Since $\Lambda$ is quasi-projective, the hypersurfaces $\Per_n(w)$ are actually algebraic hypersurfaces of $\Lambda$ (see e.g.~\cite{BB3}).
By Bassanelli and Berteloot~\cite{BB3}, the sequence $(d^{-n}[\Per_n(w)])$ 
weighted by the Lebesgue measure on the disk of center $0$ and radius $|w|$ 
converges towards the bifurcation current $T_\bif$.
Similar dynamically significant equidistribution properties towards the bifurcation current have been recently established in various contexts, as general holomorphic families of rational maps~\cite{favredujardin,BB2,okuyama:distrib} or moduli spaces of polynomials~\cite{multipliers,DistribTbif}.

The proofs developed in op. cit. do not allow establishing equidistribution phenomena towards
the bifurcation measure  $\mu_\bif$. Indeed, {any of}
the  above convergences obtained 
is essentially $L^1_{\mathrm{loc}}$ convergence of the potentials 
{of currents}, which does not guarantee continuity of the intersection.

One of the main purposes of the article is to prove the equidistribution 
of  parameters having
$p$ non-repelling cycles towards the bifurcation current $T_\bif^p$ 
{as} the {minimum of the} periods of
those cycles go{es} to $\infty$, 
with an exponential speed of convergence. 
We will then deduce several important consequences, notably in 
counting hyperbolic components of disjoint types in $\cM_d$. Notice that such counting results are of combinatorial and algebraic nature and have a priori no relation to bifurcation measures. Furthermore, they are the first general results 
in that direction so far. 

\begin{notation} \normalfont
Let $\mu:\N^*\to\{-1,0,1\}$ be the M\"obius function.
For any $n\in\N^*$, set
\[
d_n:=\sum_{m|n}\mu\left(\frac{n}{m}\right)(d^m+1) \in\N^*,
\]
which is $\sim d^n$ as $n\to\infty$. For any $p\in\N^*$, any 
$\un=(n_1,\ldots,n_p)\in(\N^*)^p$, and any 
$\underline{\rho}=(\rho_1,\ldots,\rho_p)\in]0,1]^p$, 
we set $|\underline{n}|:=\sum_{j=1}^p n_j$, so that 
$d^{|\un|}= \prod_{j=1}^pd^{n_j}$, and set
\[
d_{|\un|}:=\prod_{j=1}^pd_j 
\]
in a similar way.
For any $i\in\{0,1,2\}$ and any $n\in\N^*$, we
also set $\sigma_i(n):=\sum_{m|n}m^i$, 
 so in particular $\sigma_0\le\sigma_1\le\sigma_2$ on $\N^*$ (beware that $\sigma_2(n)\leq C n^2 \log \log n$ for some constant $C$).

For any $\un=(n_1,\ldots,n_{2d-2})\in(\N^*)^{2d-2}$ 
 and any $\uw=(w_1,\ldots,w_{2d-2})\in \C^{2d-2}$,
let $\mathrm{Stab}(\un)$ (resp.\ $\mathrm{Stab}(\un, \uw)$)
be the set of all permutations  
of the indices $\{1,2,\ldots,2d-2\}$ 
that do not change the ordered $(2d-2)$-tuple 
$(n_1,\ldots,n_{2d-2})\in(\N^*)^{2d-2}$
(resp.\ $((n_1,w_1), \dots, (n_{2d-2},w_{2d-2}))\in(\N^*\times \C)^{2d-2}$), 
so in particular $\#\mathrm{Stab}(\un,\uw)\le\#\mathrm{Stab}(\un)\le(2d-2)!$.

For $r>0$, we set $\D_r=\{|z|<r\}$, so that $\partial\D_r=\mathbb{S}_r=\{|z|=r\}$.
\end{notation}

\subsection*{Statement of the main results}

Let $\Lambda$ be a quasi-projective variety such that, either $\Lambda \subset \cM_d$, or parametrizing an algebraic family $(f_\lambda)_{\lambda\in\Lambda}$ of degree $d$ rational maps on $\P^1$.
We refer to \cite{Demailly_SMF} for basics on positive closed currents 
and intersection theory on algebraic varieties.

For any integer $1\leq p\leq \min\{\dim\Lambda,2d-2\}$, 
any $\un=(n_1,\ldots,n_p)\in(\N^*)^p$, 
and any $\underline{\rho}=(\rho_1,\ldots,\rho_p)\in]0,1]^p$, the
following positive closed current

\begin{equation}\label{def:approxcurrent}
T_{\un}^p(\underline{\rho}):=\frac{1}{d_{|\un|}}\int_{[0,2\pi]^p}
\bigwedge_{j=1}^p[\Per_{n_j}(\rho_je^{i\theta_j})]
\frac{\mathrm{d}\theta_1\cdots\mathrm{d}\theta_p}{(2\pi)^p}
\end{equation}
on $\Lambda$ is well-defined,
and coincides with $\bigwedge_{j=1}^pT_{n_j}^1(\rho_j)$
by the Fubini theorem (see e.g.~\cite{BB3}).
We say a form $\Psi$ on $\Lambda$ to be DSH if $dd^c\Psi=T^+-T^-$ for some 
positive  closed currents $T^\pm$
 of finite masses on $\Lambda$. 
We refer to \S\ref{secDSH} for the precise definition of the semi-norm 
$\|\Psi\|_{\mathrm{DSH}}^*$.

One of our principal results is the following.

\begin{Theorem}\label{tm:vitessemoyennesalgebraic} 
Let $\Lambda$ be a quasi-projective variety such that, either $\Lambda \subset \cM_d$, or parametrizing an algebraic family $(f_\lambda)_{\lambda\in\Lambda}$ of degree $d$ rational maps on $\P^1$.
Then for any compact subset $K$ in $\Lambda$, there exists a constant $C(K)>0$ such that for any integer
$1\leq p\leq \min\{\dim\Lambda,2d-2\}$, any $\un=(n_1,\ldots,n_p)\in(\N^*)^p$, 
any $\underline{\rho}=(\rho_1,\ldots,\rho_p)\in]0,1]^p$, and
any continuous $\mathrm{DSH}$-form $\Psi$ of bidegree $(m-p,m-p)$ supported in $K$,
we have
\[
\left|\left\langle T_{\un}^p(\underline{\rho})-T_\bif^p,\Psi\right\rangle\right|\leq C(K)\cdot \max_{1\leq j\leq p}\left((1+|\log\rho_j|)\frac{\sigma_2(n_j)}{d^{n_j}}\right)\cdot\|\Psi\|^*_{\mathrm{DSH}}.
\]
\end{Theorem}
We first prove this theorem in the case where $p=1$.
To do that, we show in Section~\ref{sec:lyap} a locally uniform version of the second author's result \cite{okuyama:speed} 
on the quantitative approximation of the Lyapunov exponent 
of an (individual) $f\in\rat_d$
by the average of the $\log$ of the moduli of the multipliers of all non-attracting $n$-periodic points  of $f$ (Lemma~\ref{th:error}). This leads to an error term on the proximity between $f^n(c)$ and $c$ for each critical point $c$ of $f$ and an error term on how close to $0$ the multipliers of the periodic points of $f$ are. To control 
 those terms,
we use a parametric version of a lemma of Przytycki~\cite[Lemma 1]{Przytycki3} proved by the first and third authors in \cite{distribGV}. Intersection of currents and integrations by parts lead to the result for any $p$. Theorem~\ref{tm:vitessemoyennesalgebraic} is proved in Section~\ref{sec:main}.

The following application of Theorem~\ref{tm:vitessemoyennesalgebraic}
immediately implies various accumulation properties 
to the support of $T^p_\bif$ much finer than ever known results mentioned above.

\begin{Corollary}\label{tm:outsideluripolar}
Let $\Lambda$ be a quasi-projective variety such that, either $\Lambda \subset \cM_d$, or parametrizing an algebraic family $(f_\lambda)_{\lambda\in\Lambda}$ of degree $d$ rational maps on $\P^1$.
Pick an integer $1\leq p\leq \min\{\dim\Lambda,2d-2\}$. Then 
for any sequence $(\un_k)_{k\in\N^*}$ of $p$-tuples 
$\un_k=(n_{1,k},\ldots,n_{p,k})$ in $(\N^*)^p$ 
such that $\sum_k\max_j(n_{j,k}^{-1})<\infty$, there exists a pluripolar subset $\mathcal{E}$ in $\C^{p}$ such that for any $\uw=(w_1,\ldots,w_{p})\in\C^{p}\setminus\mathcal{E}$, 
$\bigcap_{i=1}^{p}\Per_{n_{i,k}}(w_i)$ is of pure codimension $p$ in $\Lambda$ for any $k\in\N^*$ and 
\[
T^p_\bif=\lim_{k\rightarrow\infty} \frac{1}{d_{|\un_k|}}\bigwedge_{i=1}^{p}[\Per_{n_{i,k}}(w_i)]
\]
in the weak sense of currents on $\Lambda$.
\end{Corollary}

The techniques used in the proof of Corollary \ref{tm:outsideluripolar}
also give that the current equidistributed on the set of parameters having  
$p$ cycles of respective periods $n_{1,k},\ldots,n_{p,k}$ and multipliers $w_1,\ldots,w_{p}$ distributed by a PB measure on $(\P^1)^p$
converges towards the bifurcation current $T_\bif^p$
when $k\to \infty$, with the best possible order estimate
$O(\max_j(n_j^{-1}))$ as $\min_j(n_j)\to\infty$.
(see Theorem~\ref{tm:vitessemoyennesPB} below). 

\begin{remark} \normalfont \normalfont
Finally, let us observe that, as in~\cite{BB3}, Theorem~\ref{tm:vitessemoyennesalgebraic} gives another proof that Shishikura's upper bound $2d-2$ of the number of distinct cycles of Fatou components of a given rational map of degree $d$ is sharp (see~\cite{Shishikura3}). In fact, provided $\min_jn_j$ is large enough, we can construct a rational map having $2d-2$ distinct attracting periodic points of respective period $n_j$ (we no longer need to take a subsequence and have no arithmetic restrictions on the periods).
\end{remark}

Now let us focus on the moduli spaces of rational maps and hyperbolic components. Recall that the \emph{hyperbolic locus} in $\cM_d$
is the set of all conjugacy classes of hyperbolic maps (i.e. uniformly expanding maps on their Julia sets). It
is an open subset of $\cM_d$ and a connected component of this hyperbolic locus is called a hyperbolic component
in $\cM_d$.

\begin{definition}
 A rational map $f\in\rat_d$ is said to be 
\emph{hyperbolic of type} $\un=(n_1,\ldots,n_{2d-2})\in(\N^*)^{2d-2}$ 
if $f$ has $2d-2$ distinct
attracting cycles of respective exact periods $n_1,\ldots,n_{2d-2}$.

A hyperbolic component $\Omega$ in $\mathcal{M}_d$ 
 is said to be 
of {\itshape type} $\un\in(\N^*)^{2d-2}$
if, for any $[f]\in\Omega$, $f$ 
is hyperbolic of type $\un$. An hyperbolic component in $\cM_d$ is of \emph{disjoint type} if it is of type $\un$ for some $\un \in (\N^*)^{2d-2}$. 
\end{definition}

\begin{definition}
For any $\un=(n_1,\ldots,n_{2d-2})\in(\N^*)^{2d-2}$, let $ N(\un)$ denote the number of hyperbolic components of type $\un$ in $\cM_d$.
\end{definition}

A  striking application of Theorem~\ref{tm:vitessemoyennesalgebraic} 
is the following asymptotic on the global counting of hyperbolic
components. 

\begin{Theorem}\label{tm:counting}
As $\min_jn_j\to+\infty$, 
\[
\#\mathrm{Stab}(\un)\cdot
\frac{N(\un)}{d_{|\un|}}
= \int_{\mathcal{M}_d}\mu_\bif+O \left(\max_j
\left(\frac{\sigma_2(n_j)}{d^{n_j}}\right) \right).
\]
\end{Theorem}
In particular, $N(\un)>0$ if $\min_jn_j$ is large enough. This result gives a combinatorial interpretation of the mass of the bifurcation measure.
In the case $d=2$, as a consequence of Theorem~\ref{tm:counting} together with 
the precise estimates of $N(n_1,n_2)$ 
by Kiwi and Rees~\cite{kiwirees}, we can determine the (total) mass
of the bifurcation measure on $\mathcal{M}_2$.

\begin{Corollary}\label{tm:massM2}
 Let $\phi$ be the Euler totient function on $\N^*$. Then
\[
\int_{\mathcal{M}_2}\mu_\bif=\frac{1}{3}-\frac{1}{8}\sum_{n\geq1}\frac{\phi(n)}{(2^n-1)^2}.
\]
\end{Corollary}

In the proof of Theorem~\ref{tm:counting},
it is crucial that the estimate 
{in} Theorem~\ref{tm:vitessemoyennesalgebraic} involves \emph{only} the DSH-semi-norm $\|\cdot\|_{\mathrm{DSH}}^*$ of the observable. 
Notice {also} that 
the mass of a limit of positive measures is {not greater} than
the limit of the masses, so it could be possible that a proportion of components
is lost passing to the limit as they would accumulate at the boundary 
of the moduli space. Theorem~\ref{tm:counting} states that it is not the case.
The proof of Theorem~\ref{tm:counting} also relies crucially on the fact that 
the multipliers {of attracting {\em cycles}}
parametrize the hyperbolic components  
of disjoint type of $\mathcal{M}_d$.
Though such a parametrization is essentially 
classical, there seems to be 
no available proofs 
in the literature so we  include a proof of it
in Section~\ref{parametrizing_hyperbolic_components}. 
{The} proof relies on the
transversality of periodic {\em critical} orbit relations{, which} we show 
in Section~\ref{Transversality of periodic critical orbit relations}, 
following the argument of Epstein \cite{buffepstein,epstein2} 
(again such a transversality seems classical but we could not find 
a precise statement in the literature). 

There is an analogy between Theorem \ref{tm:counting}
(or Theorem \ref{tm:vitessecenters} below)
and the key works by Briend and Duval 
\cite{briendduval, briendduval2}: before their works, for $k>1$, 
it had been not known whether a holomorphic endomorphism $f$ of $\P^k$ 
of degree $d >2$ had even a single repelling cycle, 
but by establishing the equidistribution of such cycles towards
the equilibrium measure of $f$, they showed that 
there are indeed $\sim d^{kn}$, so, plenty of such cycles of exact period $n$ as $n\to\infty$.

An easy observation 
using Bézout's theorem (cf.\ Subsection \ref{sec:dynatomic})
already gives
\[
\sup_{\un\in(\N^*)^{2d-2}}\frac{N(\un)}{d^{|\un|}}<\infty.
\]
As a consequence of Theorem~\ref{tm:counting}, we also
establish the weak genericity of hyperbolic postcritically finite maps in $\mathcal{M}_d$ (see Theorem~\ref{tm:genericMd} below), which is stronger than Zariski density of them in $\cM_d$. 

~

It {might be worth} stressing that in \cite{distribGV},
we considered {the}
currents of bifurcation $T_c$
of marked critical points {$c$} 
so we were looking at unstable critical dynamics.
Although it seems to be similar, here we study directly the bifurcation current 
$T_\bif$
so the unstability of cycles (see the introduction of \cite{BB3}).
It would not be possible to have {any} ``usable'', for our purpose, locally uniform  approximation of the potential of $T_c$. 
Hyperbolic components where {some} periods {coincide}
is also something that cannot be managed using 
 marked critical points.

We also establish a quantitative equidistribution 
of parameters in hyperbolic components in $\mathcal{M}_d$
of disjoint type, having given multipliers.
For any $\un=(n_1, \dots, n_{2d-2})\in (\N^*)^{2d-2}$ 
{and any $\uw=(w_1,\ldots,w_{2d-2})\in\D^{2d-2}$},
let $\mathrm{C}_{\un,\uw}$ denote the (finite) set of all
conjugacy classes $[f]\in\mathcal{M}_d$ of hyperbolic rational maps 
$f\in\rat_d$ of type $\un$ whose attracting cycle of exact period $n_j$
has the multiplier $w_j$ for any $1\le j\le 2d-2$, and set
\[ 
\mu_{\un,\uw}:= \frac{\# \mathrm{Stab}(\un,\uw)}{d_{|\un|}}\sum_{{[f]}\in \mathrm{C}_{\un,\uw}} \delta_{[f]}.
\]
For simplicity,  we denote $\mathrm{C}_{\un,(0,\ldots,0)}$
and $\mu_{\un,(0,\ldots,0)}$ by $\mathrm{C}_\un$ and $\mu_\un$, respectively,
so that any element in $\mathrm{C}_\un$ is the center of a hyperbolic component
in $\cM_d$ of type $\un$.

The following in particular implies the weak convergence $\mu_{\un,\uw}\to \mu_\bif$
on $\mathcal{M}_d$, which is even new and it was one of our motivations
to give a proof of this convergence.

\begin{Theorem}\label{tm:vitessecenters}
{For any} compact subset $K$ in $\mathcal{M}_d$, 
there exists $C_K>0$ such that 
\begin{enumerate}
 \item for any test function $\Psi \in \mathcal{C}^2(\mathcal{M}_d)$ with support in $K$ and any $\un \in (\N^*)^{2d-2}$,
\[ 
\left| \langle \mu_\un-\mu_\bif, \Psi \rangle  \right|
\leq  C_K\cdot
\max_{1\leq j \leq 2d-2}\left( \frac{\sigma_2(n_j)}{d^{n_j}}
\right)\cdot\|\Psi\|_{\mathcal{C}^2} ,
\]
 \item  for any test function $\Psi \in \mathcal{C}^1(\mathcal{M}_d)$ with support in $K$, any $\un \in (\N^*)^{2d-2}$, and any $\uw=(w_1,\ldots,w_{2d-2})\in\D^{2d-2}$,
\[
\left| \langle \mu_{\un,\uw}-\mu_\bif, \Psi \rangle  \right|
\leq  C_K\cdot
\max_{1\leq j \leq 2d-2}\left(\frac{-1}{d^{n_j}\log|w_j|},
\frac{\sigma_2(n_j)}{d^{n_j}}\right)^{1/2}\cdot \|\Psi\|_{\mathcal{C}^1}.
\]
\end{enumerate} 
\end{Theorem}
Observe that an interpolation between Banach spaces gives a speed of convergence for all $\mathcal{C}^\alpha$-observables with $0<\alpha \leq 2$ in the case of centers and 
$0<\alpha \leq 1$ in general.

{In \cite{FRL}, 
a result similar to Theorem \ref{tm:vitessecenters} has been}
established for the centers of hyperbolic components
{of the interior of the {\em Multibrot} set in $\C$ of}
the unicritical family $(z^d+\lambda)_{\lambda\in\C}$ 
with a $\mathcal{C}^1$-estimate, {using arithmetic}.
Then in \cite{favregauthier}, 
this was generalized to the moduli space of polynomials of degree $d$ 
{under} the hypothesis that all $n_j$ are distinct,
without a speed of convergence, also using arithmetic geometry. 
Finally, in \cite{distribGV}, 
still in the moduli space of polynomials of degree $d$
{and under the hypothesis that all $n_j$ are distinct}, 
a speed of convergence similar to Theorem~\ref{tm:vitessecenters} 
for $\mathcal{C}^2$-observables was obtained using pluripotential theory 
(see also~\cite{distribGV2} for a similar result using combinatorial continuity 
of critical portraits). In the proofs of all the aforementioned results 
for the moduli space of degree $d$ polynomials, 
the {\em compactness} of the support of the bifurcation measure 
was crucial. This is not the case for $\mathcal{M}_d$ and 
one can find relatively compact hyperbolic components arbitrarily close 
to the boundary of $\mathcal{M}_d$ where continuity estimates fail 
(see e.g.~\cite{Mod2}).
Furthermore, it is not required in Theorem \ref{tm:vitessecenters} that all $n_j$ are distinct.
Notice also that our proof applies verbatim in the moduli space of degree $d$ polynomials. The proof of Theorem~\ref{tm:vitessecenters} is the content of Section~\ref{sec:equicenter}. As an application, we give a lower bound on the number of hyperbolic components of disjoint type in $\mathcal{M}_d$ whose center is in a given open subset intersecting the support of the bifurcation measure (see Theorem~\ref{tm:countinglower}).

Section~\ref{sec:poly} is devoted to the study of the moduli space $\mathcal{P}_d^\cm$ of critically marked degree $d$ polynomials. We give various results similar to those previously proved  in the moduli space $\mathcal{M}_d$ of degree $d$ rational maps. The reader is referred to Section~\ref{sec:poly} for precise statements. Let us just mention two points: first, our techniques allow us to give an asymptotic of the number of hyperbolic components with $(d-1)$ distinct attracting cycles in $\C$ of prescribed periods even when those periods are not distinct, with an explicit exponential error term (see Theorem~\ref{tm:countgoodpoly}). Second, the various rates of convergence which appear do not depend on the support of the test functions which are considered, by compactness of the support of the bifurcation measure in that context.

To finish, let us mention that, as an application of our approximation formula of the Lyapunov exponent, we give a proof of the estimate of the degeneration of the Lyapunov exponent of $f$ as $f\to\partial\rat_d$ along an analytic disk in the spirit of ~\cite{Favre-degeneration} (see~Theorem~\ref{tm:degenerate}).

                                                                                                                     \section{Preliminaries}

\subsection{Currents and DSH functions}\label{secDSH}
We refer to~\cite[Appendix A]{dinhsibony2} for more details on currents and DSH functions. Pick any quasi-projective variety $\Lambda$. Let $\beta$ be the restriction of the ambient Fubini-Study form to $\Lambda$. For any positive closed current $T$ of bidimension $(k,k)$ defined on $\Lambda$ and any Borel set $A \subset \Lambda$, we denote by $\|T\|_A$ the number
\[ 
\|T\|_A :=\int_A T \wedge \beta^k.
\]
This is the \emph{mass} of the current $T$ in $A$. We simply write $\|T\|$ for $\|T\|_\Lambda$.

Let $\Psi$ be  an $(\ell,\ell)$-form in $\Lambda$. We say that $\Psi$ is DSH if we can write $dd^c\Psi=T^+-T^-$  where $T^\pm$ are positive closed currents of finite mass in $\Lambda$. We also set
\[
\|\Psi\|^*_{\mathrm{DSH}}:=\inf\left(\|T^+\|+\|T^-\|\right),
\]
where the infimum is taken over all closed positive currents $T^\pm$ such that $dd^c\Psi=T^+-T^-$ (note that $\|T^+\|=\|T^-\|$ since they are cohomologuous).

This is not exactly the usual DSH norm but just a semi-norm. 
Nevertheless, one has 
$\|\Psi\|_{\mathrm{DSH}}^* \leq \|\Psi\|_{\mathrm{DSH}}$, where $\|\Psi\|_{\mathrm{DSH}}:=\|\Psi\|_{\mathrm{DSH}}^*+\|\Psi\|_{L^1}$. The interest of those DSH-norms lies in the fact that they behave nicely under change of coordinates. 
Furthermore, when $\Psi$ is $\mathcal{C}^2$ with support in a compact set $K$, there is a constant $C>0$ depending only on $K$ such that $\|\Psi\|_{\mathrm{DSH}}\leq C\|\Psi\|_{\mathcal{C}^2}$.

\subsection{Resultant and the space $\rat_d$}
We refer to~\cite{BB1} and~\cite{silverman-spacerat} 
for the content of this paragraph.

\begin{notation}
Let $\pi:\C^2\setminus\{0\}\to\P^1$ be the canonical projection, 
$\|\cdot\|$ be the Hermitian norm on $\C^2$, and $\wedge$ is the wedge
product on $\C^2$.
\end{notation}

A pair $F=(F_1,F_2)\in\C[x,y]_d\times\C[x,y]_d\simeq\C^{2d+2}$ of homogeneous degree $d$ 
polynomials can be identified with a degree $d$ homogeneous polynomial endomorphism of $\C^2$. 
The homogeneous resultant $\mathrm{Res}=\mathrm{Res}_d$ is the unique homogeneous degree $2d$ 
polynomial over $\C$ in $2d+2$ variables such that
$\mathrm{Res}(F)=0$ if and only if $F$ is degenerate, 
i.e. $F^{-1}(\{0\})\neq\{0\}$, and that
$\mathrm{Res}((x^d,y^d))=1$.
We thus identify the space of all degree $d$ homogeneous polynomial 
endomorphisms of $\C^2$ with $\C^{2d+2}\setminus\{\mathrm{Res}=0\}$. 

A rational map $f$ on $\P^1$ of degree $d$ admits 
a (non-degenerate homogeneous polynomial) lift, 
i.e. there exists a degree $d$ homogeneous polynomial 
endomorphism $F:\C^2\to\C^2$ such that $\mathrm{Res}(F)\neq 0$ and that
$\pi\circ F=f\circ\pi$ on $\C^2\setminus\{0\}$. Moreover,
any two homogeneous polynomial endomorphisms $F,G$ of $\C^2$ are lifts of
the same $f$ if and only if there exists $\alpha\in\C^*$ such that $F=\alpha\cdot G$.
Let us denote by $\rat_d$ the set of all degree $d$ rational maps on $\P^1$. 
Since $\mathrm{Res}$ is homogeneous, we can also identify
$\rat_d$ with $\P^{2d+1}\setminus\{\mathrm{Res}=0\}$.
In particular, it is a quasi-projective variety of dimension $2d+1$.

\subsection{The dynamical Green function of a rational map on $\P^1$}\label{sec:defGreen}

In the whole text, we denote by $\omega_{\mathrm{FS}}$ the Fubini-Study 
form on $\P^1$ normalized so that $\|\omega_{\mathrm{FS}}\|=1$ and by $[\cdot,\cdot]$ the chordal metric on $\P^1$, normalized so that $\mathrm{diam}(\P^1)=1$. Recall that, for all $w$, $dd^c_z\log[z,w]=\delta_w-\omega_{\mathrm{FS}}$.

For any $\omega_{\mathrm{FS}}$-psh function $g$ on $\P^1$, i.e. such that
\begin{gather*}
\omega_{\mathrm{FS}}+dd^c g=:\nu_g
\end{gather*}
is probability measures on $\P^1$, we define the $g$-\emph{kernel function} 
$\Phi_g$ by setting
\begin{gather}
\Phi_g(z,w):=\log[z,w]-g(z)-g(w) 
\label{eq:kernel}
\end{gather}
on $\P^1\times\P^1$.
For a probability measure $\nu'$ on $\P^1$, 
set $U_{g,\nu'}:= \int_{\P^1}  \Phi_g(\cdot,w) d\nu'(w)$ on $\P^1$. Then $dd_z^c U_{g,\nu'}=\nu'-\nu_g$, so in the particular case where $\nu'=\nu_g$, 
we deduce that 
\begin{gather*}
U_{g,\nu_g}\equiv I_g:=\int_{\P^1\times\P^1}\Phi_g(\nu_g\times\nu_g)\quad\text{on }\P^1.
\end{gather*}

Pick now $f\in\rat_d$. For all (non-degenerate homogeneous polynomial)
lift $F:\C^2\to\C^2$ of $f$, there exists a H\"older continuous 
$\omega_{\mathrm{FS}}$-psh function $g_F:\P^1\to\R$ such that
\begin{gather*}
\lim_{n\to\infty}\frac{\log\|F^n\|}{d^n}-\log\|\cdot\|=g_F\circ\pi
\end{gather*}
uniformly on $\C^2\setminus\{0\}$, which is
called the \emph{dynamical Green function} of $F$ on $\P^1$. 
Since $F$ is unique up to multiplication by $\alpha\in\C^*$ and
$g_{\alpha\cdot F}=g_F+(\log|\alpha|)/(d-1)$ for any $\alpha\in\C^*$, the positive measure
\[
\omega_{\mathrm{FS}}+dd^cg_F=:\mu_f
\]
is independent of the choice of $F$, and is in fact
the unique maximal entropy measure of $f$ on $\P^1$.
 For later use, we point out the equality
\begin{align*}
I_{g_F}=-\frac{1}{d(d-1)}\log|\mathrm{Res}(F)|,
\end{align*}
which is (a reformulation of) DeMarco's formula on the homogeneous capacity
of the filled-in Julia set of $F$ in $\C^2$.

\begin{definition}
The \emph{dynamical Green function} $g_f$ of $f$ on $\P^1$ is 
the unique $\omega_{\mathrm{FS}}$-psh function on $\P^1$ such that
$\mu_{g_f}=\mu_f$ on $\P^1$ and that $I_{g_f}=0$.
\end{definition}

\begin{remark} \normalfont \normalfont
In particular, $U_{g_f,\mu_f}\equiv I_{g_f}=0$ on $\P^1$. 
Moreover, $g_F=g_f$ for some lift $F$ of $f$,
which is unique up to multiplication by a complex number of modulus one.
\end{remark}

\subsection{The dynatomic and multiplier polynomials}\label{sec:dynatomic}

We refer to \cite[\S 4.1]{Silverman} and to \cite{BB3,BB2,Milnor3} (see also~\cite[\S 6]{favregauthier}) for the details on the dynatomic  and multiplier polynomials and the related topics. 

Pick any $f\in\rat_d$. For every $n\in\N^*$, let
\begin{itemize}
\item $\Fix(f^n)$ be the set of all fixed points of $f^n$ in $\P^1$, and
\item $\Fix^*(f^n)$ the set of all periodic points of $f$ in $\P^1$ having exact period $n$.
\end{itemize} 
The $n$-th \emph{dynatomic polynomial} of a lift $F$ of $f$ 
is a {\itshape homogeneous} polynomial
\[
\Phi_n^*(F,(z_0,z_1)):=\prod_{k:k|n}(F^k(z_0,z_1)\wedge(z_0,z_1))^{\mu(n/k)}
\]
in $z_0,z_1$ of degree $d_n$;
there is a (finite) sequence $(P_j^{(n)})_{j\in\{1,\ldots,d_n\}}$
in $\C^2\setminus\{0\}$ such that we have a factorization
$\Phi_n^*(F,(z_0,z_1))=\prod_{j=1}^{d_n}((z_0,z_1)\wedge P_j^{(n)})$,
and setting $z_j^{(n)}:=\pi(P_j^{(n)})\in\P^1$ for each $j\in\{1,\ldots,d_n\}$,
the sequence $(z_j^{(n)})_{j=1}^{d_n}$ is independent of
the choice of $(P_j^{(n)})_{j\in\{1,\ldots,d_n\}}$ and that of $F$, up to permutation. 

We recall that the set $\{z_j^{(n)}:j\in\{1,\ldots,d_n\}\}$ 
is the disjoint union of $\Fix^*(f^n)$ and the set of all periodic points $z$ of $f$ having exact period $m<n$ and
dividing $n$ and whose multiplier $(f^m)'(z)$ is a $n/m$-th primitive
root of unity, so that $(f^n)'(z)=1$ 
for every $z\in\bigl\{z_j^{(n)}:j\in\{1,\ldots,d_n\}\bigr\}\setminus\Fix^*(f^n)$ and that for every $z\in\Fix^*(f^n)$, 
we have $\#\{j\in\{1,\ldots,d_n\}:z_j^{(n)}=z\}=1$ if $(f^n)'(z)\neq 1$. 
For every $n\in\N^*$, 
the $n$-th {\em multiplier polynomial} of $f$ is the polynomial
\begin{gather}
p_n(f,w):=\left(\prod_{j=1}^{d_n}((f^n)'(z_j^{(n)})-w)\right)^{1/n}\label{eq:multiplier}
\end{gather}
in $w$ of degree $d_n/n$, which is unique up to multiplication in $n$-th roots of
unity.

Let $\Lambda$ be a quasi-projective variety parametrizing 
an algebraic family $(f_\lambda)_{\lambda\in\Lambda}$ 
of degree $d$ rational maps on $\P^1$. 
Then for any $n\in\N^*$, the {\em $n$-th multiplier polynomial} 
$p_n:\Lambda\times\C\longrightarrow\C$ {\em of} $(f_\lambda)_{\lambda\in\Lambda}$
defined by
\begin{gather*}
p_n(\lambda,w):=p_n(f_\lambda,w)
\end{gather*}
is holomorphic, and since $\Lambda$ is a quasi-projective variety, 
this $p_n:\Lambda\times\C\to\C$ is actually a morphism with 
$\deg_w(p_n(\lambda,w))=d_n/n$ for all $\lambda\in\Lambda$ and with
$\deg_\lambda(p_n(\lambda,w))\leq C d_n$ for all $w\in\C$, 
where $C>0$ depends only on the family $(f_\lambda)_{\lambda\in\Lambda}$, see e.g.~\cite[\S 2.2]{BB2}. 
For any $n\in\N^*$ and any $w\in\C$, we set
\[
\Per_n(w):=\{\lambda\in\Lambda\, ; \ p_n(\lambda,w)=0\}
\]
and denote by $[\Per_n(w)]$ the current of integration defined by the zeros
of $p_n(\cdot,w)$ on $\Lambda$. Remark that for all $w\in\C$ and all $n\in\mathbb{N}^*$, we have
\begin{equation}
\frac{1}{d_n}\|[\Per_n(w)]\|\leq C\label{eq:boundedmass}
\end{equation}
by B\'ezout Theorem.

Beware also that, since the existence of a cycle of given period and multiplier is invariant under M\"obius conjugacy, the $n$-th multiplier polynomial 
$p_n:\rat_d\times\C\to\C$ of $\rat_d$
also descends to a regular function $p_n:\cM_d\times\C\to\C$, enjoying the same properties.

\subsection{A parametric version of Przytycki lemma}

For a $\mathcal{C}^1$ map $f:\P^1\rightarrow\P^1$, 
the \emph{chordal derivative} $f^\#$ of $f$ is the non-negative real valued 
continuous function
\[
f^\#(z):= \lim_{y\to z} \frac{[f(z),f(y)]}{[z,y]}
\]
on $\P^1$. For any rational map $f\in\rat_d$, we set 
\[
M(f):=\sup_{\P^1}(f^\#)^2\in]1,+\infty[.
\]
We shall use the following,
which is a direct consequence of \cite[Lemma 3.1]{distribGV} and of the fact that the spherical and the chordal distance are equivalent on $\P^1$. 

\begin{lemma}\label{lm:Przytycki}
There exists a universal constant $0<\kappa<1$ such that for any holomorphic family $(f_\lambda)_{\lambda\in\Lambda}$ of degree $d$ rational maps with a marked critical point $c:\Lambda\to\P^1$ which does not lie persistently in a parabolic basin of $f_\lambda$ and is not persistently periodic, the following holds: for any $n\geq1$ and any $\lambda\in\Lambda$, 
if $f_\lambda^n(c(\lambda))\neq c(\lambda)$, then
\begin{itemize}
\item either $[f^n_\lambda(c(\lambda)),c(\lambda)]\geq \kappa\cdot M(f_\lambda)^{-n}$,
\item or $c(\lambda)$ lies in the immediate basin of an attracting periodic point $z(\lambda)$ of $f_\lambda$ of period dividing $n$,
$[c(\lambda),\J_\lambda]\geq \kappa M(f_\lambda)^{-n}$, and
$[f^n_\lambda(c(\lambda)),c(\lambda)]\geq 2\cdot [z(\lambda),c(\lambda)]$.
\end{itemize}
\end{lemma}

\subsection{A length-area estimate}
The modulus of an annulus $A$ conformally equivalent to 
$A'=\{z\in\C \, ; \ r<|z|<R\}$ with $0< r<R< +\infty$ is defined by 
\[
\textup{mod}(A)=\textup{mod}(A')=\frac{1}{2\pi}\log\left(\frac{R}{r}\right).
\]
We shall use the following classical estimate (\cite[Appendix]{briendduval}).

\begin{lemma}[Briend-Duval]\label{lm:BriendDuval}
There exists a universal constant $\tau>0$ such that for any quasi-projective variety $\Lambda$, any K\"ahler metric $\omega$ on $\Lambda$ and any pair of relatively compact holomorphic disks $D_1\Subset D_2$ in $\Lambda$, we have
\begin{gather*}
(\textup{diam}_\omega(D_1))^2\leq \tau\cdot\frac{\textup{Area}_\omega(D_2)}{\min(1,\textup{mod}(A))},
\end{gather*}	
where $A$ is the annulus $D_2\setminus \overline{D_1}$.
\end{lemma}

\section{Quantitative approximation of the Lyapunov exponent}\label{sec:lyap}

Our precise result here can be stated as follows.
This result relies on the combination of the arguments used in~\cite{okuyama:speed} as developed in Lemma~\ref{reductiontoattr} below and of the lemma ``\`a la Przytycki'' proved in~\cite{distribGV}.
The {\em locally uniform} speed of convergence obtained here is not as fast as the pointwise one obtained in~\cite{okuyama:speed}. This is due to our need to control the dependence of the constants on $f\in\rat_d$ in the right-hand side. Here we obtain a continuous dependence.

\begin{theorem}\label{tm:approx1}
There exists $A\geq1$ depending only on $d$ 
such that for any $r\in]0,1]$, any $f\in\rat_d$, and any $n\in\N^*$, we have
\begin{eqnarray*}
\left|\frac{1}{d_n}\int_0^{2\pi}\log\left|p_n(f,re^{i\theta})\right|\frac{\mathrm{d}\theta}{2\pi}-L(f)\right| \leq A \left(C([f])+ |\log r|\right)\frac{\sigma_2(n)}{d^n},
\end{eqnarray*}
where $C([f])=\inf\left\{\log\left(\sup_{\P^1}f_1^\#\right)+\sup_{\P^1}|g_{f_1}|\right\}$, where the infimum is taken over all $f_1\in[f]$ and $g_{f_1}$ is the dynamical Green function of $f_1$ normalized as in \S\ref{sec:defGreen}.
\end{theorem}

Of course, as the left-hand side of the inequality is invariant under M\"obius conjugacy, it is sufficient to prove that for any $0<r\leq 1$, any $n\in\mathbb{N}^*$ and any $f\in\rat_d$, we have
\[\left|\frac{1}{d_n}\int_0^{2\pi}\log\left|p_n(f,re^{i\theta})\right|\frac{\mathrm{d}\theta}{2\pi}-L(f)\right| \leq A \left(\log\left(\sup_{\P^1}f^\#\right)+\sup_{\P^1}|g_{f}|+ |\log r|\right)\frac{\sigma_2(n)}{d^n}\]
for some constant $A$ which depends only on $d$.

So we pick $f\in\rat_d$.
In the following, the sums over subsets in $\mathrm{Crit}(f)$,
$\Fix(f^n)$, or $\Fix^*(f^n)$ 
take into account the multiplicities of their elements.
For any $n\in\N^*$, the cardinality of $\Fix(f^n)$ and that of $\Fix^*(f^n)$
are $d^n+1$ and $d_n$, respectively, taking into account of the multiplicity
of each element of them as a fixed point of $f^n$.

\subsection{Initial dynamical estimates}

Recall that, by \cite[Lemma 2.4]{okuyama:speed} , we have
\begin{gather}
\log(f^\#)
=L(f)+\sum_{c\in \mathrm{Crit}(f)}\Phi_{g_f}(\cdot,c)+2(g_f\circ f-g_f)\ \text{on }\ \P^1~.\label{eq:factor}
\end{gather}
This formula plays a key role in the proofs of Lemma \ref{th:lower} and \ref{th:error}.

\begin{lemma}\label{th:lower}
Assume that $f$ has no super-attracting cycles. Then for any $n\in\N^*$ 
and any $z\in\Fix(f^n)$, we have
\begin{eqnarray*}
\frac{1}{n}\left|\sum_{c\in \mathrm{Crit}(f)}\sum_{j=0}^{n-1}\log[f^j(z),c]
-\log|(f^n)'(z)|\right|\leq B_1(f),
\end{eqnarray*}
where $B_1(f):=L(f)+2(2d-2)\sup_{\P^1}|g_f|$.
\end{lemma}

\begin{proof}
By \eqref{eq:factor} applied to $f^n$, we have
\begin{multline*}
\log((f^n)^\#)
=L(f^n)+\sum_{\tilde{c}\in \mathrm{Crit}(f^n)}\Phi_{g_{f^n}}(\cdot,\tilde{c})
+2(g_{f^n}\circ f^n-g_{f^n})\\
=n\cdot L(f)
+\sum_{c\in \mathrm{Crit}(f)}\Biggl(\sum_{j=0}^{n-1}\int_{\P^1}\Phi_{g_f}(\cdot,w)\,((f^j)^*\delta_c)(w)\Biggr)
+2(g_f\circ f^n-g_f)
\end{multline*}
on $\P^1$. By~\cite[Lemma 3.4]{OkuyamaStawiska}, for every $a\in\P^1$,
\begin{gather*}
\int_{\P^1}\Phi_{g_f}(\cdot,w)\,(f^*\delta_a)(w)=\Phi_{g_f}(f(\cdot),a)
\end{gather*}
on $\P^1$. For every $z\in\Fix(f^n)$, in particular, we  have
\begin{gather*}
\frac{1}{n}\log|(f^n)'(z)|
=L(f)+\sum_{c\in \mathrm{Crit}(f)}\frac{1}{n}\sum_{j=0}^{n-1}\Phi_{g_f}(f^j(z),c),
\end{gather*}
which with the definition \eqref{eq:kernel} of the $g_f$-kernel function 
$\Phi_{g_f}$ completes the proof.
\end{proof}

For any $n\in\N^*$ and any $0<r\leq 1$, we set
\begin{gather*}
L_n^r(f) := \frac{1}{d_n}\int_0^{2\pi}\log|p_n(f,re^{i\theta})|\frac{\mathrm{d}\theta}{2\pi}. 
\end{gather*}
If $f$ has no super-attracting cycles, for any $m,n\in\N^*$ with $m|n$ and any $0<r\leq 1$, we set
\begin{align*}
u_{m,n}(f,r):=
&\frac{1}{d^m+1}\Biggl(\sum_{c\in \mathrm{Crit}(f)}\log[f^m(c),c]
-\sum_{z\in\Fix(f^m) \atop |(f^n)'(z)|<r}\frac{1}{m}\log|(f^m)'(z)|\Biggr)\\
=&u_{m,m}(f,r^{m/n}).
\end{align*}

\begin{lemma}\label{th:error}
If $f$ has no super-attracting cycles, for any $r\in]0,1]$ and any $n\in\N^*$,
\begin{align*}
\epsilon_n(f,r)
:=\frac{1}{n(d^n+1)}\sum_{z\in\Fix(f^n) \atop |(f^n)'(z)|\geq r}\log|(f^n)'(z)|- L(f)
=u_{n,n}(f,r)+\epsilon_n'(f),
\end{align*}
where $(d^n+1)|\epsilon_n'(f)|\leq 2(2d-2)\sup_{\P^1}|g_f|$.
\end{lemma}

\begin{proof}
Pick $r\in]0,1]$ and $n\geq1$, and set $\mu_n:=\sum_{z\in\Fix(f^n)}\delta_z $. Since $(f^n)^\#(z)=|(f^n)'(z)|$ for any $z\in\Fix(f^n)$, integrating the equation \eqref{eq:factor} against $\mu_n$ gives
\begin{eqnarray*}
\frac{1}{n}\int_{\P^1}\log|(f^n)'|\,\mu_n-(d^n+1) L(f) & = & \int_{\P^1}\log(f^\#)\,\mu_n-(d^n+1) L(f)\\
& = & \sum_{c\in \mathrm{Crit}(f)}\int_{\P^1}\Phi_{g_f}(c,\cdot)\,\mu_n~.
\end{eqnarray*}
This may be rewritten
\begin{eqnarray*}
(d^n+1)\epsilon_n(f,r)&=&
\frac{1}{n}\int_{\{|(f^n)'|\ge r\}}\log|(f^n)'|\,\mu_n-(d^n+1) L(f) \\
& = & \sum_{c\in \mathrm{Crit}(f)}\int_{\P^1}\Phi_{g_f}(c,\cdot)\,\mu_n
-\int_{\{|(f^n)'|<r\}}\frac{1}{n}\log|(f^n)'|\,\mu_n.
\end{eqnarray*}
Using again that $(f^n)^\#(z)=|(f^n)'(z)|$ for any $z\in\Fix(f^n)$ and, by~\cite[Lemma 3.5]{OkuyamaStawiska}, 
\begin{gather*}
\int_{\P^1}\Phi_{g_f}(a,\cdot)\,\mu_n=\Phi_{g_f}(f^n(a),a)
\end{gather*}
for every $a\in\P^1$, the definition 
\eqref{eq:kernel} of the $g_f$-kernel function $\Phi_{g_f}$ completes the proof.
\end{proof}

\subsection{Reduction to the critical dynamics}

\begin{lemma}\label{reductiontoattr}
If $f$ has no super-attracting cycles, for any $n\in\N^*$ and any $r\in]0,1]$,
\[
\left|L_n^r(f)-L(f)-\frac{1}{d_n}\sum_{m|n}
\mu\left(\frac{n}{m}\right)(d^m+1)u_{m,n}(f,r)\right|\leq B(f,r)\frac{\sigma_{0}(n)}{d_n},
\]
where $B(f,r):=(2d-2)(2\sup_{\P^1}|g_f|+|\log r|)$.
\end{lemma}

\begin{proof}
Pick $r\in]0,1]$ and $n\geq1$. By the definition of $d_n$, we have
\begin{gather*}
1=\frac{1}{d_n}\sum_{m|n}\mu\left(\frac{n}{m}\right)(d^m+1).
\end{gather*}
For any $m\in\N^*$ dividing $n$ and
any $z\in\Fix(f^m)$, we have $(f^n)'(z)=(f^m)'(z)^{n/m}$ by the chain rule, 
and have
\[
\left|\sum_{z\in\Fix(f^m) \atop |(f^n)'(z)|\geq r}\log|(f^m)'(z)|-\sum_{z\in\Fix(f^m)}\log\max\left\{|(f^m)'(z)|,r^{m/n}\right\}\right|\leq m(2d-2)|\log r|
\]
since the number of attracting periodic points of period dividing $m$ of $f$ is at most $(2d-2)m$.
Recalling the definition \eqref{eq:multiplier} of $p_n$, we have
\begin{eqnarray*}
L_n^r(f) 
& = & \frac{1}{{n}d_n}\sum_{j=1}^{d_n}\int_0^{2\pi}\log\left|{(f^n)'(z_j^{(n)})}-re^{i\theta}\right|\frac{\mathrm{d}\theta}{2\pi}\\
& = & \frac{1}{nd_n} \sum_{z\in\Fix^*(f^n)}\log\max\left\{|(f^n)'(z)|,r\right\}\\
&= & \frac{1}{nd_n}\sum_{m|n}\mu\left(\frac{n}{m}\right)
\sum_{z\in\Fix(f^m)}\frac{n}{m}\log\max\left\{|(f^m)'(z)|,r^{m/n}\right\}\\
& = &\frac{1}{d_n}\sum_{m|n}
\mu\left(\frac{n}{m}\right)\frac{1}{m}\sum_{z\in\Fix(f^m)}\log\max\left\{|(f^m)'(z)|,r^{m/n}\right\},
\end{eqnarray*}
where the third equality is by the M\"obius inversion.
Hence recalling the definition of $\sigma_0(n)$,  by Lemma~\ref{th:error}, we have
\begin{multline*}
\left|L_n^r(f)-L(f)-\frac{1}{d_n}\sum_{m|n}\mu\left(\frac{n}{m}\right)(d^m+1)u_{m,n}(f,r)\right|\\
\le\frac{1}{d_n}\sum_{m|n}
(d^m+1)|\epsilon_m(f,r^{m/n})-u_{m,m}(f,r^{m/n})|
+\frac{(2d-2)|\log r|\cdot\sigma_0(n)}{d_n}\\
\le B(f,r)\frac{\sigma_0(n)}{d_n},
\end{multline*}
which ends the proof.
\end{proof}

\subsection{Quantitative approximation of the Lyapunov exponent}

\begin{lemma}\label{lm:inegtriangular}
For any $n\in\N^*$, any $z\in\Fix(f^n)$, and any $c\in \mathrm{Crit}(f)$,
\[
[f^n(c),c]\leq 2\cdot M(f)^n\cdot [c,z].
\]
\end{lemma}

\begin{proof}
Let $M_1:=\sup_{\P^1}f^\#>1$. It is clear that the map $f$ is $M_1$-Lipschitz in the chordal metric $[\cdot,\cdot]$. If $f^n(z)=z$, we have
\[ [f^n(c),c]\leq [f^n(c),z] +[c,z]\leq (M_1^n+1)\cdot [c,z]\]
and the conclusion follows since $M_1>1$.
\end{proof}

\begin{proof}[Proof of Theorem~\ref{tm:approx1}]
As there is no persistent parabolic and super-attracting cycle in $\rat_d$, the set $X$ of all elements in $\rat_d$ having neither
super-attracting nor parabolic cycles and no multiple critical points 
is the complement of a pluripolar subset in $\rat_d$, so $X$ is dense in $\rat_d$. 
Pick $f\in X$, $n\in\N^*$, and $r\in]0,1]$. 

(i) For any $m\in\N^*$ dividing $n$, recalling the definition of
$u_{m,n}(f,r)$,
we note that
\begin{eqnarray*}
u_{m,n}(f,r)=u_{m,n}(f,1) +\frac{1}{d^m+1}\sum_{z\in\Fix(f^m) \atop r\leq |(f^n)'(z)|<1}\frac{1}{m}\log|(f^m)'(z)|.
\end{eqnarray*}
Moreover, $f$ has at most $(2d-2)m$ attracting points of period dividing $m$, so that
\[
\left|\sum_{z\in\Fix(f^m) \atop r\leq |(f^n)'(z)|<1}\frac{1}{m}\log|(f^m)'(z)|\right|\leq \frac{(2d-2)m}{n}|\log r|
\le(2d-2)|\log r|,
\]
since $1>|(f^m)'(z)|=|(f^n)'(z)|^{m/n}\geq r^{m/n}$ for any $z\in \Fix(f^m)$, by the chain rule. Hence, recalling the definition of $\sigma_0$, we have
\[
\left|\frac{1}{d_n}\sum_{m|n}
\left(\frac{n}{m}\right)(d^m+1)\left(u_{m,n}(f,r)-u_{m,n}(f,1)\right)\right|\leq(2d-2)|\log r|\cdot\frac{\sigma_0(n)}{d_n}.
\]

(ii) For any $m\in\N^*$ dividing $n$, we have
\begin{eqnarray*}
\sum_{z\in \Fix(f^m) \atop |(f^n)'(z)|<1}\sum_{j=0}^{m-1}\log[f^j(z),c] 
& = & m\cdot\sum_{z\in \Fix(f^m) \atop |(f^m)'(z)|<1}\log[z,c].
\end{eqnarray*}
Recalling the definition of $u_{m,n}(f,r)$ and 
applying Lemma~\ref{th:lower} to each such $z\in \Fix(f^m)$ that $|(f^n)'(z)|<1$, 
we have
\begin{multline*}
\left|(d^m+1)u_{m,n}(f,1) -\sum_{c\in \mathrm{Crit}(f)}\Biggl(\log[f^m(c),c]
-\sum_{z\in \Fix(f^m) \atop |(f^m)'(z)|<1}\log[z,c]\Biggr)\right|\\
=\left|\sum_{z\in\Fix(f^m) \atop |(f^n)'(z)|<r}\frac{1}{m}\Biggl(
\sum_{c\in \mathrm{Crit}(f)}\sum_{j=0}^{m-1}\log[f^j(z),c]-\log|(f^m)'(z)|\Biggr)\right|
\leq  (2d-2)m\cdot B_1(f),
\end{multline*}
where the last inequality holds 
since $f$ has at most $(2d-2)m$ attracting points of period dividing $m$. 
Recalling the definition of $\sigma_1(n)$, we have
\begin{multline*}
\Biggl|\frac{1}{d_n}\sum_{m|n}\mu\left(\frac{n}{m}\right)(d^m+1)u_{m,n}(f,1)\\
\quad\quad -\frac{1}{d_n}\sum_{m|n}\mu\left(\frac{n}{m}\right)\sum_{c\in \mathrm{Crit}(f)}\Biggl(\log[f^m(c),c]
-\sum_{z\in \Fix(f^m) \atop |(f^m)'(z)|<1}\log[z,c]\Biggr)\Biggr|
\le(2d-2)B_1(f)\frac{\sigma_1(n)}{d_n}.
\end{multline*}
We finally reduced the proof of Theorem \ref{tm:approx1} to estimating
\[
\delta_n(f)=\delta_n(f,1):=\frac{1}{d_n}\sum_{m|n}\mu\left(\frac{n}{m}\right)\sum_{c\in \mathrm{Crit}(f)}\Biggl(\log[f^m(c),c]
-\sum_{z\in \Fix(f^m) \atop |(f^m)'(z)|<1}\log[z,c]\Biggr).
\]

(iii) We claim that for any $c\in \mathrm{Crit}(f)$, 
\[
\left|\log[f^m(c),c]-\sum_{z\in \Fix(f^m)\atop |(f^m)'(z)|<1}\log[z,c] \right|\leq 2(2d-2)m^2\left(\log M(f)-\log\left(\frac{\kappa}{2}\right)\right),
\]
where $\kappa\in(0,1)$ is the absolute constant appearing in Lemma \ref{lm:Przytycki}; pick $c\in \mathrm{Crit}(f)$ and $m\in\N^*$ dividing $n$. 
Recall that $\sup_{z,w\in \P^1}[z,w]\leq 1$. 
Assume first that
$\kappa\cdot M(f)^{-m}\leq [f^m(c),c]$. Then by Lemma~\ref{lm:inegtriangular}, 
we deduce that for any $z\in\Fix(f^m)$,
\begin{gather*}
\kappa\cdot 2^{-1}M(f)^{-2m}\leq 2^{-1}M(f)^{-m} [f^m(c),c]\leq [z,c],
\end{gather*}
so that since $f$ has at most $(2d-2)m$ attracting periodic points of period $m$,
we have
\begin{align*}
-m\log M(f)+\log\kappa\leq \log[f^m(c),c]-\sum_{z\in \Fix(f^m) \atop |(f^m)'(z)|<1}\log[z,c],
\end{align*}
and
\begin{align*}
 \log[f^m(c),c]-\sum_{z\in \Fix(f^m) \atop |(f^m)'(z)|<1}\log[z,c] & \leq  (2d-2)m\left(2m\log M(f)-\log\left(\frac{\kappa}{2}\right)\right)\\
& \leq  2(2d-2)m^2\left(\log M(f)-\log\left(\frac{\kappa}{2}\right)\right).
\end{align*}
Assume next that $\kappa\cdot M(f)^{-m}> [f^m(c),c]$. 
By Lemma~\ref{lm:Przytycki} applied to the trivial family $(f)$ and 
its (constant) marked critical point $c$
(recall that the constant $\kappa$ given by Lemma~\ref{lm:Przytycki} depends only on $d$),
$c$ belongs to the immediate basin of an attracting periodic point $z_0$ of $f$ of period $k$ dividing $m$, and we have
\[
2[f^{m}(c),c]\geq[z_0,c]\quad\text{and}\quad[c,\J_f]\geq \kappa M(f)^{-m}.
\]
Hence we have $-\log2\leq \log[f^m(c),c]-\log[z_0,c]$, so that
\begin{gather*}
-\log 2\le\log[f^m(c),c]-\sum_{z\in \Fix(f^m)\atop |(f^m)'(z)|<1}\log[z,c].
\end{gather*}
Noting that any attracting point $z\in\Fix(f^m)\setminus\{z_0\}$ lies in a Fatou component of $f$ which does not contain $z_0$, we also have
$1\geq[z,c]\geq [c,\J_f]\geq \kappa M(f)^{-m}$
for every such $z\in\Fix(f^m)\setminus\{z_0\}$ that $|(f^m)(z)|<1$.  Moreover, by Lemma~\ref{lm:inegtriangular}, $[f^m(c),c]\le M(f)^m[z_0,c]$. Hence,
since $f$ has at most $(2d-2)m$ attracting periodic points of period dividing $m$,
we have
\begin{eqnarray*}
\log[f^m(c),c]-\sum_{z\in \Fix(f^m)\atop |(f^m)'(z)|<1}\log[z,c] 
&\leq& (2d-2)m\cdot \left(m\log M(f)-\log\kappa\right)\\
&\leq& (2d-2)m^2\left(\log M(f)-\log\kappa\right).
\end{eqnarray*}
Hence the claim holds.

Since $f$ has exactly $2d-2$ critical points taking into account their 
multiplicities, letting $C_2:=2(2d-2)^2\max\left\{1,|\log(\kappa/2)|\right\}$, we have
\[
\left|\delta_n(f)\right| \leq C_2\cdot\left(\log M(f)+1\right)\frac{\sigma_2(n)}{d_n},
\]
 by the definition of $\sigma_2(n)$. 

(iv) Recall that
\[\frac{1}{2}\log d\leq L(f)=\int_{\P^1}\log(f^\#)\,\mu_f\leq \log\left(\sup_{\P^1}f^\#\right)\]
and that 
\begin{eqnarray*}
d_n & = & \sum_{m | n}\mu\left(\frac{n}{m}\right)(d^m+1)\\
& \geq & \sum_{m | n}\mu\left(\frac{n}{m}\right)d^m\geq \left(1-d^{-1}\right) d^n,
\end{eqnarray*}
by definition of $d_n$. Hence, all the above intermediate estimates yield
\begin{eqnarray}
\left|\frac{1}{d_n}\int_0^{2\pi}\log\left|p_n(f,re^{i\theta})\right|\frac{\mathrm{d}\theta}{2\pi}-L(f)\right| \leq B_3(f,r)\frac{\sigma_2(n)}{d^n}
\label{almostdone}
\end{eqnarray}
for any $f\in X$, where
\[
B_3(f,r)=C_3 \cdot \left(\sup_{\P^1}|g_f|+\log\left(\sup_{\P^1}f^\#\right)+|\log r|\right)
\]
for some constant $C_3>0$ that depends only on $d$. 
Since both sides of \eqref{almostdone} depend continuously on $f\in \rat_d$
and $X$ is dense in $\rat_d$, the above estimate \eqref{almostdone} 
still holds for any $f\in\rat_d$. 
\end{proof}

\subsection{Application: degeneration of the Lyapunov exponent}

Consider a holomorphic family $(f_t)_{t\in\D^*}$ of degree $d>1$ rational maps parameterized by the punctured unit disk, and assume it extends to a meromorphic family over $\D$, i.e. $f_t\in\mathcal{O}(\D)[t^{-1}](z)$.

\begin{theorem}\label{tm:degenerate}
There exists a non-negative $\alpha\in\R$ such that, as $t\to 0$,
\[
L(f_t)=\alpha\cdot\log|t|^{-1}+o(\log|t|^{-1}).
\]
\end{theorem}
This is a special case of \cite[Theorem C]{Favre-degeneration} and can also be obtained as the combination of~\cite[Proposition 3.1]{DeMarco-intersection} and~\cite[Theorem 1.4]{DeMarco2}. We provide here a simple proof as an application of Theorem~\ref{tm:approx1}.

\begin{proof}
We can write
$p_n(t,w)=t^{-N_n}h_n(t,w)$, where $h_n:\D\times\C\to\C$ 
is holomorphic and $N_n\in\mathbb{N}$. 
We rely on the following key lemma.

\begin{lemma}\label{lm:atinfinity}
There exist $C_1,C_2>0$ such that for any $t\in\D^*_{1/2}$,
\[
\sup_{z\in\mathbb{P}^1}\max\left\{\left|g_{f_t}(z)\right|,\log(f_t^\#(z))\right\} \leq C_1\log|t|^{-1}+C_2.
\]
\end{lemma}
Once Lemma~\ref{lm:atinfinity} is at our disposal, by Theorem~\ref{tm:approx1}, 
there is $C>0$ such that for any $n\in\N^*$ and any $t\in\D^*_{1/2}$,
\[
\left| L(f_t)-\frac{1}{d_n}\int_0^{2\pi}\log|p_n(f_t,e^{i\theta})|\frac{\mathrm{d}\theta}{2\pi}\right|\leq C\left(C_1\log|t|^{-1}+C_2\right)\cdot \frac{\sigma_2(n)}{d^n},
\]
so that dividing both sides by $\log|t|^{-1}$ and making $t\to0$, 
there is $C'>0$ such that for all $n\in\N^*$
\[
\frac{N_n}{d_n}-C'\frac{\sigma_2(n)}{d^n}\leq \liminf_{t\to0}\frac{L(f_t)}{\log|t|^{-1}}\leq \limsup_{t\to0}\frac{L(f_t)}{\log|t|^{-1}}\leq \frac{N_n}{d_n}+C'\frac{\sigma_2(n)}{d^n},
\]
and then making $n\to\infty$, we get
\[
\lim_{t\to0}\frac{L(f_t)}{\log|t|^{-1}}=\lim_{n\to\infty}\frac{N_n}{d_n}=:\alpha\geq0.
\]
This concludes the proof. 
\end{proof}

\begin{proof}[Proof of Lemma~\ref{lm:atinfinity}]
There is a meromorphic family $(F_t)_{t\in\D}$ of homogeneous polynomial
endomorphisms of $\C^2$ such that for every $t\in\D^*$, $F_t$ is a lift of $f_t$
and that the holomorphic function $t\mapsto \mathrm{Res}(F_t)$ on $\D$ 
vanishes only at $t=0$.
According to \cite[Lemma 3.3]{DeMarco-intersection} (or \cite[Proposition 4.4]{Favre-degeneration}), there exist constants $C\geq1$ and $\beta>0$ such that
\begin{eqnarray}
\frac{1}{C}|t|^{\beta}\leq \frac{\|F_t(p)\|}{\|p\|^d}\leq C\label{eq:degenerate}
\end{eqnarray}
for any $p\in\C^2\setminus\{0\}$ and any $t\in\D^*$.

For any $t\in\D^*$, set $u_t(z):=\log(\|F_t(p)\|/\|p\|^d)$ 
on $\P^1$, where $p\in\pi^{-1}(z)$. The function $u_t$ on $\P^1$ is well-defined 
by the homogeneity of $F_t$. Recalling the definition of $g_{F_t}$, we have
$g_{F_t}(z)=\sum_{n=0}^\infty(u_t\circ f_t^n(z))/d^{n+1}$
uniformly on $\P^1$, so that by \eqref{eq:degenerate},
\[
\sup_{z\in\mathbb{P}^1}\left|g_{F_t}(z)\right|\leq \frac{1}{d-1}\sup_{z\in\mathbb{P}^1}\left|u_t(z)\right|\leq \frac{1}{d-1}\left(\beta\log|t|^{-1}+\log C\right).
\]
Recalling the definition of $I_{g_{F_t}}$ and 
the formula $I_{g_{F_t}}=-(\log|\mathrm{Res}(F_t)|)/(d(d-1))$,
we also have $g_{f_t}=g_{F_t}+(\log|\mathrm{Res}(F_t)|)/(2d(d-1))$ on $\P^1$
for every $t\in\D^*$.
Hence we obtain the desired upper bound of $\sup_{z\in\mathbb{P}^1}|g_{f_t}(z)|$
since $t\mapsto\mathrm{Res}(F_t)$ is a holomorphic function of $\D$ 
vanishing only at $t=0$.

To conclude the proof, we use the same strategy for giving an upper bound for $\log\sup_{z\in\mathbb{P}^1}f_t^\#(z)$. Recall the following formula
\[
f_t^\#(z)=\frac{1}{d}\left|\det DF_t(p)\right|\frac{\|p\|^2}{\|F_t(p)\|^2}
\]
on $\P^1$, where $p\in\pi^{-1}(z)$ (see, e.g.,~\cite[Theorem~4.3]{Jonsson-lyap}).
Hence by \eqref{eq:degenerate}, we have
\[
\log(f_t^\#(z))\leq \log\frac{\left|\det DF_t(p)\right|}{d\cdot \|p\|^{2d-2}}
+2\left(\beta\log|t|^{-1}+\log C\right).
\]
Write now $F_t=(P_t,Q_t)$ where $P_t(z,w)=\sum\limits_{j=0}^d a_j(t)z^jw^{d-j}$ and $Q_t(z,w)=\sum\limits_{j=0}^d b_j(t)z^j w^{d-j}$ with $a_j(t)=t^{-\gamma}\tilde{a}_j(t)$ and $b_j(t)=t^{-\gamma}\tilde{b}_j(t)$ for some $\gamma\in\N^*$ and some $\tilde{a},\tilde{b}\in\mathcal{O}(\D)$.
In particular, there exists a constant $C'\geq1$ such that for all $t\in\D(0,\frac{1}{2})$ and all $0\leq j\leq d$, we have $\max\{|\tilde{a}_j(t)|,|\tilde b_j(t)|\}\leq C'$, and
\begin{align*}
|\det DF_t(p)| & = \left|\frac{\partial P_t}{\partial z}(p)\frac{\partial Q_t}{\partial w}(p)-\frac{\partial P_t}{\partial w}(p)\frac{\partial Q_t}{\partial z}(p)\right|\\
& \leq 2d^2|t|^{-2\gamma}\sum_{j,\ell=0}^d|\tilde{a}_j(t)\tilde{b}_\ell(t)|\cdot\|p\|^{2(d-1)}\\
& \leq 2d^4\cdot{C'}^2\|p\|^{2(d-1)}\cdot |t|^{-2\gamma}.
\end{align*}
This gives $\log(f_t^\#(z))\leq 2(\beta+\gamma)\log|t|^{-1}+\log C''$ for some constant $C''\geq1$, which ends the proof.
\end{proof}

\section{Equidistribution towards the bifurcation currents}\label{sec:main}

Let $\Lambda$ be a quasi-projective variety such that, either $\Lambda \subset \cM_d$, or parametrizing an algebraic family $(f_\lambda)_{\lambda\in\Lambda}$ of degree $d$ rational maps on $\P^1$.
\subsection{The proof of Theorem~\ref{tm:vitessemoyennesalgebraic}}

Pick any compact subset $K$ in $\Lambda$, and set
$C_1(K):=
\sup_{\lambda\in K}
C([f_\lambda])\geq\frac{1}{2}\log d$, where $C([f_\lambda])$ is given by Theorem~\ref{tm:approx1}.
We remark that for every $n\in\N^*$ and every $\rho\in]0,1]$,
\begin{gather*}
T_n^1(\rho):=\frac{1}{d_n}\int_0^{2\pi}[\Per_n(\rho e^{i\theta})]\frac{\mathrm{d}\theta}{2\pi}=dd^c\left(\frac{1}{d_n}\int_0^{2\pi}\log|p_n(\lambda,\rho e^{i\theta})|\frac{\mathrm{d}\theta}{2\pi}\right).
\end{gather*}
 Pick any $1\leq p\leq \min\{m,2d-2\}$, 
any $\un=(n_1,\ldots,n_p)\in(\N^*)^p$, and 
any $\underline{\rho}=(\rho_1,\ldots,\rho_p)\in]0,1]^p$. 

Assume first that $p=1$, i.e. $\un=n\in\N^*$ and $\underline{\rho}=\rho\in]0,1]$,  
and pick any continuous DSH $(m-1,m-1)$-form $\Psi$ on $\Lambda$
supported in $K$.
By definition, we can write $dd^c \Psi= T^+-T^-$, where $T^\pm$ are positive 
measures of finite masses on $\Lambda$.
By Stokes and Theorem~\ref{tm:approx1},
we have
\begin{align*}
\left|\left\langle T_n^1(\rho)-T_\bif,\Psi\right\rangle\right| 
& =  \left|\int_K\left(\frac{1}{d_n}\int_0^{2\pi}\log|p_n(\lambda,\rho e^{i\theta})|\frac{\mathrm{d}\theta}{2\pi}-L(\lambda)\right) dd^c\Psi\right|\\
& \leq \int_K\left|\frac{1}{d_n}\int_0^{2\pi}\log|p_n(\lambda,\rho e^{i\theta})|\frac{\mathrm{d}\theta}{2\pi}-L(\lambda)\right|(T^+ + T^-) \\
& \leq 
\left(AC_1(K) \left(1+\left|\log \rho\right|\right)\frac{\sigma_2(n)}{d^n}\right)
(\|T^+\|+\|T^-\|),
\end{align*}
which completes the proof of Theorem \ref{tm:vitessemoyennesalgebraic} in this case 
by the definition of $\|\Psi\|^*_{\mathrm{DSH}}$.

We now assume that 
$2\leq p\leq \min\{m,2d-2\}$. 
Setting 
$S_j=S_j(\un,\underline{\rho}):=
(\bigwedge_{1\le\ell<j}T_\bif)\wedge(\bigwedge_{j<k\le p}T_{n_k}(\rho_k))$ for any $1\le j\le p$, which is a positive closed current of bidegree
$(p-1,p-1)$ on $\Lambda$,
we have
\begin{eqnarray}
T_\un^p(\underline{\rho})-T_\bif^p 
& = & \sum_{j=1}^pS_j\wedge(T_{n_j}(\rho_j)-T_\bif).\label{eq:decomposition}
\end{eqnarray}
Pick any continuous DSH $(m-p,m-p)$-form $\Psi$ on $\Lambda$
supported in $K$, and write $dd^c \Psi= T^+-T^-$ where $T^\pm$ are 
positive closed $(m-p+1,m-p+1)$ currents of finite masses 
on $\Lambda$.
Then by Stokes formula, we have
\begin{eqnarray*}
\left\langle T_\un^p(\underline{\rho})-T_\bif^p,\Psi\right\rangle 
&=&\sum_{j=1}^p\left\langle S_j\wedge(T_{n_j}^1(\rho_j)-T_\bif),\Psi\right\rangle\\
& = & \sum_{j=1}^p\int_K\left(\frac{1}{d_n}\int_0^{2\pi}\log|p_n(\lambda,\rho e^{i\theta})|\frac{\mathrm{d}\theta}{2\pi}-L(\lambda)\right) S_j\wedge dd^c\Psi.
\end{eqnarray*}
Since the masses can be computed in cohomology, 
by \eqref{eq:boundedmass},
there is $C_2>0$ independent of $K$, $\Psi$ and $T^{\pm}$, such that
for every $1\le j\le p$,
$\int_\Lambda\ S_j\wedge (T^++T^-) \leq C_2 (\|T^+\|+\|T^-\|)$.
Then by Theorem~\ref{tm:approx1}, we have
\begin{align*}
&\left|
\int_K\left(\frac{1}{d_n}\int_0^{2\pi}\log|p_n(\lambda,\rho e^{i\theta})|\frac{\mathrm{d}\theta}{2\pi}-L(\lambda)\right) S_j\wedge dd^c\Psi\right| \\
\leq & 
\left(AC_1(K)\left(1+|\log \rho_j|\right)\frac{\sigma_2(n_j)}{d^{n_j}}\right)\int_\Lambda
S_j\wedge (T^++T^-)\\
\leq & 
AC_1(K)C_2\left(1+|\log \rho_j|\right)\frac{\sigma_2(n_j)}{d^{n_j}}(\|T^+\|+\|T^-\|)
\end{align*}
for any $1\le j\le p$, which ends the proof of Theorem \ref{tm:vitessemoyennesalgebraic}. \qed

\begin{remark} \normalfont \normalfont
As in \cite{BB3}, we deduce from Theorem~\ref{tm:vitessemoyennesalgebraic} the density in the support of $T_\bif^p$ of parameters having $p$ distinct neutral cycles. We can actually give a more precise statement: taking any sequence of $p$-tuple of integers $(\un_k)$ in $(\N^*)^p$ such that $\min_j n_{j,k} \to \infty$, 
we have that the set of parameters $\lambda$
such that $f_\lambda$ has $p$ distinct neutral cycles of exact periods $n_{1,k}, \dots, n_{p,k}$ for some $k\in\N^*$  are dense in the support of $T_\bif^p$.
\end{remark}

\subsection{The proof of Corollary~\ref{tm:outsideluripolar}}\label{sec:PB}
Pick $1\leq p\leq \min\{m,2d-2\}$. We recall some basics on PB measures.
For each $\rho>0$, let $\lambda_{\mathbb{S}_\rho}$
the Lebesgue probability measure on the circle $\mathbb{S}_\rho$. Let $\theta:\R^+ \to \R^+$ be a smooth function with compact support in $]0,1[$ such that $\int_0^1 \theta=1$.
We consider the smooth measure $\tilde{\nu}$ defined as 
\[\tilde{\nu}:= \bigotimes_{1\leq j\leq p} \int_0^1 \lambda_{\mathbb{S}_{\rho_j}} \theta(\rho_j) d\rho_j. \]
We say that a probability measure $\nu$ on $(\P^1)^p$ is PB 
(or has bounded potential) if there exists a constant $C \ge 0$ 
such that 
\[ 
|\langle \nu-\tilde{\nu},
\varphi \rangle | \leq  C \|\varphi \|_{\mathrm{DSH}}^* 
\]
for all $\varphi$ which is DSH on $(\P^1)^p$, and then
let $C_\nu\ge 0$ be the minimal $C\ge 0$ satisfying the above inequality for every $\varphi$.
For example, $\tilde{\nu}$ is PB on $(\P^1)^p$,
and $\lambda_{\mathbb{S}_\rho}$ is PB on $\P^1$. We claim that the positive closed $(p,p)$-current 
\[
T_{\un}^p(\nu):=\frac{1}{d_{|\un|}}\int_{(\P^1)^p}
\bigwedge_{j=1}^p[\Per_{n_j}(w_j)]\nu(w_1,\cdots,w_p)
\]
on $\Lambda$ is well-defined
for any PB measure $\nu$ on $(\P^1)^p$ 
and any $\un=(n_1,\ldots,n_p)\in(\N^*)^p$. Indeed,
the set of all $\uw=(w_1,\ldots,w_{p})\in\C^{p}$ such that  
$\bigcap_{i=1}^{p}\Per_{n_{i}}(w_i)$ is not of pure codimension $p$ in $\Lambda$ 
is analytic. Hence for any $\uw=(w_1,\ldots,w_{p})\in\C^{p}$
except for a pluripolar subset and any $\un=(n_1,\ldots,n_p)\in(\N^*)^p$,
the current $\bigwedge_{i=1}^{p}[\Per_{n_{i}}(w_i)]$ on $\Lambda$
is well defined. So in particular, since PB measures on $(\P^1)^p$
give no mass to pluripolar sets, the current $T_{\un}^p(\nu)$ is also well defined.

Observe that $T_{\un}^p(\tilde{\nu})$ give no mass to pluripolar sets (hence analytic sets) since it has bounded potentials. Arguing as above shows that for any $\uw=(w_1,\ldots,w_{p})\in\C^{p}$ outside 
a pluripolar subset and any $\un=(n_1,\ldots,n_p)\in(\N^*)^p$, the current $\bigwedge_{i=1}^{p}[\Per_{n_{i}}(w_i)]$ is well-defined. So for any  PB measures on $(\P^1)^p$, the current $T_{\un}^p(\nu)$ gives no mass to analytic sets. 

Here is another description of $(d_{|\un|})^{-1}\bigwedge_{i=1}^{p}[\Per_{n_i}(w_i)]$ and
$T^p_{\un}(\nu)$;
let $\Gamma_{\un}$ be the analytic set of dimension $m$ in $\Lambda\times(\P^1)^p$
defined as
\begin{gather*}
\Gamma_{\un}:=\{(\lambda,(z_1,\ldots,z_p))\in\Lambda\times(\P^1)^p:
z_j \in \mathrm{Fix}^*(f^{n_j}_\lambda)\text{ for every }1\le j\le p\}.
\end{gather*}
Let $F_\un:\Gamma_\un  \to (\P^1)^p$ be a holomorphic map defined by
\begin{gather*}
F_\un (\lambda, z_1,\dots,z_p)=((f_\lambda^{n_1})'(z_1), \dots, (f_\lambda^{n_p})'(z_p)), 
\end{gather*}
and $\mathcal{P} : \Gamma_{\un} \to \Lambda$ be the restriction to $\Gamma_{\un}$
of the projection $\Lambda\times(\P^1)^p\to\Lambda$. Consider $\pi:\widetilde{\Gamma_{\un}}\to \Gamma_{\un}$ a desingularization of $\Gamma_{\un}$. The map $\widetilde{F_\un}:= F_\un \circ \pi$ is holomorphic and the map $\widetilde{\mathcal{P}}:=\mathcal{P} \circ \pi$ is an analytic map. If $\nu$ is a smooth PB measure in $(\P^1)^p$, then:
 \begin{equation}
 T^p_{\un}(\nu) = \frac{1}{\prod_j n_j d_{n_j}} \widetilde{\mathcal{P}}_*\left( \left({\widetilde{F_\un}}|_{\tilde\Gamma_{\un}}\right)^*(\nu)\right)\quad\text{on }\Lambda.
 \end{equation}   
Indeed, observe that, when testing against a smooth form, there is always one term that is smooth when computing the pull-back and push-forward. 

\begin{theorem}\label{tm:vitessemoyennesPB}
Let $\Lambda$ be a quasi-projective variety such that, either $\Lambda \subset \cM_d$, or parametrizing an algebraic family $(f_\lambda)_{\lambda\in\Lambda}$ of degree $d$ rational maps on $\P^1$.
Then for any compact subset $K$ in $\Lambda$, 
there exists $C(K)>0$ such that for any $1\leq p\leq \min\{m,2d-2\}$, 
any $\un=(n_1,\ldots,n_p)\in(\N^*)^p$,
any PB measure $\nu$ in $(\P^1)^p$, 
and any continuous $\mathrm{DSH}$-form $\Psi$ of bidegree $(m-p,m-p)$ on $\Lambda$
supported in $K$, we have
\[
\left|\left\langle T_{\un}^p(\nu)-T_\bif^p,\Psi\right\rangle\right|\leq C(K)\cdot\left(1+ C_\nu \right)\max_{1\le j\le p}\left(\frac{1}{n_j}\right)\|\Psi\|_{\mathrm{DSH}}.
\]
\end{theorem}

\begin{proof}
Pick $1\leq p\leq \min\{m,2d-2\}$ and a PB measure $\nu$ on $(\P^1)^p$. Consider first the case where $\nu$ is smooth. 
Pick $\un=(n_1,\ldots,n_p)\in(\N^*)^p$ and 
a smooth DSH form $\Psi$ of bidegree $(m-p,m-p)$ on $\Lambda$
with compact support in $K$.
By Theorem~\ref{tm:vitessemoyennesalgebraic} and our choice of $\tilde{\nu}$,
there is $C(K)>0$ depending only on $K$ such that
\[
\left|\left\langle T_{\un}^p(\tilde{\nu})-T_\bif^p,\Psi\right\rangle\right|
\leq C(K)\max_{1\leq j\leq p}\left(\frac{\sigma_2(n_j)}{d^{n_j}}\right)\|\Psi\|_{\mathrm{DSH}},
\]
and we will show that
\[
\left|\left\langle T_{\un}^p(\tilde{\nu})-T_{\un}^p(\nu),\Psi\right\rangle\right|
\leq C(K)C_\nu \max_{1\leq j\leq p}\left(\frac{1}{n_j}\right)\|\Psi\|_{\mathrm{DSH}}.
\]
By the above description of $T_{\un}^p(\nu)$ and
the definition of PB measures, we have
\begin{gather*}
 |\left\langle T_\un^p(\tilde{\nu})-T^p_{\un}(\nu),\Psi\right\rangle| 
\leq  C_\nu \left\| \frac{1}{\prod_j n_j d_{n_j}}\left(\widetilde{F_\un}\right)_*(\widetilde{\mathcal{P}}^*(\Psi)) \right\|_{\mathrm{DSH}}^*.
\end{gather*}
As taking $dd^c$ commutes with taking pull-pack or push-forward, 
writing as $dd^c \Psi=T^+-T^-$, 
where $T^\pm$ are smooth (because $\Psi$ is smooth) positive closed currents of bidegree $(m-p+1,m-p+1)$
of finite masses on $\Lambda$,
one simply has to estimate the mass
$\|(\widetilde{F})_*(\widetilde{\mathcal{P}}^*(T^\pm))\|$. 
Computing those masses can be done in cohomology testing against
$\bigwedge_{i\neq j} \omega_i$ for all $1 \leq j \leq p$, where $\omega_i$ is the Fubini Study form on the $i$-th factor of $ (\P^1)^p$. By duality,
this is the same as controlling:
	\begin{align*}
	\left\langle  \frac{1}{\prod\limits_{j=1}^p n_j d_{n_j}}\widetilde{\mathcal{P}}_*\left( \left(\widetilde{{F_\un}}\right)^*\left(\bigwedge_{i\neq j} \omega_i \right)\right)  ,\  T^\pm \right\rangle
	\end{align*}
for any $j$. Finally, for any $j$, one has to control the mass:
	\[ \left\| \frac{1}{\prod\limits_{j=1}^p n_j d_{n_j}}\widetilde{\mathcal{P}}_*\left( \left(\widetilde{{F_\un}}\right)^*\left(\bigwedge_{i\neq j} \omega_i \right)\right)\right\|. \]
By symmetry, consider the case where $j=p$. Let $\un'=(n_1,\dots,n_{p-1})$ and consider the associated map $\widetilde{F_{\un'}}$. Now take a generic point $(z^0_i)_{i\leq p-1} \in \left(\P^1 \right)^{p-1}$ and consider the line $L := \{z=(z_1,\ldots,z_p) \in \left(\P^1 \right)^p, \ \forall i\leq p-1, \ z_i=z_i^0\}$.  Then the degree of the preimage of this line under $\widetilde{F_{\un}}$ is $d_{n_p}$ times the degree of the preimage of the point $(z^0_i)_{i\leq p-1}$ under $\widetilde{F_{\un'}}$. So pushing-forward, we see that
		\[ \left\|  \frac{1}{\prod\limits_{j=1}^p n_j d_{n_j}}\widetilde{\mathcal{P}}_*\left( \left(\widetilde{{F_\un}}\right)^*\left(\bigwedge_{i\neq p} \omega_i \right)\right)\right\| \leq C \frac{1}{n_p},\]
for some constant $C\geq0$ that does not depend on $\un$.

In particular, we deduce that 
 \[\left\| \frac{1}{\prod\limits_{j=1}^p n_j d_{n_j}}\left(\widetilde{F_\un}\right)_*(\widetilde{\mathcal{P}}^*(\Psi)) \right\|_{\mathrm{DSH}}^*  \leq C \times \max_{1\le j\le p}\left(\frac{1}{n_j}\right)\|\Psi\|_{\mathrm{DSH}}, \]
where $C\geq0$ is (another) constant that does not depends on $\un$, which implies the wanted result for $\Psi$ and $\nu$ smooth. By a regularization argument \cite{DinhSibonysuper}, the result follows for $\Psi$ continuous, replacing $C(K)$ by a constant given by a (small) larger neighborhood of $K$. Finally, we extend the result to any PB measure $\nu$ using again an approximation of $\nu$.
\end{proof}

\begin{remark} \normalfont \normalfont 
The order $O(\max_j(n_j^{-1}))$ as $\min_j(n_j)\to\infty$
in the right-hand side is sharp. 
Indeed, for the quadratic polynomials family $(z^2+\lambda)_{\lambda\in\C}$, 
it has been shown in \cite{multipliers} that 
the sequence $(2^{-n+1}[\Per(n,e^{2n})])_{n}$ (recall $2_n\sim 2^n$
as $n\to\infty$)
of measures on $\C$
converges to $dd^c\max\{g,4-2 \log 2\}$, where $g$ is the Green function of the Mandelbrot set. Since $4-2 \log 2>0$, 
this measure is not proportional to $\mu_\bif$. 
On the other hand, if $\nu_n=\lambda_{\mathbb{S}_{e^{2n}}}$ we have
$C_{\nu_n}=O(n)$ as $n\to\infty$, 
where $\lambda_{\mathbb{S}_{e^{2n}}}$ is the probability Lebesgue measure 
on the circle of center $0$ and radius $e^{2n}$ in $\C$, which is PB. 
So one cannot improve the order $O(n^{-1})$ as $n\to\infty$
in the right-hand side for this family; otherwise,
$2^{-n}[\Per(n,e^{2n})]$ would tend to $\mu_\bif$ as $n\to\infty$.
\end{remark}

\begin{proof}[Proof of Corollary~$\ref{tm:outsideluripolar}$]
Pick an integer $1\leq p\leq \min\{\dim\Lambda,2d-2\}$ and 
a sequence $(\un_k)_{k\in\N^*}$ in $(\N^*)^p$,
$\un_k=(n_{1,k},\ldots,n_{p,k})$, 
such that $\sum_{k\in\N^*}\max_j(n_{j,k}^{-1})<\infty$.
Pick a compact subset $K$ of $\Lambda$, and recall that we denote by $\beta$ 
the K\"ahler form on $\Lambda$.
Take $\Psi$ a smooth form with support in $K$, then $ \|\Psi \|_{\mathcal{C}^2}\cdot\beta^{m-p+1} \pm dd^c \Psi \geq 0$. We have seen in the above proof that for any smooth $PB$ measure $\nu$, we have:
\[
\left|\left\langle T_{\un}^p(\tilde{\nu})-T_{\un}^p(\nu),\Psi\right\rangle\right|
\leq  \|\Psi \|_{\mathcal{C}^2}\cdot \left|\langle U_\nu,  \frac{1}{\prod_j n_j d_{n_j}}\left(\widetilde{F_\un}\right)_*(\widetilde{\mathcal{P}}^*(\beta^{m-p+1})) \rangle \right |
\]	
where $ U_\nu$ denote the negative Green quasi-potential of $\nu$: $\nu +dd^c U_\nu=\tilde{\nu}$.   In particular, this extend to any measure $\delta_\uw$ for $\uw \in \C^p$ outside a pluripolar set (we let $U_\uw$ denote the negative Green quasi-potential of $\delta_\uw$, i.e. satisfying $\delta_\uw +dd^c U_\uw=\tilde{\delta_\uw}$).
 Observe that, as shown in the above proof,
the positive closed $(1,1)$-current
\[   \frac{1}{\prod_j n_j d_{n_j}}\left(\widetilde{F_\un}\right)_*(\widetilde{\mathcal{P}}^*(\beta^{m-p+1})) 
\]
on $(\P^1)^p$ has mass $O(\max_j(n_{j}^{-1}))$ as $\min_j n_j\to\infty$,
so that the positive closed $(1,1)$-current 
\[  
T_{\mathrm{test}}:=  \sum_{k=1}^\infty \frac{1}{\prod_j n_{j,k} d_{n_{j,k}}}\left(\widetilde{F_{\un_k}}\right)_*(\widetilde{\mathcal{P}}^*(\beta^{m-p+1}))
\]
on $(\P^1)^p$ has finite mass. In particular, the set of  all
$\uw \in \C^p$ for which
the negative Green quasi-potential $U_\uw$ of $\delta_\uw$,
satisfies $\langle U_\uw ,  T_{\mathrm{test}} \rangle =-\infty$ is pluripolar. 

Hence for any $\uw\in\C^p$ except for a  pluripolar subset 
and for any $\mathcal{C}^2$-test form $\Psi$ {of bidegree $(m-p,m-p)$ 
on $\Lambda$} with support in $K$, we have
\begin{align*}
\left| \left\langle T_{\un_k}^p(\tilde{\nu})-\frac{1}{d_{|\un_k|}}\bigwedge_{i=1}^{p}[\Per_{n_{i,k}}(w_i)],\Psi\right\rangle \right|  
\leq  
\left| \left\langle U_\uw,    \frac{\|\Psi\|_{C^2}}{\prod_j n_{j,k} d_{n_{j,k}}}\left(\widetilde{F_{\un_k}}\right)_*(\widetilde{\mathcal{P}}^*(\beta^{m-p+1}))\right\rangle \right|
\end{align*}
goes to $0$ as $k\to\infty$ and the conclusion follows.
\end{proof}

\section{Transversality of periodic critical orbit relations}\label{Transversality of periodic critical orbit relations}

\subsection{Infinitesimal deformations of rational maps}
Pick $f\in\rat_d$. The orbit 
\[
\mathcal{O}(f):=\{\phi^{-1}\circ f\circ \phi\in\rat_d: \phi\in\PSL_2(\C)\},
\]
of $f$ under the conjugacy action of $\PSL_2(\C)$ on $\rat_d$ is a $3$ dimensional complex analytic submanifold in $\rat_d$.

A tangent vector to $\rat_d$ at $f$ is an equivalence class of holomorphic
maps $\phi : \D \to \rat_d$ such that $\phi(0) = f$ under the relation 
$\phi\sim \psi$ iff $\phi'(0) = \psi'(0)$. 
The vector space of all tangent vectors at $f$ is denoted by $ T_f\rat_d$.
A tangent vector $\zeta\in T_f\rat_d$ can be identified to a 
section of the line bundle $f^*(T\P^1)$, where $T\P^1$ denotes 
the {holomorphic} tangent bundle on $\P^1$.
Concretely, on any local chart $U\subset\P^1$, 
$\zeta$ is a map  $z\in U \mapsto \zeta(z) \in T_{f(z)} \P^1$ and we shall view it as
a holomorphic function on $U$ by identifying the fiber $T_{f(z)}\P^1$ with $\C$.
To any tangent vector $\zeta\in T_f\rat_d$, we may thus attach a rational
vector field on $\P^1$ 
\begin{gather*}
\eta_\zeta(z)\pe  -D_zf^{-1} \cdot \zeta(z) \in  T_{z} \P^1,
\end{gather*}
whose poles are in $\mathrm{Crit}(f)$.
If $f$ has only simple critical points, then no poles of $\eta_\zeta$ are
of order more than $1$.

If $f$ is postcritically finite, i.e., the postcritical set
\begin{gather*}
\mathcal{P}(f):=\bigcup_{n\in\N^*}f^n(\mathrm{Crit}(f))
\end{gather*}
of $f$ is a finite subset in $\P^1$, then
we denote by $\mathcal{T}(\mathcal{P}(f))$ 
the vector field on $\mathcal{P}(f)$,
and a vector field $\tau\in\mathcal{T}(\mathcal{P}(f))$ is said to be \emph{guided} by $\zeta\in T_f\rat_d$ if
\begin{gather*}
\tau=f^*\tau+\eta_\zeta\quad\text{on }\mathcal{P}(f)\text{ and}\\
\tau\circ f=\zeta\quad\text{on }\mathrm{Crit}(f).
\end{gather*}

For the sequel, we will rely on the following crucial result 
(see \cite{buffepstein,favregauthier}). 

\begin{proposition}[Buff-Epstein]
If $f\in\rat_d$ is postcritically finite 
and neither is conjugate to $z^{\pm 2}$ nor is a Latt\`es map, 
then a tangent vector 
$\zeta\in T_f\rat_d$ is tangent to $\mathcal{O}(f)$ if and only if 
there is a vector field $\tau\in\mathcal{T}(\mathcal{P}(f))$ guided by $\zeta$.\label{prop:BE}
\end{proposition}

\subsection{A transversality of periodic critical orbit relations}

Let $f\in\rat_d$ be 
postcritically finite and hyperbolic of disjoint type, and
let $c_1,\ldots,c_{2d-2}$ be $2d-2$ distinct critical points of $f$.
For any $1\leq i\leq 2d-2$, 
there is $p_i\in\N^*$ such that $c_i\in\Fix^*(f^{p_i})$, and
there is an open neighborhood $U$ of $f$ in $\rat_d$ small enough so that
$c_1,\ldots,c_{2d-2}$ can be followed holomorphically on $U$,
that is for any $1\leq i\leq 2d-2$, 
there is a holomorphic map $c_i:U\to\P^1$ such that 
$c_i(f)=c_i$ and that $c_i(g)\in\mathrm{Crit}(g)$ for every $g\in U$.

We can choose an atlas of $\P^1$ such that there is an affine chart of $\P^1$
containing $c_1(g),\ldots,c_{2d-2}(g)$ for every $g\in U$, 
and define a map $\mathcal{V}:U\longrightarrow\C^{2d-2}$ by
\begin{gather*}
\mathcal{V}(g):=
\bigl(g^{p_1}(c_1(g))-c_1(g),\ldots,g^{p_{2d-2}}(c_{2d-2}(g))-c_{2d-2}(g)\bigr),
\quad g\in U.
\end{gather*}

We will need the following.

\begin{theorem}\label{tm:trans1}
Let $f\in\rat_d$ be 
postcritically finite and hyperbolic of disjoint type.
If $f$ is not conjugate to $z^{\pm 2}$, then
the linear map $D_f\mathcal{V}:T_f\rat_d\to T_{0}\C^{2d-2}$ 
is surjective and $\ker(D_f\mathcal{V})=T_f\mathcal{O}(f)$.
\end{theorem}

Though this result seems folklore, we could not find it in the above form
in the literature. 
We provide here a proof for the sake of completeness,
which is
very much inspired by \cite{buffepstein,epstein2} 
(see also~\cite{favregauthier}).

\subsection{The proof of Theorem~\ref{tm:trans1}}

From now on, we write
$$\dot u:=\displaystyle\left. \frac{du_t}{dt}\right|_{t=0}$$
for any holomorphic map $t\mapsto u_t$ defined on a disk $\D$.

\begin{proof}[Proof of Theorem \ref{tm:trans1}]
Under our assumption, the postcritically finite map $f$ is neither a Latt\`es map, nor conjugate to $z^{\pm 2}$. Let us pick $\zeta\in\ker(D_f\mathcal{V})$, and choose a holomorphic disc $t\mapsto f_t\in\rat_d$ with $f_0=f$ and such that $\dot f=\zeta$. We shall use Proposition \ref{prop:BE} and build a vector field $\tau\in\mathcal{T}(\mathcal{P}(f))$ which is guided by $\zeta$. Then counting dimensions will complete the proof.

For any $n\in\N$ and any $1\leq i\leq 2d-2$,
we set $c_i(t)=c_i(f_t)$,
\begin{center}
$v_{n,i}(t):= f_t^n(c_{i}(t))$,
\end{center}
$c_i := c_{i}(0)$, and $v_{n,i} := v_{n,i}(0)$. 
It is clear that for any $n\ge 0$, we have
\begin{eqnarray}
\dot v_{n+1,i}=\zeta(v_{n,i})+D_{v_{n,i}}f\cdot\dot v_{n,i}~.
\label{eq:iterate}
\end{eqnarray}
We shall deduce the following from this equation.

\begin{lemma}\label{lmdirect}
For all $n,m \in\N^*$, and all $1\leq i,j\leq 2d-2$, if $v_{n,i} = v_{m,j}$, then $\dot v_{n,i} = \dot v_{m,j}$.
\end{lemma}

Taking this result for granted , we continue to define a vector field 
$\tau$ on $\mathcal{P}(f)$ that is guided by $\zeta$.
For any $x\in \mathcal{P}(f)$, 
we set $\tau(x) := \dot v_{n,i}$ for some $1\leq i\leq 2d-2$
and some $n\in\N$ such that $x = v_{n,i}$.
The previous lemma shows that $\tau$ is well-defined at $x$.
It remains to check that $\tau$ is guided by $\zeta$. The equality $\tau (f(c_i)) = \zeta (c_i)$
follows from the definition of $\tau$ and \eqref{eq:iterate}. When $x=v_{n,i}$ is not a critical point, then multiplying \eqref{eq:iterate} 
by $D_xf^{-1}$ gives $\tau = f^*\tau + \eta_\zeta$ at $x$.

When $x = c_i$ is a critical point, since $x$ is a simple critical point, we may choose coordinates $z$ at $c_i$ and $w$ at $f(c_i)$ such that
$w = f_t(z) = z^2+ t (a + O(z)) + O(t^2)$. Since we may follow the critical point for $|t|$ small, we may also suppose that $c_{i}(t) = 0$ for all $t$
so that $f_t(z) = z^2 + t (a + O(z^2)) + O(t^2)$.
We thus obtain 
$\zeta (z) = ( a + O(z)) \frac{\partial}{\partial w}$, and
$\eta_\zeta(z) = ( -\frac{a}{2z} + O(z))\frac{\partial}{\partial z}$.
Observe that in our coordinates we have 
$\tau(c_i) =\dot c_{i} = 0$, and 
$\tau(P(c_i)) = \left. \frac{d}{dt}\right|_{t=0} f_t(c_{i}(t)) =a\,\frac{\partial}{\partial z}$.
We may thus extend $\tau$ locally at $c_i$ and $P(c_i)$ holomorphically 
by setting
$\tau (z) \equiv 0$ and $\tau(w) \equiv a$.
It follows that 
$$
f^*\tau (z) + \eta_\zeta (z) - \tau (z) 
= 
\frac{a}{f'(z)} \,\frac{\partial}{\partial z}+ \left( -\frac{a}{2z} + O(z)\right) \frac{\partial}{\partial z} - 0 = O(z) \frac{\partial}{\partial z}
~.
$$
It follows that 
$f^*\tau + \eta_\zeta = \tau $ at any critical point, which concludes the proof.
\end{proof}

\begin{proof}[Proof of Lemma~\ref{lmdirect}]
First, fix $i$. To simplify notation
we write $v_k,c,p$ instead of $v_{k,i}, c_i,p_i$ respectively.
Recall that $p$ is a multiple of the exact period of $v_0 = c$. 
For any $l \ge 1$, iterating the assertion \eqref{eq:iterate} and using the fact that $D(f^p)$ is vanishing at all points of the cycle and that $v_{k+p} = v_k$ for all $k\ge 0$, 
give
\begin{eqnarray*}
\dot v_{lp} 
\!\!& = &\!\!
\zeta(v_{lp -1}) + D_{v_{lp -1}}f \cdot  \zeta(v_{lp -2}) + \cdots + D_{v_{(l-1)p +1}}\!f^{p-1} \cdot  \zeta(v_{(l-1)p}) + D_{v_{(l-1)p}}f^{p} \cdot  \dot v_{(l-1)p}
\\
\!\!& = &\!\!
\zeta(v_{lp -1}) + D_{v_{lp -1}}f \cdot  \zeta(v_{lp -2}) + \cdots + D_{v_{(l-1)p+1}}f \cdot  \zeta(v_{(l-1)p}) 
\\
\!\!& = &\!\!
\zeta(v_{p -1}) + D_{v_{p -1}}f\cdot  \zeta(v_{p -2}) + \cdots + D_{v_{1}}f \cdot  \zeta(v_0) 
= 
\dot v_{p} 
\end{eqnarray*}
Since $\zeta\in\ker(D_f\mathcal{V})$, we also have $\dot v_{0}-\dot v_{p} = \dot c-\dot v_{p}=D_f\mathcal{V}\cdot \zeta=0~,$ whence $\dot v_{lp} = \dot v_{0}$ for all $l\ge 1$ .
Again by~\eqref{eq:iterate} we get 
$$\dot v_{lp+1} = \zeta (v_{lq}) + D_{v_{lp}} f\cdot v_{lp} = 
\zeta (v_{0}) + D_{v_{0}} f \cdot v_{0} = 
\dot v_{1}~.$$
An immediate induction on $k\ge 0$ then proves $\dot v_{lp + k} = \dot v_{k}$ for all $l\ge 0$.
This proves the lemma whenever $i=j$. 

We now consider the case when $v_{k,i}=c_j$ for some $j\neq i$ and some $k\ge1$. Observe that, up to taking the exact period instead of $p_i$ and $p_j$, we may assume $p \pe p_i = p_j$.
Since $f$ has simple critical points, we may assume that $1\le k\le p-1$, and hence $k$ is uniquely determined. 
Using \eqref{eq:iterate}, we find
$$\dot v_{k+1,i}= \zeta(v_{k,i})+D_{v_{k,i}}f\cdot \dot v_{k,i}=\zeta(c_j) = \dot v_{1,j}~.$$
Again by \eqref{eq:iterate} and an easy induction, it follows that
$\dot v_{k+m,i}= \dot v_{m,j}$ for all $m\ge 1$.

Finally, suppose $v_{n,i} = v_{m,j}$. Up to permuting $i$ and $j$ we may assume that
$n = k+ m  + pl$ for some $l\ge 0$, and we find
\[\dot v_{n,i}= 
\dot v_{k+m+pl,i}=
\dot v_{k+m+p,i}=\dot v_{m+p,j} = 
\dot v_{m,j}~.\]
This concludes the proof.
\end{proof}

\subsection{Application to the space $\rat_d^0$}

Denote by $\rat_d^0$ the space of degree $d$ rational maps on $\P^1$
fixing $0$, $1$ and $\infty$. To be more precise, let us parametrize $\rat_d$ by
\[f([z:t]) = \left[\sum_{i=0}^d a_i z^it^{d-i}:\sum_{i=0}^d b_iz^it^{d-i}\right], \ [z:t]\in\P^1,\]
with $[a_d:\cdots:a_0:b_d:\cdots:b_0] \in \P^{2d+1}\setminus\{\textup{Res}=0\}$. The space $\rat_d^0$ is then determined by the equations 
$b_d =0, a_0 = 0,\sum_i a_i=\sum_j b_j$ 
and is thus clearly a smooth subvariety of $\rat_d$ of pure dimension $2d-2$.

\begin{lemma}\label{lm:transorbit}
The complex submanifolds 
$\rat_d^0$ and $\mathcal{O}(f)$ in $\rat_d$
intersect transversely at 
any $f\in\rat_d^0$.
\end{lemma}

\begin{proof}
Let $\zeta\in T_f \mathcal{O}(f)\cap T_f\rat_d^0$. Then there exists a holomorphic germ $m_t\in\textup{Aut}(\P^1)$ centered at $m_0=\id$ such that $f_t=m_t^{-1}\circ f\circ m_t$ and $\zeta=\dot f$. Moreover, since $(f_t)_t$ is tangent to $\rat_d^0$, there are fixed points of $f_t$ of the form $x_t=O(t^2)$, $y_t=1+O(t^2)$ and $z_t$ with $z_t=1/O(t^2)$.

Writing $m_t(z)=(a_tz+b_t)/(c_tz+d_t)$ with $a_td_t-c_tb_t=1$, we get $x_t=-b_t/a_t$, $y_t=(b_t-d_t)/(c_t-a_t)$ and $z_t=-d_t/c_t$. As $m_0=\id$, we have $a_t=1+\alpha t+O(t^2)$, $b_t=\beta t+O(t^2)$, $c_t=\gamma t+O(t^2)$ and $d_t=1+\delta t+O(t^2)$.

We thus get $-\beta t +O(t^2)=O(t^2)$, $-\gamma t+O(t^2)=O(t^2)$, $1+(\alpha+\delta) t+O(t^2)=1$ and $1+(\delta -\beta +\gamma-\alpha)t+O(t^2)=1+O(t^2)$, whence $\alpha=\beta=\gamma=\delta=0$. As a consequence, $m_t=\id+O(t^2)$ and $m_t^{-1}=\id+O(t^2)$. Finally, differentiating $f_t=m_t^{-1}\circ f\circ m_t$ with respect to $t$ and evaluating at $t=0$ gives
$$\zeta=Df\cdot \dot m=Df\cdot 0=0~.$$
This proves $T_f\mathcal{O}(f) \cap T_f\rat_d^0 = (0)$.
\end{proof}

As above, we pick $f\in\rat_d^0$ which is hyperbolic and postcritically finite with simple critical points $c_1,\ldots,c_{2d-2}$. We also assume that for $1\leq i\leq 2d-2$, there exists $p_i\geq1$ such that $c_i\in\Fix^*(f^{p_i})$. Let $U\subset\rat_d^0$ be a neighborhood of $f$ in which $c_i$ can be followed holomorphically as a critical point $c_i(g)$ of $g$ for all $i$. We can choose an atlas of $\P^1$ such that
there is an affine chart containing $c_1(g),\ldots,c_{2d-2}(g)$
 for every $g\in U$. 
 We let
\[\mathcal{V}(g):=\left(g^{p_1}(c_1(g))-c_1(g),\ldots,g^{p_{2d-2}}(c_{2d-2}(g))-c_{2d-2}(g)\right), \ g\in U.\]
From Theorem~\ref{tm:trans1} and Lemma~\ref{lm:transorbit}, we  directly get the following.

\begin{corollary}\label{tm:transversality-ratd0}
Pick any postcritically finite and hyperbolic $f\in\rat_d^0$
of disjoint type. The map
$\mathcal{V}_f:U\to\C^{2d-2}$ is a local biholomorphism at $f$.
\end{corollary}

\section{Counting the centers of hyperbolic components of disjoint type in $\mathcal{M}_d$}

\subsection{The marked moduli spaces $\mathcal{M}_d^\fm$ and $\mathcal{M}_d^\tm$}\label{sec:moduli}

We follow closely ~\cite[Section 9]{Milnor-hyperbolic}.

A {\em fixed marked} degree $d$ rational map on $\P^1$ 
is a $(d+2)$-tuple $(f,x_1,\ldots,x_{d+1})$ where $f\in\rat_d$, and 
$(x_1,\ldots,x_{d+1})\in(\P^1)^{d+1}$ 
is a $(d+1)$-tuple of all the fixed points of $f$, taking into account
their multiplicities. The space $\rat_d^\fm$
of all degree $d$ fixed marked rational maps is a smooth quasi-projective subvariety 
of dimension $2d+1$ in $\rat_d\times(\P^1)^{d+1}$.
The conjugacy action of $\mathrm{PSL}_2(\C)$ on $\rat_d$ canonically
extends to $\rat_d^\fm$,
and the moduli space $\mathcal{M}_d^\fm$ of degree $d$ is
the orbit space of $\PSL_2(\C)$ in $\rat_d^\fm$.

The space $\mathcal{M}_d^\fm$ is an irreducible quasi-projective variety of dimension $2d-2$ and its singular points are contained in the subvariety of $\mathcal{M}_d^\fm$ consisting of all classes $[(f,x_1,\ldots,x_{d+1})]$ such that $\#\{x_1,\ldots,x_{d+1}\}\le 2$.
In particular, $\mathcal{M}_d^\fm$  is smooth at any class $[(f,x_1,\ldots,x_{d+1})]$ such that $f$ is hyperbolic.
Let $\rat_d^{0,\fm}$ be the space of all fixed marked degree $d$ rational maps $(f,x_1,\ldots,x_{d+1})$ such that
$x_1=0$, $x_2=1$ and $x_3=\infty$. 
It is clear that $\rat_d^{0,\fm}$ is smooth and quasi-projective. 
Any smooth point in $\mathcal{M}_d^\fm$ admits a unique representative in $\rat_d^{0,\fm}$ since an element in $\PSL_2(\C)$ fixes three distinct 
points if and only if it is the identity.

A \emph{totally marked} degree $d$ rational map on $\P^1$ is a $3d$-tuple 
$(f,x_1,\ldots,x_{d+1},c_1,\ldots,c_{2d-2})$ 
where $(f,x_1,\ldots,x_{d+1})\in\rat_d^\fm$,
and $(c_1,\ldots,c_{2d-2})\in(\P^1)^{2d-2}$ 
is a $(2d-2)$-tuples of all critical points of $f$,
taking into account their multiplicities. 
The space $\rat_d^\tm$ of all degree $d$ totally marked rational maps 
is a smooth quasi-projective subvariety of dimension $2d+1$
in $\rat_d^\fm\times(\P^1)^{2d-2}\subset\rat_d\times(\P^1)^{3d-1}$. 
The conjugacy action of $\mathrm{PSL}_2(\C)$ on $\rat_d$ canonically
extends to $\rat_d^\tm$, and the moduli space $\mathcal{M}_d^\tm$ is
the orbit space of $\PSL_2(\C)$ in $\rat_d^\tm$.

The space $\mathcal{M}_d^\tm$ is an irreducible quasi-projective variety of dimension $2d-2$ and its singular points are contained in the subvariety of $\mathcal{M}_d^\tm$ consisting of all classes 
$[(f,x_1,\ldots,x_{d+1},c_1,\ldots,c_{2d-2})]$ such that
either $\#\{x_1,\ldots,x_{d+1}\}\le 2$ or $\#\{c_1,\ldots,c_{2d-2}\}<2d-2$.
In particular, $\mathcal{M}_d^\tm$ is smooth 
at any class $[(f,x_1,\ldots,x_{d+1},c_1,\ldots,c_{2d-2})]$ such that
$f$ is hyperbolic \emph{of disjoint type}.
Let $\rat_d^{0,\tm}$ be the space of all fixed marked 
degree $d$ rational maps $(f,x_1,\ldots,x_{d+1},c_1,\ldots,c_{2d-2})$ 
such that $x_1=0$, $x_2=1$ and $x_3=\infty$. 
It is clear that $\rat_d^{0,\tm}$ is smooth and quasi-projective. 
Any smooth point in $\mathcal{M}_d^\tm$ 
admits a unique representative in $\rat_d^{0,\tm}$ 
since an element in $\PSL_2(\C)$ fixes three distinct 
points if and only if it is the identity.

We note that the same construction of $\mathcal{L}$, $\mu_\bif$, $\Per_n(w)$, and $T_{\un}^p(\underline{\rho})$
as on $\cM_d$ works on $\cM_d^\fm$ and $\cM_d^\tm$.

\subsection{Parametrizing hyperbolic components of $\mathcal{M}_d^\fm$ of disjoint type}\label{parametrizing_hyperbolic_components}

A hyperbolic component $\Omega$ in $\mathcal{M}_d^\fm$ is simply connected 
and contains a {\em center}, which is by definition
the unique point $[(f,x_1,\ldots,x_{d+1})]$ such that $\#\mathcal{P}(f)<\infty$,
if $\mathcal{J}_f$ is connected for any $[(f,x_1,\ldots,x_{d+1})]\in\Omega$ 
by \cite[Theorem~9.3]{Milnor-hyperbolic}.

For any $\un=(n_1,\ldots,n_{2d-2})\in(\N^*)^{2d-2}$ and any hyperbolic component $\Omega$ in $\mathcal{M}_d^\fm$
of type $\un$,
$\J_f$ is connected for any $[(f,x_1,\ldots,x_{d+1})]\in\Omega$, since all Fatou components of $f$ are then topological disks by \cite[Proposition p.231]{Prz-connectedness}. In particular $\Omega$ has a center. 
We will also use the following in the sequel.

\begin{lemma}\label{lm:markingcrit}
For any $\un=(n_1,\ldots,n_{2d-2})\in(\N^*)^{2d-2}$,
any hyperbolic component $\Omega$ in $\mathcal{M}_d^\fm$ of type $\un$
is simply connected and the critical points are marked throughout $\Omega$. More precisely, there are a holomorphic family $(f_\lambda)_{\lambda\in\Omega}$ of degree $d$ rational maps and
markings $x_1,\ldots,x_{d+1},c_1,\ldots,c_{2d-2}:\Omega\to\P^1$
of all fixed points and all critical points of $(f_\lambda)_{\lambda\in\Omega}$,
such that
$\lambda=[(f_\lambda,x_1(\lambda),\ldots,x_{d+1}(\lambda))]$
for any $\lambda\in\Omega$.
\end{lemma}

\begin{proof}
We have already seen that $\Omega$ is simply connected.
Let $\tau:\mathcal{M}_d^\tm\to\mathcal{M}_d^\fm$
be the natural finite branched cover.
For any component $\tilde{\Omega}$ of $\tau^{-1}(\Omega)$, 
$\tilde{\Omega}$ is a hyperbolic component of disjoint type in $\mathcal{M}_d^\tm$
and $\tau|_{\tilde{\Omega}}:\tilde{\Omega}\to\Omega$ is an unramified cover, so
is a biholomorphism, since $\Omega$ is simply connected.

Let $p:\rat_d^{0,\tm}\to\mathcal{M}_d^\tm$ 
be the restriction of the canonical projection
$\rat_d^\tm\to\mathcal{M}_d^\tm$ to $\rat_d^{0,\tm}$,
so that $p$ is a biholomorphism to its image, which contains $\tilde{\Omega}$,
and set $\hat{\Omega}:=p^{-1}(\tilde{\Omega})$. Then
$p|_{\hat{\Omega}}:\hat{\Omega}\to\tilde{\Omega}$ is also a biholomorphism.
We can define the holomorphic maps $f_\cdot,x_1,\ldots,x_{d+1},c_1,\ldots,c_{2d-2}$
by $(f_\lambda,x_1(\lambda),\ldots,x_{d+1}(\lambda),c_1(\lambda),\ldots,c_{2d-2}(\lambda)):=(p\circ\tau)|_{\hat{\Omega}}^{-1}(\lambda)$, $\lambda\in\Omega$.
\end{proof}

Pick any hyperbolic component $\Omega$ in $\mathcal{M}_d^\fm$ of 
type $\un=(n_1,\ldots,n_{2d-2})\in(\N^*)^{2d-2}$,
and choose a holomorphic family $(f_\lambda)_{\lambda\in\Omega}$ 
of degree $d$ rational maps and
markings $x_1,\ldots,x_{d+1},c_1,\ldots,c_{2d-2}:\Omega\to\P^1$
of all fixed points and critical points of $(f_\lambda)_{\lambda\in\Omega}$,
taking into account their multiplicities such that 
$\lambda=[(f_\lambda,x_1(\lambda),\ldots,x_{d+1}(\lambda))]$
by Lemma \ref{lm:markingcrit}.
For any $1\le i\le 2d-2$ and any $\lambda\in\Omega$,
let $w_i(\lambda)\in\D$ be the multiplier of the attracting cycle $\mathcal{C}_i(\lambda)$ 
of $f_\lambda$ whose immediate attractive basin contains $c_i(\lambda)$,
and assume that $\mathcal{C}_i$ is of exact period $n_i$.
The \emph{multiplier map} $\mathcal{W}_\Omega:\Omega\to\D^{2d-2}$ on $\Omega$ is defined by
\[
\mathcal{W}_\Omega(\lambda):=(w_1(\lambda),\ldots,w_{2d-2}(\lambda)),\quad\lambda\in\Omega.
\]
Let $\lambda_\Omega$ be the center in $\Omega$.
Noting also that $\#\mathcal{P}(f_\lambda)<\infty$
for any $\lambda\in\mathcal{W}^{-1}\{0\}$, 
we have $\mathcal{W}^{-1}\{0\}=\{\lambda_\Omega\}$.

\begin{theorem}\label{prouve_le_bebe}
The map $\mathcal{W}_\Omega:\Omega\to\D^{2d-2}$ is a biholomorphism.
\end{theorem}

\begin{proof}
Write $\mathcal{W}$ for $\mathcal{W}_\Omega$. 
First, we prove that $\mathcal{W}$ is surjective and finite. According to \cite[\S 3 page 179]{grauertremmert}, this implies that $\mathcal{W}$ is a finite 
and possibly ramified covering. Next, we show that 
$\mathcal{W}$ is locally invertible at $\lambda_\Omega$. Then 
$\mathcal{W}$ has degree $1$, i.e., is a biholomorphism.

Let us first prove that $\mathcal{W}$ is surjective. We proceed using the classical surgery argument: for any $0<\varepsilon<1$, we construct a continuous map $\sigma : \D (0, 1 - \varepsilon) ^{2d-2} \to \Omega$ such that $ \mathcal{W} \circ \sigma = \id$. We sketch the construction referring to~\cite[Theorem~VIII.2.1]{carleson} or \cite{douadybourbaki} for detail.

Choose $\lambda\in\Omega$, let $(f,x_1,\ldots,x_{d+1})$ be a representative of $\lambda$ and let $U_{1,i}, \ldots , U_{n_i,i}$ be the immediate basin of attraction of these attracting cycles of $f$ indexed by $1\leq i \leq 2d-2$ such that the unique critical point attracted to this cycle belongs to $U_{1,i}$. Since it is a simple critical point, $U_{1,i}$ is simply connected and there exists a conformal map $\varphi_i : U_{1,i} \to \D$
such that 
\[\varphi_i \circ  f^{n_i} \circ \varphi_i^{-1} (\xi) = \xi\cdot\frac{\xi +w_i }{1 + \bar{w_i} \xi}, \text{ for any } |\xi| < 1~,\]
where $w_i$ is the multiplier of the $i$-th cycle in our ordering.
For any $\rho = (\rho_1,\ldots,\rho_{2d-2}) \in \D (0, 1 - \varepsilon) ^{2d-2}$, we can define a continuous
map $\tilde{f}_\rho$ by setting 
$\tilde{f}_\rho=f$ outside the union of all $U_{j,i}$, 
and such that 
$\varphi_i \circ  \tilde{f}_\rho^{n_i}\circ \varphi_i^{-1} (\xi) = \xi\cdot \frac{\xi +\rho_i }{1 + \bar{\rho_i} \xi}$
 on a fixed disk $|\xi| < 1 -r$ containing the critical point of  the latter Blashke product. Notice that $(\tilde{f}_\rho)_{\rho}$ is a continuous family of quasi-regular maps of $\P^1$.
We now solve the Beltrami equation for the unique Beltrami form which is $0$ on the complement  the $U_{j,i}$'s and invariant under $\tilde{f}_\rho$: there is a continuous family of quasiconformal homeomorphism $\psi_\rho : \P^1 \to \P^1$ such that $f_\rho := \psi_\rho\circ \tilde{f}_\rho \circ \psi_\rho^{-1}$ is holomorphic and depends again continuously on $\rho$.
The map $\sigma(\rho):=[(\psi_\rho\circ \tilde{f}_\rho \circ \psi_\rho^{-1},
\psi_\rho(x_1(\lambda)),\ldots,\psi_\rho(x_{d+1}(\lambda)))]$ has
the desired properties.

Let us show that $\mathcal{W}$ is finite, i.e.,
$\#\mathcal{W}^{-1}(\underline{w})<\infty$
for any $\underline{w}=(w_1,\ldots,w_{2d-2})\in\D^{2d-2}$.
Let $p:\rat_d^{0,\fm}\to\mathcal{M}_d^\fm$ be the restriction
of the canonical projection $\rat_d^\fm\to\mathcal{M}_d^\fm$, so that
$p$ is a biholomorphism to its image, which contains $\Omega$, and 
we set $\tilde{\Omega}:=p^{-1}(\Omega)$ and 
$W:=\mathcal{W}\circ p:\tilde{\Omega}\to\D^{2d-2}$.
Suppose to the contrary that for some $\underline{w}\in\D^{2d-2}$, 
$\#\mathcal{W}^{-1}(\underline{w})=\infty$,
or equivalently, $\#W^{-1}\{\underline{w}\}=\infty$.
Then there is a quasi-projective subvariety $\Lambda$ of dimension $>0$ in 
$\bigcap_i\Per_{n_i}(w_i)\cap\rat_d^0\cap\tilde{\Omega}$ such that
the holomorphic family $(f_{p(\lambda)})_{\lambda\in\Lambda}$ on $\Lambda$
has no bifurcation since all its marked critical points $p^*c_1,\ldots,p^*c_{2d-2}$ 
are persistently attracted to attracting cycles.
Since $\mathcal{J}_{f_{p(\lambda)}}\neq\P^1$ for some $\lambda\in\Lambda$, 
$(f_{p(\lambda)})_{\lambda\in\Lambda}$ is not a family of Latt\`es maps.
Hence by \cite[Theorem 2.2]{McMullen4}, 
$(f_{p(\lambda)})_{\lambda\in\Lambda}$ is trivial, but then
the quasi-projective variety $p(\Lambda)$ in $\mathcal{M}_d^\fm$
is still of dimension $>0$ and must be mapped to a singleton in $\mathcal{M}_d$,
so of dimension $0$,
by the natural finite branched covering $\mathcal{M}_d^\fm\to\mathcal{M}_d$.
This is a contradiction.

Let us see the locally invertibility of $\mathcal{W}$ 
at $\lambda_\Omega$. Since $p$ is a biholomorphism on 
$\tilde{\Omega}$, the map $\mathcal{W}$ is locally invertible at $\lambda_\Omega$ 
if and only if $W$ is locally invertible at 
$a:=(f(\lambda_\Omega),x_1(\lambda_\Omega),\ldots,x_{d+1}(\lambda_\Omega))
\in\rat_d^{0,\fm}$.
The conclusion follows from Lemma \ref{th:inverse} below 
by the inverse function theorem.
\end{proof}

\begin{lemma}\label{th:inverse}
The linear map $D_aW$ is invertible.
\end{lemma}

\begin{proof}
Let $\pi:\rat_d^{0,\fm}\to\rat_d^0\subset\rat_d$ be the natural branched cover,
and set $\hat{\Omega} :=\pi(\tilde{\Omega})$, which is 
the hyperbolic component in
$\rat_d^0$ containing $f$ with $\lambda_\Omega=[(f,x_1,\ldots,x_{d+1})]$. 
Let us remark that, since $f$ has 
only simple fixed points, the restriction $\pi_{\tilde{\Omega}}:\tilde{\Omega}\to\hat{\Omega}$ of $\pi$ to $\tilde{\Omega}$ is an (unramified) cover. 
We can choose an atlas of $\P^1$ such that there is an affine chart of $\P^1$
containing $\{c_1(g),\ldots,c_{2d-2}(g)\}$ for every $g\in\hat{\Omega}$, 
and define $V:\hat{\Omega}\to\C^{2d-2}$ by
\[
V(g):=\left(g^{n_1}(c_1(g))-c_1(g),\ldots,g^{n_{2d-2}}(c_{2d-2}(g))-c_{2d-2}(g)\right),\quad g\in\hat{\Omega}.
\]
According to Corollary \ref{tm:transversality-ratd0}, 
we have $\ker(D_fV)=\{0\}$.
Beware that $\hat{W}:=W\circ\pi^{-1}$ is  
 a holomorphic map from an open neighborhood of $f$
in $\hat{\Omega}$ to $\C^{2d-2}$, so it is sufficient to prove that 
$\ker(D_f\hat{W})\subset\ker(D_fV)$. 

Let $v\in T_f\rat_d^0$, and pick a holomorphic disk $(f_t)_{t\in\D}$ in $\rat_d^0$
such that $f_0=f$ and $\dot f=v$. For any $t\in\D$ and any $1\le i\le 2d-2$,
set $w_i(t):=w_i(f_t)$, $c_i(t):=c_i(f_t)$, $\hat{W}(t)=\hat{W}(f_t)$ and $V(t)=V(f_t)$. For any $t\in\D$ and any $1\le i\le 2d-2$, 
let $\mathcal{C}_i(t)$ be the attracting cycle 
of $f_t$ whose immediate attractive basin contains $c_i(t)$, so that
there is a holomorphic function $z_i$ on $\D$ such that 
$z_i(t)\in\mathcal{C}_i(t)$ for any $t\in\D$ 
(so $w_i(t)=(f_t^{n_i})'(z_i(t))$) and that $z_i(0)=c_i(0)$.
Then 
for any $1\le i\le 2d-2$, we find
\begin{align*}
\dot w_i = &\left.\frac{d (f_t^{n_i})'}{d t}\right|_{t=0}(z_i(0))+(f^{n_i})''(z_i(0))\cdot\dot z_i\\
=&\left.\frac{\partial (f_t^{n_i})'}{\partial t}\right|_{t=0}(c_i(0))+(f^{n_i})''(c_i(0))\cdot\dot z_i,
\end{align*}
and since $(f_t^{n_i})'(c_i(t))=0$ for any $t\in\D$, we also have
\begin{gather*}
0 = \left.\frac{\partial (f_t^{n_i})'}{\partial t}\right|_{t=0}(c_i(0))+(f^{n_i})''(c_i(0))\cdot\dot c_i(0).
\end{gather*}
Hence for any $1\le i\le 2d-2$,
\begin{gather*}
\dot w_i  = (f^{n_i})''(c_i(0))\cdot\left(\dot z_i-\dot c_i\right),
\end{gather*}
and we also note that 
$(f^{n_i})''(c_i(0))\neq0$ since $f$ is hyperbolic of disjoint type. 
If in addition $v\in\ker(D_f\hat{W})$, then 
for any $1\le i\le 2d-2$, $\dot{w}_i=0$ (and by definition $z_i(0)=c_i(0)$) hence
we have $z_i(t)-c_i(t)=O(t^2)$. 
For any $1\le i\le 2d-2$,
the $i$-th component of $V(t)$ is
\begin{align*}
f_t^{n_i}(c_i(t))-c_i(t) 
= & f_t^{n_i}(z_i(t))+(f_t^{n_i})'(z_i(t))(c_i(t)-z_i(t))+O((c_i(t)-z_i(t))^2)-c_i(t)\\
= &(1-w_i(t))(z_i(t)-c_i(t))+O(t^4)=O(t^2)
\end{align*}
as $t\to 0$, so that $v\in\ker(D_fV)$.
\end{proof}

\subsection{Counting hyperbolic components : the mass of $\mu_\bif$ in $\mathcal{M}_d$}\label{sec:counting}
We now prove Theorem~\ref{tm:counting} and Corollary~\ref{tm:massM2}.
To avoid confusions, for any $\underline{\rho}$ and any $\un$, denote $T_\un^{2d-2}(\underline{\rho})$ and $\mu_\bif$ on $\cM_d^\fm$
by $T_\un^{2d-2,\fm}(\underline{\rho})$ and $\mu_\bif^\fm$, respectively.

Observe first the following.

\begin{lemma}\label{5.5}
Fix any $\underline{\rho}\in ]0,1[^{2d-2}$ and any
$\un\in(\N^*)^{2d-2}$. Then
\begin{enumerate}
\item the measure $T_{\un}^{2d-2,\fm}(\underline{\rho})$ has full mass on 
the union of all hyperbolic components in $\mathcal{M}_d^\fm$
of type $\un$, and for any such a component $\Omega^\fm$,
\[
(T_{\un}^{2d-2,\fm}(\underline{\rho}))(\Omega^\fm)=\frac{\# \mathrm{Stab}(\un,\underline{\rho})}{d_{|\un|}}.
\]
\item the measure $T_{\un}^{2d-2}(\underline{\rho})$ has full mass on 
the union of all hyperbolic components in $\mathcal{M}_d$
of type $\un$, and for any such a component $\Omega$,
\[
(T_{\un}^{2d-2}(\underline{\rho}))(\Omega)=\frac{\# \mathrm{Stab}(\un,\underline{\rho})}{d_{|\un|}}.
\]
\end{enumerate}
\end{lemma}

\begin{proof} Consider the case $\mathcal{M}_d^\fm$ first.
Pick $\underline{\rho}\in ]0,1[^{2d-2}$, and observe that 
 \[T_{\un}^{2d-2,\fm}(\underline{\rho})=\frac{1}{(2\pi)^{2d-2}d_{|\un|}}\int_{[0,2\pi]^{2d-2}, \forall i\neq j, \ \theta_i\neq \theta_j }\bigwedge_{j=1}^{2d-2}[\Per_{n_j}(\rho_je^{i\theta_j})]\mathrm{d}\theta_1\cdots\mathrm{d}\theta_{2d-2},\]
as finite measures on $\mathcal{M}_d^\fm$, since we only remove a set  of Lebesgue measure zero in $[0,2\pi]^{2d-2}$.
Hence $T_\un^{2d-2,\fm}(\underline{\rho})$-almost every point has $2d-2$ distinct attracting cycles.
For the second part, let $\Omega$ be a hyperbolic component in $\mathcal{M}_d^\fm$ 
of type $\un=(n_1,\ldots,n_{2d-2})$. 
By Theorem~\ref{prouve_le_bebe}, we know that
the multiplier map $\mathcal{W}=(\mathcal{W}_1,\dots, \mathcal{W}_{2d-2})
:\Omega\rightarrow\D^{2d-2}$ is a biholomorphism. In particular, the intersection $\bigcap_{j=0}^{2d-2} \Per_{n_j}(0)$ is smooth and transverse hence for generic $\uw \in \C^{2d-2}$  the intersection $\bigcap_{j=0}^{2d-2} \Per_{n_j}(w_i)$ is still smooth and transverse. This implies
\begin{equation}
d_{|\un|} T_{\un}^{2d-2,\fm}(\underline{\rho})
=\sum_{\sigma \in \mathrm{Stab}(\un,\underline{\rho}) } \bigwedge_{i=1}^{2d-2} 
dd^c \log\max\{|\mathcal{W}_{\sigma(i)}|,\rho_{\sigma(i)}\} \label{pourallervite}
\end{equation}
on $\Omega$,
which has mass $\#\mathrm{Stab}(\un, \underline{\rho})$ on $\Omega$.  This concludes the proof for the case of $\mathcal{M}_d^\fm$. \\

Let $\pi_\fm:\mathcal{M}_d^\fm\to\mathcal{M}_d$
be the natural finite branched cover of degree $(d+1)!$. Observe that it branches only over $\Per_1(1)$ hence not over hyperbolic component of disjoint type. Let $\Omega\subset \mathcal{M}_d$ be a component of type $\un$ and let $\tilde{\Omega}$ be a connected component of $\pi_\fm^{-1}(\Omega)$. Then the restriction ${\pi_\fm}|_{\tilde{\Omega}}:\tilde{\Omega}\to\Omega$ is an unbranched cover. Since multipliers do not depend on the marking of critical points, the multiplier map of $\tilde{\Omega}$ induces a biholomorphism $\mathcal{W}:\Omega\to\D^{2d-2}$.
We now observe that $T_{\un}^{2d-2,\fm}(\underline{\rho})=(\pi_\fm)^*(T_{\un}^{2d-2}(\underline{\rho}))$ and the conclusion follows as above.
\end{proof}

Recall that we denoted by $\rat_d^{0,\fm}$ the space of fixed marked degree $d$ rational maps fixing $0$, $1$ and $\infty$.
Let 
\begin{gather*}
p:\rat_d^{0,\fm}\to\mathcal{M}_d 
\end{gather*}
be the canonical projection. Then the image of $p$ contains $\mathcal{M}_d\setminus\Per_1(1)$ and the restriction
\[\tilde{p}:=p|_{\rat_d^{0,\fm}\setminus\Per_1(1)}\]
of $p$ to $\rat_d^{0,\fm}\setminus\Per_1(1)$ is a finite branched cover to $\mathcal{M}_d\setminus\Per_1(1)$. Moreover, $p$ can only branch at parameters $(f,x_1,\ldots,x_{d+1})$ where at least two of $x_1,\ldots,x_{d+1}$ coincide. In particular, $\tilde{p}$ does not ramify.

For a while, we will work on $\rat_d^{0,\fm}$.
For any $\underline{\rho}\in ]0,1[^{2d-2}$ and any
$\un\in(\N^*)^{2d-2}$,
let $\tilde T_{\un}^{2d-2}(\underline{\rho})$ and $\tilde \mu_\bif$ be the corresponding measures defined on the family $\rat_d^{0,\fm}$. Note that by construction,
\[
\tilde T_{\un}^{2d-2}(\underline{\rho})
=(\tilde{p})^*(T_{\un}^{2d-2}(\underline{\rho}))\quad\text{and}\quad \tilde \mu_\bif=(\tilde{p})^*(\mu_\bif).
\]
Since $\rat_d^{0,\fm}$ is an affine variety, we can assume $\rat_d^{0,\fm}\subset \C^N$ for some $N$.
Consider the function $\log^+ |Z|$, defined on $ \C^N$, and let $\varphi : \rat_d^{0,\fm} \to \R $ be its restriction to  $ \rat_d^{0,\fm}$.  The function $\varphi$ is psh, continuous, non-negative and $dd^c \varphi$ has finite mass in $\rat_d^{0,\fm}$. We have the lemma:

\begin{lemma}\label{tobedonebythomas} 
There exist constants $C_1,C_2>0$ that depend only on $d$ such that, for any compact subset $K$ of $\rat_d^{0,\fm}$, if $C(K)$ is the constant in Theorem~\ref{tm:vitessemoyennesalgebraic}, then we have the following inequality:
\[ C(K) \leq C_1\cdot \|\varphi\|_{\infty,K}+C_2.\]
\end{lemma}
\begin{proof}
We follow closely the proof of \cite[Proposition~3.1]{DeMarco-intersection} (see also~\cite[Proposition~4.4]{Favre-degeneration}) and adapt it to the present situation. Since $H_1(\rat_d,\R)=0$, by~\cite[Lemma~4.9]{BB1}, there exists a family of non-degenerate homogeneous polynomial lifts to $\C^2$ of the family $\rat_d$. We thus may choose a family $F$ of non-degenerate homogeneous polynomial lifts to $\C^2$ of the family $\rat_d^{0,\fm}$. Let $V:=\rat_d^{0,\fm}$. We may regard this family $F=(F_1,F_2)$ as a homogeneous non-degenerate polynomial maps with coefficients in the ring $\C[V]$. Note that $\mathrm{Res}(F)\in\C[V]$ and, in particular, $\left|\log|\mathrm{Res}(F)|\right|\leq \alpha\varphi(f)+\beta$ for some constants $\alpha,\beta\geq0$ independent of $f\in V$.
	
We now want to prove that there exists $m\geq1$ and $C>0$ such that
\[\frac{1}{C}e^{-m\varphi(f)}\leq\frac{\|F(x,y)\|^2}{\|(x,y)\|^{2d}}\leq C e^{m\varphi(f)},\]
for all $f\in\rat_d^{0,\fm}$ and all $(x,y)\in\C^2\setminus\{0\}$.
We work with the norm $\|(x,y)\|=\max\{|x|,|y|\}$ on $\C^2$. The upper bound follows easily from the fact that $F_1,F_2\in\C[V][x,y]$ and from the triangle inequality.
As $F_1$ and $F_2$ are homogeneous, it is sufficient to verify the lower bound whenever $\|(x,y)\|=1$ with either $x=1$ or $y=1$.

By the item $(c)$ of \cite[Proposition 2.13]{Silverman}, there exists homogeneous polynomials $g_1,g_2,h_1,h_2\in \C[V][x, y]$ such that
\begin{eqnarray}
g_1(x, y)F_1(x, y)+g_2(x, y)F_2(x, y) = \mathrm{Res}(F)x^{2d-1},\label{resultant1}
\end{eqnarray}
and
\begin{eqnarray}
h_1(x, y)F_1(x, y)+h_2(x, y)F_2(x, y) = \mathrm{Res}(F)y^{2d-1}.\label{resultant2}
\end{eqnarray}
Again, since $g_1,g_2,h_1,h_2\in\C[V][x,y]$, we have 
\[\max\left\{|g_1(x,y)|,|g_2(x,y)|,|h_1(x,y)|,|h_2(x,y)|\right\}\leq Ae^{B\varphi(f)}\]
as soon as $\|(x,y)\|\leq1$, for some constants $A,B\geq0$ independent of $f\in V$. When $x=1$, equation~\eqref{resultant1} gives
\[|\mathrm{Res}(F)|\leq 4\max\left\{|g_1(x,y)|,|g_2(x,y)|\right\}\cdot \|F(x,y)\|\leq 4A^{B\varphi(f)}\|F(x,y)\|.\]
We proceed similarly with equation~\eqref{resultant2} when $y=1$ and the conclusion follows.

Following exactly the proof of Lemma~\ref{lm:atinfinity} gives $C_1,C_2\geq0$ such that
\[\max\left\{\sup_{z\in\P^1}\log f^\#(z),\sup_{z\in\P^1}|g_f(z)|\right\}\leq C_1\varphi(f)+C_2\]
for all $f\in\rat_d^{0,\fm}$. The conclusion follows, taking the sup on $K$.
\end{proof}

Recall that we picked $\un\in(\N^*)^{2d-2}$.
Let $\varepsilon >0$, and set $\underline{\rho}=(1/2,\dots,1/2)$,
so in particular that $\mathrm{Stab}(\un,\underline{\rho})=\mathrm{Stab}(\un)$.
Take some $R>0$ so large that the support of $\tilde T_{\un}^{2d-2}(\underline{\rho})$ is in the intersection $B(0,R)$ between the ball 
of radius $R$ and center $0$ in $ \C^N$ and $\rat_d^{0,\fm}$.
Observe that this is possible since there are at most finitely many of type $\un$ and for a hyperbolic component $\Omega$ of type $\un$,
$\mathcal{W}_\omega^{-1}( \D_{1/2}^{2d-2})\subset\Omega$ is relatively compact in $\rat_d^{0,\fm}$ (Note that, for $d=2$, it is known to be true for the whole component \cite{epstein}). 

Increasing $R$ if necessary, we can assume that 
$\bigl|\|\tilde\mu_\bif\|-\tilde \mu_\bif(B(0,R))\bigr| \leq \varepsilon$. For any $A>0$, we take the test function $\Psi_A$:
\[
\Psi_A:= \frac{1}{A} \min\{\max \left(\varphi, A \right)- 2 A, 0\}.
\]
Then, $\Psi_A$ is DSH on $V$ and continuous and $dd^c \Psi_A=T_A^+-T_A^-$
for some positive closed currents of finite masses, where 
$\|T^\pm_A \|\leq C'/A$ and
for some $C'>0$ depending neither on $A$ nor on $T^\pm_A$
(e.g. \cite[Lemma 2.2.6 ]{DinhSibonysuper}). Then observe that $\Psi_A$ is equal to $-1$ in $B(0,e^A)$, and $0$ outside $B(0,e^{2A})$. Applying Theorem~\ref{tm:vitessemoyennesalgebraic} with the control of Lemma~\ref{tobedonebythomas} implies:
\[
\left| \langle\tilde T_{\un}^{2d-2}(\underline{\rho}),\Psi_A   \rangle -   \langle \tilde \mu_\bif,\Psi_A \rangle \right| 
\leq  (1+\log 2) \left(C_1\varphi (e^{2A})+C_2\right) \max_j\left(\frac{\sigma_2(n_j)}{d^{n_j}}\right)\frac{C'}{A}.            
\]
Taking $A=\log R$ so that $\varphi (e^{2A})=2A$, there is a constant $C_d>0$ depending only on $d$, we have
$\left| \| \tilde T_{\un}^{2d-2}(\underline{\rho})\|- \mu_\bif(B(0,R)) \right| \leq  C_d\max_j(\sigma_2(n_j)/d^{n_j})$,
so that,  as $\varepsilon\to0$, 
\[
\left| \| \tilde T_{\un}^{2d-2}(\underline{\rho})\|-\|\tilde \mu_\bif\| \right|
\leq  C_d\max_j\left(\frac{\sigma_2(n_j)}{d^{n_j}}\right).
\]
Let us go back to $\cM_d$. Since the measures $T_{\un}^{2d-2}(\underline{\rho})$ and $\mu_\bif$ (resp. $\tilde T_{\un}^{2d-2}(\underline{\rho})$ and $\tilde \mu_\bif$) give no mass to algebraic subvarieties of $\mathcal{M}_d$ (resp.\ of $\rat_d^{0,\fm}$), we have
\begin{eqnarray*}
\left| \| T_{\un}^{2d-2}(\underline{\rho})\|-   \| \mu_\bif\| \right| & = & \left|\left\langle T_{\un}^{2d-2}(\underline{\rho})- \mu_\bif,\mathbf{1}_{\mathcal{M}_d\setminus\Per_1(1)}\right\rangle\right|\\
& = & \frac{1}{\deg(\tilde{p})}\left\langle T_{\un}^{2d-2}(\underline{\rho})-\mu_\bif,(\tilde{p})_*(\tilde{p})^*\left(\mathbf{1}_{\mathcal{M}_d\setminus\Per_1(1)}\right)\right\rangle\\
& = & \frac{1}{\deg(\tilde{p})}\left\langle \tilde T_{\un}^{2d-2}(\underline{\rho})-\tilde \mu_\bif,\mathbf{1}_{\rat_d^{0,\fm}\setminus\Per_1(1)}\right\rangle\\
& = & \frac{1}{\deg(\tilde{p})}\left| \| \tilde T_{\un}^{2d-2}(\underline{\rho})\|-   \|\tilde \mu_\bif\| \right|, 
\end{eqnarray*}
which together with $\|T_{\un}^{2d-2}(\underline{\rho})\|
=\#\mathrm{Stab}(\un)N(\un)/d_{|\un|}$ by Lemma \ref{5.5}
completes the proof of Theorem~\ref{tm:counting}. \qed

\begin{proof}[Proof of Corollary~\ref{tm:massM2}]
Let us begin with describing \cite[Theorem 1.1]{kiwirees} by Kiwi and Rees. 
Take $m > n \geq 2$, they computed, in the critically marked moduli space $\mathcal{M}_2^\mathrm{cm}$, the number $n_{IV}(n,m)$ of all hyperbolic components $\Omega$ in $\cM_2^\cm$
of type $(n,m)$
such that any $[(f,c_1,c_2)]\in\Omega$
has two distinct attracting cycles of respective exact periods $j,k$ 
with $j|n$ and $k|m$ and their immediate attractive basins 
contain $c_1,c_2$, respectively. They prove
\[
n_{IV}(n,m)=\left(\frac{5}{3}2^{n-3}+\frac{1}{12}-\frac{1}{4}\sum_{q=2}^n\frac{\phi(q)\nu_q(n)}{2^q-1}\right)2^m+\varepsilon_1(n,m),
\]
where $|\varepsilon_1(n,m)|\leq 2^{n}+2^{2 \mathrm{gcd}(n,m)}$ and
\[\left|\nu_q(n)-\frac{2^n}{2(2^q-1)}\right|\leq \frac{1}{2}.\]
For any $m > n \geq 2$, their computation in particular yields
\begin{align*}
n_{IV}(n,m) 
& = \left(\frac{1}{3}-\frac{1}{8}\sum_{q=1}^n\frac{\phi(q)}{(2^q-1)^2}\right)2^{n+m}
+\varepsilon_2(n,m)
\end{align*}
where there exists $C\geq1$ such that $|\varepsilon_2(n,m)|\leq C\cdot 2^m$.
We now note that the natural projection $\pi:\mathcal{M}_2^\mathrm{cm}\rightarrow\mathcal{M}_2$ is of degree $2$ and is unramified over any $[f]\in\mathcal{M}_2$
of disjoint type, and for any $[(f,c_1,c_2)]\in\mathcal{M}^\cm_2$, 
$\pi^{-1}\{[f]\}=\{[f,c,c'],[f,c',c]\}$.
In particular,
\[
n_{IV}(n,n+1)=\sum_{j|n,\, k|(n+1)}N(j,k).
\]
Since we have $N(j,k)\leq C\cdot 2^{j+k}$ by Bézout and $d_n\cdot d_{n+1}=2^{2n+1}+O(2^n)$, the above gives
\[\frac{N(n,n+1)}{d_n\cdot d_{n+1}}=\frac{n_{IV}(n,n+1)}{2^{2n+1}}+o(1)=\frac{1}{3}-\frac{1}{8}\sum_{q=1}^{+\infty}\frac{\phi(q)}{(2^q-1)^2}+o(1), \ \text{as} \ n\to \infty.\]
The conclusion follows from Theorem~\ref{tm:counting}, 
since $\#\mathrm{Stab}((n,n+1))=1$.
\end{proof}

\subsection{Weak genericity of postcritically finite hyperbolic rational maps}\label{sec:genericity}

The moduli space $\mathcal{M}_d$ of degree $d$ rational maps is known to be an irreducible affine variety of dimension $2d-2$ which is defined over $\mathbb{Q}$ (see \cite{silverman-spacerat,Milnor3}); and all non-flexible Latt\`es postcritically finite degree $d$ rational maps are known to be defined over $\bar{\mathbb{Q}}$ (see e.g.~\cite{Silverman}).
These properties are, for the moduli space of critically marked degree $d$ polynomials, the starting point of the work \cite{favregauthier}. The idea developed there is to apply Yuan's equidistribution theorem \cite{yuan}
of points of small height on arithmetic projective varieties.

If we want to adapt this strategy to the present context, we have to complete  three important steps. We will be intentionally vague in the description of the first two steps, since a precise explanation requires more machineries.
\begin{enumerate}
\item We have to define an adelic semi-positive metric on an ample line bundle $L\to\bar{\mathcal{M}}_d$ (for some projective compactification $\bar{\mathcal{M}}_d$ of $\mathcal{M}_d$). We also have to check that for the induced Weil height function $h$ on $\bar{\mathcal{M}}_d(\bar{\mathbb{Q}})$, the set of postcritically finite parameters is exactly the set of points of height $0$.
\item We have to verify that the curvature current $c_1(L)$ of the above metrized line bundle satisfies $c_1(L)^{2d-2}=\alpha\cdot\mu_\bif$ for some positive constant $\alpha>0$.
\item Finally, we have to show that {\em any} sequence $(X_k)$ 
of Galois invariant finite sets of postcritically finite parameters is \emph{weakly generic} in $\cM_d$, that is, for any proper affine subvariety $C$ in $\mathcal{M}_d$ 
defined over $\mathbb{Q}$, 
\begin{gather*}
\mathrm{Card}(X_k\cap C)=o\left(\mathrm{Card}(X_k)\right)
\end{gather*}
as $\mathrm{Card}(X_k)\to\infty$.
Beware that this is stronger than the Zariski density of $\bigcup_{k\geq1}X_k$ in $\cM_d$.
\end{enumerate}
Contrary to the case of polynomials, items 1. and 2. seem very difficult to establish and could even be wrong as stated. This is an interesting problem since it would provide a description of the distribution of postcritically finite parameters at \emph{all places}, i.e., in the Berkovich analytification of the moduli space $\mathcal{M}_d(\C_v)$ of complex $v$-adic degree $d$ rational maps on $\P^1(\C_v)$
for all places $v\in M_\Q$.

For any $\un=(n_1,\ldots,n_{2d-2})\in (\N^{*})^{2d-2}$, we set
\begin{multline*}
X_{\un}:=\{[f]\in\mathcal{M}_d:f\text{ has }2d-2
\text{ periodic critical points } c_1,\ldots,c_{2d-2}\\
\text{ of respective exact periods }n_1, \dots, n_{2d-2}\},
\end{multline*}
so that $\mathrm{C}_{\un}\subset X_{\un}$.
A consequence of our counting of hyperbolic components is that any sequences of sets of centers of hyperbolic components of disjoint type is weakly generic. 

\begin{Theorem}\label{tm:genericMd}
For any sequence $(\un(k))_k$ of $(2d-2)$-tuples $\un(k)=(n_{1,k},\ldots,n_{2d-2,k})$ in $(\N^*)^{2d-2}$
satisfying $\min_j(n_j(k))\to\infty$ as $k\to\infty$,
the sequence $(X_{\un(k)})_k$ is Galois-invariant and weakly generic in $\cM_d$.
\end{Theorem}

\begin{remark} \normalfont \normalfont \normalfont
This result in particular implies that $\bigcup_kX_{\un(k)}$ is Zariski dense in $\mathcal{M}_d$, which refines \cite[Theorem A]{DeMarco-intersection}.
\end{remark}

For proving this weak genericity property, we prove a stronger result in 
the moduli space $\mathcal{M}_d^\cm$ of {\em critically marked}
degree $d$ rational maps on $\P^1$, i.e., the orbit space of $\PSL_2(\C)$
in the space $\rat_d^\cm$ of {\em critically marked}
degree $d$ rational maps $(f,c_1,\ldots,c_{2d-2})$,
where $f\in\rat_d$ and $(c_1,\ldots,c_{2d-2})$ is a $(2d-2)$-tuple of
all critical points of $f$, counted with multiplicity.
This is also an irreducible affine variety of dimension $2d-2$ which is defined over $\Q$ and the natural finite branched cover $p:\mathcal{M}_d^\cm\to\mathcal{M}_d$ 
is of degree $(2d-2)!$ and also defined over $\Q$.
We note that the same construction of $\mathcal{L}$, $\mu_\bif$, $\Per_n(w)$,
and $T_{\un}^p(\underline{\rho})$
as on $\cM_d$ works on $\cM_d^\cm$.

For any $n\in\N^*$ and any $1\leq j\leq 2d-2$, let $\Per_j(n)$ be the analytic hypersurface
\[
\Per_j(n):=\{[(f,c_1,\ldots,c_{2d-2})]\in\mathcal{M}_d^\cm:\Phi_n(F,C_j)=0\}
\]
in $\cM_d^{\cm}$, where $F$ and $C_j$ are lifts of $f$ and $c_j$, respectively; 
the degree of the hypersurface $\Per_j(n)$ is bounded from above by $Cd^n$ 
for some constant $C\geq1$ depending only on $d$
(see e.g.\ \cite{Silverman} for more details), and
since $\Lambda$ is quasi-projective, $\Per_j(n)$ 
is actually an algebraic hypersurface of $\cM_d^\cm$.

For any $\un=(n_1,\ldots,n_{2d-2})\in(\N^*)^{2d-2}$, we set
\[
Y_{\un}:=\bigcap_{j=1}^{2d-2}\Per_j(n_j)\subset\mathcal{M}_d^\cm.
\]
We prove here the following as an application of our counting result.

\begin{theorem}\label{tm:generic}
For any proper algebraic subvariety $V$ in $\mathcal{M}_d^\cm$, there exists a constant $C>0$ such that for all $\un\in(\N^*)^{2d-2}$, we have
\[\frac{\mathrm{Card}\left(Y_{\un}\cap V\right)}{\mathrm{Card}\left(Y_{\un}\right)}\leq C\cdot d^{-(\min_j n_j)/2}.\]
\end{theorem}

For the proof, we follow the strategy of~\cite[Theorem 5.3]{favregauthier} and we rely on the following, which is just an adaptation of 
\cite[Lemma~5.4]{favregauthier}.

\begin{lemma}\label{lm:reductioncount}
Let $V$ be any irreducible algebraic subvariety of dimension $q$ in $\mathcal{M}^\cm_d$ and let $p$ be a smooth point in $\mathcal{M}_d^\cm$. Assume that $V$ is also smooth at $p$. Pick hypersurfaces $H_1,\ldots,H_{2d-2}$ 
intersecting transversely at $p$.
Then there is $I\subset\{1,\ldots,2d-2\}$ of cardinality $2d-2-q$ such that $p$ is an isolated point of $V\cap\bigcap_{j\in I}H_j$.
\end{lemma}

\begin{proof}[Proof of Lemma~\ref{lm:reductioncount}]
In a local coordinate $(z_1,\ldots,z_{2d-2})$ around $p$, we may suppose that $H_i=\{z_i=0\}$ for each $i$. Since $V$ is smooth, its tangent space at $p$ has dimension $N:=2d-2-q$ and one of the $N$-forms $\omega_I:=dz^I$ for $|I|=N$ satisfies $\omega_I|_{T_pV}\neq0$. Since the kernel of $\omega_I$ is the tangent space of $T_pH_i$ at $p$, we conclude that the intersection between $V$ and the $H_i$'s with $i\in I$ is transverse.
\end{proof}

\begin{proof}[Proof of Theorem~\ref{tm:generic}]
The case $\dim V=0$ is an immediate consequence of Theorem~\ref{tm:counting}, since $V$ is a finite set in that case. We thus assume $q:=\dim V\in\{1,\ldots,2d-3\}$. 

Pick $\un=(n_1,\ldots,n_{2d-2})\in(\N^*)^{2d-2}$.
Let $Z_{\un}$ be the subset in $Y_{\un}$ consisting of 
all $[(f,c_1,\ldots,c_{2d-2})]\in\cM_d$ such that the orbits of
$c_{1},\ldots,c_{2d-2}$ of $f$ are also disjoint.
We claim that there is a constant $C'>0$ depending only on $d$ such that
\[
\mathrm{Card}\left(Y_{\un}\setminus Z_{\un}\right)\leq C' d^{|\un|-(\min_i n_i)/2};
\] 
for, since $Y_{\un}\setminus Z_{\un}$ consists of all
$[f]\in\cM_d^\cm$ such that $f$ has a super-attracting cycles of exact period $n_i$ and containing at least two distinct critical points for some $i$, we have
\[
Y_{\un}\setminus Z_{\un}\subset\bigcup_{i=0}^{2d-2}\left(\bigcup_{j\neq i \atop n_i=n_j}\bigcup_{k=0}^{[n_i/2]}\{f^k(c_i)=c_j\}\cap\bigcap_{\ell:\ell\neq i}\Per_\ell(n_\ell)\right).
\]
Since also $\deg(\{f^k(c_i)=c_j\})\leq C d^k$ for some $C>0$ independent of $i,j$ and $k$, the claim holds by B\'ezout's Theorem.

Set $N=N_q:=2d-2-q\in\N^*$. 
Let $V_\mathrm{reg}$ be the regular locus of $V$. We also claim that
there is a constant $C''>0$ depending only on $d$ such that
\[
\mathrm{Card}\left(V_\mathrm{reg}\cap Z_{\un}\right)
\leq C''\deg(V)\sum_I d^{\sum_{j=1}^Nn_{i_j}},
\]
where and below the sums $\sum_I$ range over all 
$N$-tuples $I=(i_1,\ldots,i_N)$ of distinct indices in $\{1,\ldots,2d-2\}$;
indeed, for any such choice $I$, we set
$Y_I:=\bigcap_{j=1}^N\Per_{i_j}(n_{i_j})$, 
and let $F_I$ be the set of all isolated points of $V\cap Y_I$.
By B\'ezout's Theorem, we have
\begin{equation}\label{eqbdbezout}
\mathrm{Card}(F_I) \leq \deg (V)  \prod_{j=1}^N \deg (\Per_{i_j}(n_{i_j}) ) = C_2\deg(V)d^{\sum_{j=1}^Nn_{i_j}}
\end{equation}
for some constant $C_2>0$ depending only on $d$. 
Since $\mathrm{Card}\left(V_\mathrm{reg}\cap Z_{\un}\right) \leq 
\sum_I\mathrm{Card}(F_I)$ by Lemma~\ref{lm:reductioncount},
the claim holds.
According to Theorem~\ref{tm:counting}, 
there is a constant $C_4>0$ depending only on $d$ such that 
$\mathrm{Card}(Y_{\un})\geq\mathrm{Card}(Z_{\un})\geq C_4 d^{|\un|}$
provided $\min_jn_j$ is large enough.

 Hence, 
the above two claims imply
\begin{gather*}
\mathrm{Card}\left(V_\mathrm{reg}\cap Z_{\un} \right)
\leq C_5\sum_I d^{-\sum_{j\notin I}n_{i_j}}\mathrm{Card}(Y_{\un})\quad\text{and}\\
\mathrm{Card}\left(V_\mathrm{reg}\setminus Z_{\un} \right)\le
\mathrm{Card}\left(Y_{\un}\setminus Z_{\un}\right)
\leq C_5 d^{-(\min_i n_i)/2}\mathrm{Card}(Y_{\un}), 
\end{gather*}
where $C_5>0$ depends only on $V$, $d$ and $q$. 

Since $V_{\mathrm{sing}}:= V\setminus V_\mathrm{reg}$ is an
algebraic subset in $\cM_d^\cm$ of codimension $2d-2-q+1$, the proof is complete
by a finite induction.
\end{proof}

\begin{remark} \normalfont 
If $n_i\neq n_j$ for all $i\neq j$, the proof gives
\[\frac{\mathrm{Card}\left(Y_{\un}\cap V\right)}{\mathrm{Card}\left(Y_{\un}\right)}\leq C\cdot d^{-\min\left(n_{j_1}+\cdots+n_{j_q}\right)}
\]
where the minimum is taken over all $q$-tuples $J=(j_1,\ldots,j_q)$ of distinct indices.
\end{remark}

\begin{proof}[Proof of Theorem~\ref{tm:genericMd}]
Pick any $\un=(n_1,\ldots,n_{2d-2})\in(\N^*)^{2d-2}$, and 
for any permutation $\sigma\in\mathfrak{S}_{2d-2}$, 
set $\sigma(\un):=(n_{\sigma(1)},\ldots,n_{\sigma(2d-2)})$. Let us first remark that 
$p^{-1}(X_{\un})=\bigcup_{\sigma\in\mathfrak{S}_{2d-2}}Y_{\sigma(\un)}$.
The Galois-invariance of $X_{\un}$ follows from the Galois-invariance of $Y_{\sigma(\un)}$ for any $\sigma\in\mathfrak{S}_{2d-2}$. Similarly, for any irreducible subvariety $Z\subset \mathcal{M}_d$ defined over $\Q$, we can apply Theorem~\ref{tm:generic} to any irreducible component of the algebraic subset $V=p^{-1}(Z)$ in $\mathcal{M}_d^\cm$. The fact that $p$ is a finite branched cover together with the assumption $\min_j(n_j(k))\to+\infty$ as $k\to\infty$ completes the proof.
\end{proof}

\section{Distribution of hyperbolic maps with given multipliers in $\mathcal{M}_d$}\label{sec:equicenter}
This section is devoted to the proof of Theorem~\ref{tm:vitessecenters}. 
Pick any $\un=(n_1,\ldots,n_{2d-2})\in(\N^*)^{2d-2}$ and 
{any $\uw=(w_1,\ldots,w_{2d-2})\in \D^{2d-2}$}.

\subsection{A reduction to work on $\cM_d^\fm$}
Recall that we denoted by $\pi_\fm:\mathcal{M}_d^\fm\to\mathcal{M}_d$ 
the natural finite branched cover of degree $\deg(\pi_\fm)=(d+1)!$ and recall the definition of $C_{\un,\uw}$ from the introduction, and the notations
$\mu_\bif^\fm$ and $T_\un^{2d-2,\fm}(\underline{\rho})$ introduced in
Subsection~\ref{sec:counting}. Set $C^\fm_{\un,\uw}:=\pi_\fm^{-1} (C_{\un,\uw})$,
\[
\mu_{\un,\uw}^\fm:=\frac{\#\mathrm{Stab}(\un,{\uw})}{d_{|\un|}}\sum_{C^\fm_{\un,\uw}}\delta_{[(f,x_1,\ldots,x_{d+1})]},
\]
and $\mu_{\un}^\fm:=\mu_{\un,(0,\ldots,0)}^\fm$.
Then $\mu_{\un,\uw}^\fm=(\pi_\fm)^*(\mu_{\un,\uw})$ for all $\uw$
and $\mu_\un^\fm=(\pi_\fm)^*(\mu_\un)$.
In particular, for any DSH and continuous function $\tilde{\Psi}$ on $\mathcal{M}_d$ with compact support, we have
\begin{align*}
\left\langle \mu_{\un,\uw}-\mu_\bif,\tilde{\Psi}\right\rangle 
=& \frac{1}{(d+1)!}\left\langle \mu_{\un,\uw}-\mu_\bif,(\pi_\fm)_*(\pi_\fm)^*\tilde{\Psi}\right\rangle\\
=& \frac{1}{(d+1)!}\left\langle \mu_{\un,\uw}^\fm-\mu_\bif^\fm,(\pi_\fm)^*\tilde{\Psi}\right\rangle.
\end{align*}
Hence, it is sufficient to have the wanted estimates in the fixed marked moduli space $\mathcal{M}_d^\fm$. So, pick  any compact set $K$ {\em in} $\mathcal{M}_d^\fm$, and
any either $\mathcal{C}^1$ or $\mathcal{C}^2$ function $\Psi$ on $\cM_d^\fm$
with support in $K$. {Set 
\begin{gather*}
\underline{\rho}=(\rho_1,\dots,\rho_{2d-2})
:=(\max(|w_1|, 1/2),\ldots,\max(|w_{2d-2}|, 1/2))\in([1/2,1[)^{2d-2}, 
\end{gather*}
so that $\rho_j\in[|w_j|,1[$ for any $1\le j\le 2d-2$.}
By Theorem~\ref{tm:vitessemoyennesalgebraic} and 
the upper bound of the DSH-norm by the $\mathcal{C}^2$-norm, 
there is $C(K)>0$ depending only on $K$ such that
\begin{equation*}
\left|\left\langle T_{\un}^{2d-2,\fm}(\underline{\rho})-\mu_\bif^\fm,\Psi\right\rangle\right|\leq C(K)\max_{1\leq j\leq 2d-2}\left(\frac{\sigma_2(n_j)}{d^{n_j}}\right)\|\Psi\|_{\mathcal{C}^2}\quad\text{ if }\Psi\text{ is }\mathcal{C}^2
\end{equation*}
and then, by {interpolation between Banach spaces}, that
\begin{equation*}
\left|\left\langle T_{\un}^{2d-2,\fm}(\underline{\rho})-\mu_\bif^\fm,\Psi\right\rangle\right|\leq C(K)\max_{1\leq j \leq 2d-2}\left(\frac{\sigma_2(n_j)}{d^{n_j}}\right)^{1/2}\|\Psi\|_{\mathcal{C}^1}\quad\text{  if }\Psi\text{ is }\mathcal{C}^1.
\end{equation*}
Once we show
\begin{gather}
\left|\langle \mu_{\un}^\fm- T_{\un}^{2d-2,\fm}(\underline{\rho}), \Psi \rangle \right| \leq 
C\max_{1\leq j\leq 2d-2}\left(\frac{1}{d^{n_j}}\right)\cdot\|\Psi\|_{\mathcal{C}^2}
\quad\text{ if }\Psi\text{ is }\mathcal{C}^2\label{eq:nonCS} 
\end{gather}
and
\begin{gather}
\left| \langle \mu_{\un,\uw}^\fm- T_{\un}^{2d-2,\fm}(\underline{\rho}), \Psi \rangle \right| 
\leq C \max_{1\leq j\leq 2d-2} \left(\frac{-1}{d^{n_j}\log\rho_j}\right)^{1/2} \|\Psi\|_{\mathcal{C}^1}\quad\text{ if }\Psi\text{ is }\mathcal{C}^1,\label{eq:CauchSchwarz} 
\end{gather}
where $C>0$ depends only on $d$, the proof of Theorem \ref{tm:vitessecenters}
will be complete.

\subsection{A reduction to work on algebraic curves}\label{sec:reductioncurve}
Observe that the measure $\mu_{\un,\uw}^\fm- T_{\un}^{2d-2,\fm}(\underline{\rho})$
on $\cM_d^\fm$
has its support contained in the union of all hyperbolic components of type $\un$. Let $\Omega$ be such a component, and 
$\mathcal{W} =(\mathcal{W}_1,\ldots,\mathcal{W}_{2d-2}):\Omega\rightarrow\D^{2d-2}$ the  multiplier map on $\Omega$. 
Letting $\lambda_r$ be the normalized Lebesgue measure on 
$\partial\D_r$, by~\eqref{pourallervite} we have
\begin{align*}
&d_{|\un|}(T_{\un}^{2d-2,\fm}(\underline{\rho})-\mu_{\un,\uw}^\fm )\\
=&\# \mathrm{Stab}( \un, \uw)  \mathcal{W}^*(\lambda_{\rho_1} \otimes \cdots \otimes \lambda_{\rho_{2d-2}} ) - \# \mathrm{Stab}( \un,\uw) \bigwedge_{i=1}^{2d-2}[\Per_{n_i}(w_i)]\\
=&\sum_{\sigma \in \mathrm{Stab}( \un, \uw) } \mathcal{W}_{\sigma(1)}^*(\lambda_{\rho_1} )\wedge \dots \wedge \mathcal{W}_{\sigma(2d-2)}^*(\lambda_{\rho_{2d-2}} )- \mathcal{W}_{\sigma(1)}^*(\delta_{w_1})\wedge \dots \wedge \mathcal{W}_{\sigma(2d-2)}^*(\delta_{w_{2d-2}})\\
=& 
\sum_{\sigma \in \mathrm{Stab}( \un,\uw )} \sum_{j=1}^{2d-2} S_{\sigma,j}
\end{align*}
on $\Omega$, 
where for any $\sigma\in \mathrm{Stab}(\un)$ and any $1\le j\le 2d-2$, the measure
$S_{\sigma,j}$ is defined as
\begin{multline*}
S_{\sigma,j}:=\left(\bigwedge_{1\leq i<j}\mathcal{W}_{\sigma(i)}^*(\lambda_{\rho_i}) \right)\wedge \left( \mathcal{W}_{\sigma(j)}^*(\lambda_{\rho_j}) -  \mathcal{W}_{\sigma(j)}^*(\delta_{w_j}) \right)\wedge  \left( \bigwedge_{k>j} \mathcal{W}_{\sigma(k)}^*(\delta_{w_k})\right)\\
=\int_{\tilde{\mathbb{S}}_{j-1}} 
\bigwedge_{1\leq i<j}\mathcal{W}_{\sigma(i)}^*(\delta_{u_i}) \wedge \left( \mathcal{W}_{\sigma(j)}^*(\lambda_{\rho_j}) - \mathcal{W}_{\sigma(j)}^*(\delta_{w_j}) \right)\wedge \bigwedge_{k>j} \mathcal{W}_{\sigma(k)}^*(\delta_{w_k})
\mathrm{d}\lambda_{\tilde{\mathbb{S}}_{j-1}}(u_1,\dots,u_{j-1})     
\end{multline*}
on $\Omega$, setting
$\lambda_{\tilde{\mathbb{S}}_{j-1}}:= \lambda_{\rho_1}\otimes\cdots\otimes\lambda_{\rho_{j-1}}$ on
$\tilde{\mathbb{S}}_{j-1}:=\partial\D_{\rho_1} \times \cdots \times\partial\D_{\rho_{j-1}}$ if $j>1$ and $S_{\sigma,1}:= \left( \mathcal{W}_{\sigma(1)}^*(\lambda_{\rho_1}) - \mathcal{W}_{\sigma(1)}^*(\delta_{w_1}) \right)\wedge \bigwedge_{k>1} \mathcal{W}_{\sigma(k)}^*(\delta_{w_k})$.    Recall that a wedge-product over the empty set is equal to $1$ and that an intersection over the empty set is the whole space.

For any $\sigma \in \mathrm{Stab}(\un,\uw)$, any $1\le j\le 2d-2$,
and any $u\in\tilde{\mathbb{S}}_{j-1}$ (if $j>1$),
let $\Lambda_{\sigma,j}(u)$ or $\Lambda_{\sigma,1}$
be the set of all $[(f,x_1,\ldots,x_{d+1})]\in\cM_d^\fm$
having a cycle of exact period $n_{\sigma(i)}$ and multiplier $u_i\in\D_{\rho_i}$ 
for any $1\leq i<j$ and a cycle of exact period $n_{\sigma(k)}$ and multiplier 
$w_k$ for any $k>j$. Hence
\[\Lambda_{\sigma,j}(u)\subset\bigcap_{1\leq i<j}\Per_{n_i}(u_{\sigma(i)}) \cap \bigcap_{k>j} \Per_{n_k}(w_{\sigma(k)})\]
is an algebraic curve and, by B\'ezout's theorem,
its area is $\le C\cdot d^{|\un|-n_j}$ for some constant $C$ 
depending only on $d$. Set
$\mathcal{W}_{\Lambda_{\sigma,j}(u)}:=p_j\circ(\mathcal{W}|_{\Omega\cap \Lambda_{\sigma,j}(u)})$ or $\mathcal{W}_{\Lambda_{\sigma,1}}:=p_1\circ(\mathcal{W}|_{\Omega\cap \Lambda_{\sigma,1}})$,
where $p_j:\D^{2d-2}\to\D$ is the projection on the $j$-th coordinate. Then the measure 
\begin{equation*}
\bigwedge_{1\leq i<j}\mathcal{W}_{\sigma(i)}^*(\delta_{u_i}) \wedge \left( \mathcal{W}_{\sigma(j)}^*(\lambda_{\rho_j}) - \mathcal{W}_{\sigma(j)}^*(\delta_{w_j}) \right)\wedge \bigwedge_{k>j} \mathcal{W}_{\sigma(k)}^*(\delta_{w_k})
\end{equation*}
is equal to
\begin{gather}
\mathcal{W}_{\Lambda_{\sigma,j}(u)}^*(\lambda_{{\rho_j}}-\delta_{w_j}).\label{decomposition}
\end{gather} 
on $\Omega\cap \Lambda_{\sigma,j}(u)$ if $j>1$, and the measure $S_{\sigma,1}$ is itself equal to $\mathcal{W}_{\Lambda_{\sigma,1}}(\lambda_{\rho_1}-\delta_{w_1})$ on $\Omega\cap \Lambda_{\sigma,1}$.

\subsection{Proof of \eqref{eq:CauchSchwarz}: the case of arbitrary multipliers in $\D^{2d-2}$}
Assume that $\Psi$ is $\mathcal{C}^1$ and test \eqref{decomposition} against $\Psi$. For any $\sigma \in \mathrm{Stab}( \un,\uw)$, any $1\le j\le 2d-2$, and any  $u:=(u_i)_{i<j}\in\tilde{\mathbb{S}}_{j-1}$ (if $j>1$). We continue to fix $\Omega$ as in Subsection~\eqref{sec:reductioncurve} for a while and  let $O:=\mathcal{W}_{\Lambda_{\sigma,j}(u)}^{-1}(w_j)$. Then
\begin{align*}
\int_{\Lambda_{\sigma,j}(u) \cap \Omega} \Psi \cdot \mathcal{W}_{\Lambda_{\sigma,j}(u)}^*(\lambda_{{\rho_j}}-\delta_{w_j} )= \int_{\Lambda_{\sigma,j}(u) \cap \Omega} \left(\Psi -\Psi(O) \right) \cdot \mathcal{W}_{\Lambda_{\sigma,j}(u)}^*(\lambda_{{\rho_j}}),
\end{align*}
so that by the mean value inequality:
\begin{align*}
\left|\int_{\Lambda_{\sigma,j}(u) \cap \Omega} \Psi \cdot \mathcal{W}_{\Lambda_{\sigma,j}(u)}^*(\lambda_{{\rho_j}}-\delta_{w_j} ) \right| 
{\le}&\int_{\Lambda_{\sigma,j}(u) \cap \Omega} \left|\Psi -\Psi(O) \right| \cdot \mathcal{W}_{\Lambda_{\sigma,j}(u)}^*(\lambda_{\rho_j}) \\
\leq& C\cdot\|\Psi\|_{\mathcal{C}^1} \mathrm{diam}( \mathcal{W}_{\Lambda_{\sigma,j}(u)}^{-1}(\D(0,{\rho_j}))),
\end{align*}
where the diameter is computed with respect to the distance induced by $\beta$ and the constant $C>0$ only depends on the choice of the $\mathcal{C}^1$-norm.  
By the length-area estimate (Lemma~\ref{lm:BriendDuval}), we have
\[
\mathrm{diam}( \mathcal{W}_{\Lambda_{\sigma,j}(u)}^{-1}(\D(0,{\rho_j})))^2 \leq \tau\cdot\frac{\textup{Area}(\Omega\cap \Lambda_{\sigma,j}(u))}{\min\{1,\frac{1}{2\pi}\log(1/{\rho_j})\}} 
=\tau\cdot\frac{\textup{Area}(\Omega\cap \Lambda_{\sigma,j}(u))}{|\log\rho_j|/(2\pi)}
\]
since $\mathcal{W}_{\Lambda_{\sigma,j}(u)}^{-1}(\D(0,{\rho_j})) \Subset \mathcal{W}_{\Lambda_{\sigma,j}(u)}^{-1}(\D(0,1))=\Omega\cap \Lambda_{\sigma,j}(u)$ 
in $\Lambda_{\sigma,j}(u)$ are holomorphic disks and
$\rho_j\geq1/2$.
Using Cauchy-Schwarz inequality gives
\begin{align*}
&\left| \sum_{\Omega} \int_{\Lambda_{\sigma,j}(u) \cap \Omega} \Psi\cdot \mathcal{W}_{\Lambda_{\sigma,j}(u)}^*(\lambda_{{\rho_j}}-\delta_{\rho_j} ) \right|\\
\leq &C\cdot\|\Psi\|_{\mathcal{C}^1}\cdot \sum_{\Omega} \left(\tau\cdot\frac{\textup{Area}(\Omega\cap \Lambda_{\sigma,j}(u))}{|\log\rho_j|/(2\pi)}\right)^{1/2} \\
\leq &C \left(\sum_{\Omega}\frac{\tau}{
|\log\rho_j|/(2\pi)}\right)^{1/2}\left(\sum_{\Omega}\textup{Area}(\Omega\cap \Lambda_{\sigma,j}(u))\right)^{1/2}\|\Psi\|_{\mathcal{C}^1}\\
\leq &C \left(\frac{\tau}{
|\log\rho_j|/(2\pi)}\right)^{1/2}N_\fm(\un)^{1/2}\left(\textup{Area}( \Lambda_{\sigma,j}(u))\right)^{1/2}\|\Psi\|_{\mathcal{C}^1},
\end{align*}
so that recalling that $N_\fm(\un)=N_{\cM_d^\fm}(\un)\leq C_1d^{|\un|}$
 by Bézout's theorem
and $ \textup{Area}(\Lambda_{\sigma,j}(u)) \leq C_2 d^{| \un|-n_j}$, where $C_1,C_2>0$ depend only on $d$, we have
\begin{align*}
\left| \sum_{\Omega} \int_{\Lambda_{\sigma,j}(u) \cap \Omega} \Psi\cdot \mathcal{W}_{\Lambda_{\sigma,j}(u)}^*(\lambda_{\rho_j}-\delta_{\rho_j} ) \right| 
&\leq C_3 \left(\frac{d^{-n_j}\tau}{
|\log\rho_j|/(2\pi)
}\right)^{1/2}  d^{|\un|}\|\Psi\|_{\mathcal{C}^1}, 
\end{align*}
where $C_3>0$ depends only on $d$. Similarly,
\begin{align*}
\left| \sum_{\Omega} \int_{\Lambda_{\sigma,1} \cap \Omega} \Psi\cdot \mathcal{W}_{\Lambda_{\sigma,1}}^*(\lambda_{\rho_1}-\delta_{\rho_1} ) \right| 
&\leq C_3 \left(\frac{d^{-n_1}\tau}{
	|\log\rho_1|/(2\pi)
}\right)^{1/2}  d^{|\un|}\|\Psi\|_{\mathcal{C}^1}. 
\end{align*}
  Since the right-hand sides are independent
of $u$ and $\sigma$, recalling \eqref{decomposition},
we have
\begin{align*}
\left| \langle \mu_{\un,\underline{\rho}}^\fm- T_{\un}^{2d-2,\fm}(\underline{\rho}), \Psi \rangle \right| 
\leq C_4 {\max_{1\leq j \leq 2d-2}} \left(\frac{1}{d^{n_j}|\log\rho_j|}\right)^{1/2} \|\Psi\|_{\mathcal{C}^1}, 
\end{align*}
where $C_4>0$ depends only on $d$ (and actually not on $K$). 
Hence \eqref{eq:CauchSchwarz} holds.

\subsection{Proof of \eqref{eq:nonCS}: the center of components}

Assume that $\Psi$ is $\mathcal{C}^2$ and test \eqref{decomposition} against $\Psi$.  Pick any $\sigma \in \mathrm{Stab}(\un)$, any $1\le j\le 2d-2$, and any  $(u_i)_{i<j}\in\tilde{\mathbb{S}}_{j-1}$ (if $j>1$). We continue to fix $\Omega$ as in Subsection~\eqref{sec:reductioncurve} for a while and
let $O:=\mathcal{W}_{\Lambda_{\sigma,j}(u)}^{-1}(0,\ldots,0)$ be the center of the disk $\Omega \cap \Lambda_{\sigma,j}(u)$. Then we have
\begin{align*}
&\int_{\Lambda_{\sigma,j}(u) \cap \Omega} \Psi \cdot \mathcal{W}_{\Lambda_{\sigma,j}(u)}^*(\lambda_{{\rho_j}}-\delta_{0})\\
=& \int_{\Lambda_{\sigma,j}(u) \cap \Omega} \left(\Psi -\Psi(O) \right) \cdot \mathcal{W}_{\Lambda_{\sigma,j}(u)}^*(\lambda_{{\rho_j}}) \\
=&  \int_{\Lambda_{\sigma,j}(u) \cap \Omega} \left(\Psi(z)-\left(\mathrm{D}_{O}\Psi\right)(z -O)  -\Psi(O)\right) ((\mathcal{W}_{\Lambda_{\sigma,j}(u)})^*(\lambda_{\rho_j}))(z),
\end{align*}
the latter equality in which holds by harmonicity, that is,
\[
\int_{\Lambda_{\sigma,j}(u) \cap \Omega} \left(\mathrm{D}_{O}\Psi\right)(z -O) \cdot (\mathcal{W}_{\Lambda_{\sigma,j}(u)})^*(\lambda_{{\rho_j}})(z)=  \left(\mathrm{D}_{O}\Psi\right)(O -O)=0.
\] 
By the mean value inequality, we have
\begin{align*}
&\left|\int_{\Lambda_{\sigma,j}(u) \cap \Omega} \Psi \cdot \mathcal{W}_{\Lambda_{\sigma,j}(u)}^*(\lambda_{{\rho_j}}-\delta_{0} ) \right|\\ 
&\le\int_{\Lambda_{\sigma,j}(u) \cap \Omega}  \left|\Psi(z)-\left(\mathrm{D}_{O}\Psi\right)(z -O)  -\Psi(O)\right| \cdot(\mathcal{W}_{\Lambda_{\sigma,j}(u)})^*(\lambda_{\rho_j})(z) \\
&\leq C \cdot \|\Psi\|_{\mathcal{C}^2}\cdot\mathrm{diam}( \mathcal{W}_{\Lambda_{\sigma,j}(u)}^{-1}\left(\D(0,{\rho_j}))\right)^2,
\end{align*}
where the diameter is again computed with respect to the distance induced by $\beta$ and the constant $C>$ depends only on the choice of the $\mathcal{C}^2$-norm.  
Again by the length-area estimate (Lemma~\ref{lm:BriendDuval}), we have
\begin{align*}
\left|\int_{\Lambda_{\sigma,j}(u) \cap \Omega} \Psi\cdot \mathcal{W}_{\Lambda_{\sigma,j}(u)}^*(\lambda_{{\rho_j}}-\delta_{0} ) \right|
& \leq C\cdot\|\Psi\|_{\mathcal{C}^2} \cdot \left(\tau\cdot\frac{\textup{Area}(\Omega\cap \Lambda_{\sigma,j}(u))}{|\log(1/2)|/(2\pi)}\right)
\end{align*}
since $\mathcal{W}_{\Lambda_{\sigma,j}(u)}^{-1}(\D(0,{\rho_j})) 
\Subset \mathcal{W}_{\Lambda_{\sigma,j}(u)}^{-1}(\D(0,1))=\Omega\cap \Lambda_{\sigma,j}(u)$ in $\Lambda_{\sigma,j}(u)$
are holomorphic disks (here, $\rho_j\equiv 1/2$ by definition). Hence
\begin{align*}
\left|\int_{\Lambda_{\sigma,j}(u) \cap \Omega} \Psi \cdot \mathcal{W}_{\Lambda_{\sigma,j}(u)}^*(\lambda_{{\rho_j}}-\delta_{0} ) \right|
& \leq C'\cdot\|\Psi\|_{\mathcal{C}^2}\cdot\textup{Area}(\Omega\cap \Lambda_{\sigma,j}(u)),
\end{align*}
where $C'>0$ depends only on $d$, so 
recalling that 
$\sum_{\Omega}\textup{Area}(\Omega\cap \Lambda_{\sigma,j}(u))\leq \textup{Area}(\Lambda_{\sigma,j}(u)) 
\leq C_1d^{| \un|-n_j}$, where $C_1>0$ depends only on $d$,
\begin{align*}
\left| \sum_{\Omega}\int_{\Lambda_{\sigma,j}(u) \cap \Omega} \Psi\cdot\mathcal{W}_{\Lambda_{\sigma,j}(u)}^*(\lambda_{{\rho_j}}-\delta_{0} ) \right| 
& \leq C'C_1\cdot\|\Psi\|_{\mathcal{C}^2} d^{|\un|-n_j}.
\end{align*}
 Similarly,
 \[
 \left| \sum_{\Omega}\int_{\Lambda_{\sigma,1}\cap \Omega} \Psi\cdot\mathcal{W}_{\Lambda_{\sigma,1}}^*(\lambda_{{\rho_1}}-\delta_{0} ) \right| \leq C'C_1\cdot\|\Psi\|_{\mathcal{C}^2} d^{|\un|-n_1}.\]
  Since the right-hand sides are independent
of $u$ and $\sigma$, recalling \eqref{decomposition}, we have
\begin{gather*}
\left| \langle \mu_{\un}^\fm- T_{\un}^{2d-2,\fm}(\underline{\rho}), \Psi \rangle \right| \leq 
C''\|\Psi\|_{\mathcal{C}^2} {\max_{1\leq j\leq 2d-2}}d^{-n_j},  
\end{gather*}
where $C''>0$ depends only on $d$. Hence \eqref{eq:nonCS} holds. \qed

\begin{remark} \label{6.1} \normalfont 
One cannot hope to the (even qualitative) convergence of $\mu_\un$ to $\mu_\bif$ for bounded DSH observables. Indeed, consider the DSH function 
\[\phi_A:= \min(\max\{\log |Z|,-A\}/A, 0)\] 
on $\C^{2d-2}$ for some $A>0$, 
which is identically equal to $0$ outside the unit ball and equal to $-1$ in $0$. Furthermore, it is DSH and its DSH norm can be taken arbitrarily small by taking $A$ large enough. By a change of coordinates, one can then construct a DSH function in $\mathcal{M}_{d}$ which is equal to $-1$ at the center of a given hyperbolic component and $0$ outside that component, with arbitrarily small DSH norm. Summing over sufficiently many hyperbolic components gives a contradiction to the convergence of  $\mu_\un$ to $\mu_\bif$. 

Nevertheless, it would be interesting to find a space of test functions independent of the choice of coordinates for which a similar statement as Theorem~\ref{tm:vitessecenters} holds. 
\end{remark}

\subsection{Application: lower bound on the local number of components of type $\un$}
For a set $L \subset  \mathcal{M}_d$ and $\un \in \N^{2d-2}$, we denote by $N_L(\un)$ the number of hyperbolic components of type $\un$  in $\cM_d$ whose center is in $L$. As a consequence of Theorem~\ref{tm:vitessecenters}, we prove the following local lower bound on the number $N_L(\un)$ of components of type $\un$ in $\cM_d$ intersecting $L$. This result can be seen as a local version of Theorem~\ref{tm:counting} (since $N_{\mathcal{M}_d}(\un)=N(\un)$).

\begin{theorem}\label{tm:countinglower}
	For any open subset $L$ in $\mathcal{M}_d$ 
		such that $\mu_\bif(L)>0$, 
	we have
	\[
	\liminf_{\min_j n_j\to\infty}\frac{N_L(\un)}{d^{|\un|}}>0.
	\]
	Moreover, for any $(n_1,\ldots,n_{d-1})\in(\N^*)^{d-1}$ 
	such that $n_i \neq n_j$ for any $1\leq i,j \leq d-1$, we have
	\[
	\liminf_{\min_{d\leq j\leq 2d-2}n_j\to\infty}\frac{N(n_1,\ldots,n_{d-1},n_d,\ldots,n_{2d-2})}{d_{|(n_1,\ldots,n_{d-1},n_d,\ldots,n_{2d-2})|}}>0.
	\]
\end{theorem}
\begin{proof}[Proof of the first assertion of Theorem \ref{tm:countinglower}] 
	Let $L$ be an open subset in $\mathcal{M}_d$. As $\mu_\bif$ is inner regular, we can find a compact set $K \Subset L$ such that $\mu_\bif(K) >0$. Pick 
	a smooth cut-off function $\Psi$ 
	with support in $L$ such that $\Psi=1$ on $K$, 
	so that $\langle \mu_\bif, \Psi\rangle \geq \mu_\bif(K) >0$. By  Theorem~\ref{tm:vitessecenters} applied to the compact set $K':= \supp{\Psi}$, we have that:
\[ 
\left| \langle \mu_\un-\mu_\bif, \Psi \rangle  \right|
\leq  C_{K'}\cdot
\max_{1\leq j \leq 2d-2}\left( \frac{\sigma_2(n_j)}{d^{n_j}}
\right)\cdot\|\Psi\|_{\mathcal{C}^2} .\]
As $\langle \mu_\un, \Psi \rangle \leq  \# \mathrm{Stab}(\un). N_L(\un) /d_{|\un|}$ by definition of $\mu_\un$, we deduce:
\[ 
\mu_\bif(K) \leq  \frac{\# \mathrm{Stab}(\un) .N_L(\un)}{d_{|\un|}}+
 C_{K'}\cdot
\max_{1\leq j \leq 2d-2}\left( \frac{\sigma_2(n_j)}{d^{n_j}}
\right)\cdot\|\Psi\|_{\mathcal{C}^2} .\]
Now, the last term on the right-hand side of the equality goes to $0$ when  $\min_j n_j \to \infty$, so it is $< \mu_\bif(K) /2$ for $\min_j n_j$ large enough. The result follows since $d_{|\un|} \simeq d^{|\un|}$. 
\end{proof}

\begin{remark} \normalfont 
	Recall that
	for any Misiurewicz rational map $f_0$, the class $[f_0] \in \mathcal{M}_d$ is in 
	$\mathrm{supp}(\mu_\bif)$ \cite{Article1, Article2}. 
	Hence, in any small neighborhood $V([f_0])$ of $[f_0]$, they are $O(d^{|\un|})$ hyperbolic components of type $\un$ that intersect  $V([f_0])$.
\end{remark}

\begin{proof}[Proof of the second assertion of Theorem~\ref{tm:countinglower}]
 We consider the algebraic variety $\Lambda$ defined by $\Lambda= \bigcap_{j=1}^{d-1}\Per_{n_j}(0) \cap \mathcal{M}_d$.
	
	Observe first that $\Lambda\neq \varnothing$ since we can find a polynomial of degree $d$ in $\Lambda$ : recall that the algebraic relations defining $\Per_{n_j}(0)$ do not intersect at the boundary of the moduli space of polynomials of degree $d$ so their intersection is non-empty by cohomology (this is where we use that the $n_j$ are distinct). In particular, $\Lambda$ has complex dimension at least $d-1$. We now rely on the following.
	
	\begin{lemma}\label{lm:criterion}
		Let $(f_\lambda)_{\lambda\in X}$ be any algebraic family of degree $d$ rational maps of dimension $k\geq1$ which is reduced, i.e. $\{\lambda\in X\, ; [f_\lambda]=[f_{\lambda_0}]\}$ is finite for all $\lambda_0\in X$. Assume $X$ contains no flexible Latt\`es parameter. Then $(dd^cL)^k>0$ on $X$.
	\end{lemma}

	By definition, $\Lambda$ is reduced  and, since any $f$ with $[f]\in\Lambda$ has (at least) an attracting cycle, $\Lambda$ contains no class of Latt\`es maps. Let $k:=\dim\Lambda\geq d-1$. Then $(dd^cL)^k>0$ on $\Lambda$ by Lemma~\ref{lm:criterion} above. Since any map $f$ with $[f]\in\Lambda$ has $d-1$ attracting cycles, we also have $(dd^cL)^d=0$, so that $k= d-1$ and $\mu_{\bif,\Lambda}:=\left.(dd^c L)^{d-1}\right|_{\Lambda}$ is a non-zero positive measure on $\Lambda$.
	
	Finally, we proceed as in the proof of the first assertion of 
	Theorem \ref{tm:countinglower}
	to produce sufficiently many hyperbolic components in $\Lambda$ using that:
	\begin{itemize}
		\item a hyperbolic component $\Omega'$ in $\Lambda$
		such that any $f$ with $[f]\in \Omega'$ has $d-1$ distinct attracting cycles of respective period $n_d,\ldots,n_{2d-2}$ corresponds uniquely to a hyperbolic component $\Omega$ in $\mathcal{M}_d$
		such that $f\in \Omega$ has $2d-2$ distinct attracting cycles of respective period $n_1,\ldots,n_{2d-2}$ and that $\Omega'= \Omega\cap \Lambda$.
		\item the multiplier map $\mathcal{W}':\Omega'\rightarrow\D^{d-1}$ on such component is still a biholomorphism as the restriction of the multiplier map $\mathcal{W}:\Omega\rightarrow\D^{2d-2}$ to $\Omega\cap \Lambda$.
		\item in particular, we can prove an analogue of Theorem~\ref{tm:vitessecenters} in $\Lambda$ (the proof is exactly the same).
	\end{itemize}
We then proceed as above.\end{proof}
	
We now give the proof of Lemma~\ref{lm:criterion}. The proof, which follows closely that of \cite[Lemmas 2.4 and 2.5]{buffepstein}, relies on the rigidity of stable algebraic families due to McMullen.
\begin{proof}[Proof of Lemma~\ref{lm:criterion}]
	Up to taking a finite branched cover of $X$, we may assume that $(f_\lambda)_{\lambda\in X}$ is endowed with $2d-2$ marked critical points, i.e. that there exist morphisms $c_1,\ldots,c_{2d-2}:X\to\P^1$ such that $\mathrm{Crit}(f_\lambda)=\{c_1(\lambda),\ldots,c_{2d-2}(\lambda)\}$ counted with multiplicity. Since $(f_\lambda)_{\lambda\in X}$ is a reduced algebraic family that contains no flexible Latt\`es maps, its bifurcation locus $\mathrm{Bif}(X)$ is non-empty by \cite[Theorem 2.2]{McMullen4}. 
	By Ma\~n\'e-Sad-Sullivan's description of bifurcation loci~\cite{MSS}, this implies that there exists $m_1\geq1$ and $\theta_1\in\R\setminus\Q$ such that $\Per_{m_1}(e^{2i\pi\theta_1})$ is a non-empty proper subvariety of $X$. We repeat the argument to find $m_2>m_1$ and $\theta_2\in\R\setminus\Q$ such that $\Per_{m_2}(e^{2i\pi\theta_2})\cap \Per_{m_1}(e^{2i\pi\theta_1})\neq\varnothing$ has codimension $2$ in $X$. Applying this argument inductively gives $m_1<\cdots<m_k$ and $\theta_1,\ldots,\theta_k\in\R\setminus\Q$ such that
	$\Per_{m_1}(e^{2i\pi\theta_1})\cap\cdots\cap\Per_{m_k}(e^{2i\pi\theta_k})$ is a non-empty finite subset of $X$.
	
	We define an algebraic curve in $X$ by
	$C:=\bigcap_{j=1}^{k-1}\Per_{m_j}(e^{2i\pi\theta_j})$.
	As there exists a parameter $\tilde{\lambda}_0\in C$ for which $f_{\tilde{\lambda}_0}$ has a neutral cycle, non-persistent in $C$, the bifurcation locus $\mathrm{Bif}(C)=\supp(dd^cL|_C)$ of $C$ is non-empty. 
	Pick $\lambda_0\in \mathrm{Bif}(C)$ and let $U$ be a small neighborhood of that $\lambda_0$ in $X$. By Montel Theorem, there exists $\lambda_1\in U\cap C$, $1\leq i_1\leq 2d-2$ and $q_1,r_1\geq1$ such that $f_{\lambda_1}^{q_1}(c_{i_1}(\lambda_1))=f_{\lambda_1}^{q_1+r_1}(c_{i_1}(\lambda_1))$ and $f_{\lambda_1}^{q_1}(c_{i_1}(\lambda_1))$ is a repelling periodic point of $f_{\lambda_1}$ of exact period $r_1$, and such that this relation is not persistent through $C$ (see e.g.~\cite[Lemma 2.3]{favredujardin}). Let
	\[C_1:=\{\lambda\in U\, ; \ f_\lambda^{q_1}(c_{i_1}(\lambda))=f_{\lambda}^{q_1+r_1}(c_{i_1}(\lambda))\}\cap\bigcap_{j=1}^{k-2}\Per_{m_j}(e^{2i\pi\theta_j})\subset U.\]
	By the same argument as above, we find $\lambda_2\in U\cap C_1$ very close to $\lambda_1$, $1\leq i_2\leq 2d-2$ distinct from $i_1$, $q_2\geq1$ and $r_2>r_1$ such that $f_{\lambda_2}^{q_2}(c_{i_2}(\lambda_2))=f_{\lambda_2}^{q_2+r_2}(c_{i_2}(\lambda_2))$ and $f_{\lambda_2}^{q_2}(c_{i_2}(\lambda_2))$ is a repelling periodic point of $f_{\lambda_2}$ of exact period $r_2$, and such that this relation is not persistent through $C_1$. By a finite induction, we find a parameter $\lambda_{k-1}\in X$, positive integers $q_1,\ldots,q_{k-1}$ and $r_1<r_2<\ldots<r_{k-1}$ and pairwise distinct indices $i_1,\ldots,i_{k-1}$ such that for all $1\leq j\leq k-1$:
	\[\lambda_{k-1}\in X_j:=\{\lambda\in X\, ; \ f_{\lambda}^{q_j}(c_{i_j}(\lambda))=f_{\lambda}^{q_j+r_j}(c_{i_j}(\lambda))\},\]
 $f_{\lambda_{k-1}}^{q_j}(c_{i_j}(\lambda_{k-1}))$ is a repelling periodic point for $f_{\lambda_{k-1}}$ of exact periods $r_j$ and the intersection of the $X_j$'s is proper at $\lambda_{k-1}$. 
	Moreover, the intersection $X_1\cap\cdots \cap X_{k-1}\cap\Per_{m_k}(e^{2i\pi\theta_k})$ is proper, hence we may proceed as above to find $\lambda_k\in X_1\cap\cdots \cap X_{k-1}$, positive integers $q_k$ and $r_k$ and $i_k\notin \{i_1,\ldots,i_{k-1}\}$ such that 
	\[\lambda_k\in X_k:=\{\lambda\in X\, ; \ f_{\lambda}^{q_k}(c_{i_k}(\lambda))=f_{\lambda}^{q_k+r_k}(c_{i_k}(\lambda))\}\]
	and, for all $1\leq j\leq k-1$, $f_{\lambda_{k}}^{q_j}(c_{i_j}(\lambda_k))$ is a repelling periodic point for $f_{\lambda_{k}}$ of exact periods $r_j$ and the intersection of the $X_j$'s is proper at $\lambda_{k}$. 
	By \cite[Theorem 6.2]{Article1}, we have $\lambda_k\in\supp((dd^cL)^k)$ and the proof is complete.
\end{proof}

\begin{remark} \normalfont The positivity of $N_L(\un)$ and $N(\un)$ is only guaranteed by Theorem~\ref{tm:countinglower} when all the $n_j$ are larger than some $n_0$ (or at least half of them).  Indeed, consider the case where $d+1$ of the $n_i$ are equal to $1$ and $d\geq 3$: then a rational map of degree $d$ has exactly $d+1$ fixed points (with multiplicity) and at most $d$ of them are attracting by the holomorphic index formula (see~\cite[\S9]{Milnor4}). So we have $N(1,\dots,1, n_{d+2}, \dots, n_{2d-2} )=0$. For $d=3$ it is also possible to get 0 components by allowing to many distinct $2$-cycles.
	Nevertheless, an explicit bound of $n_0$ would be highly interesting.     
\end{remark}

\section{Distribution of hyperbolic maps in $\mathcal{P}_d^\cm$}\label{sec:poly}

\subsection{A good parametrization of $\mathcal{P}_d^\cm$}

We refer to \cite[\S 5]{favredujardin} and \cite[\S 2]{distribGV} for the material of this section.
Recall that the \emph{critically marked moduli space} $\mathcal{P}_d^\cm$ of degree $d$ polynomials is the space of affine conjugacy classes of degree $d$ polynomials with $d-1$ marked critical points in $\C$. We define a finite branched cover of $\C^{d-1}\to\mathcal{P}_d^\cm$ as follows. For $c=(c_1,\ldots,c_{d-2})\in\C^{d-2}$ and $a\in\C$, let
\[P_{c,a}(z):=\frac{1}{d}z^d+\sum_{j=2}^{d-1}(-1)^{d-j}\frac{\sigma_{d-j}(c)}{j}z^j+a^d~, \ z\in\C~,\]
where $\sigma_k(c)$ is the monic elementary degree $k$ symmetric polynomial in the $c_i$'s. This family is known to be a finite branched cover of $\mathcal{P}_d^\cm$. Remark also that the (finite) critical points of $P_{c,a}$ are exactly $c_0, c_1,\ldots,c_{d-2}$, taking into account their multiplicity, where we set $c_0:=0$, and that they depend algebraically on $(c,a)\in\C^{d-1}$.
From now on,
we work on the parameter space $\C^{d-1}$ of the family
$(P_{c,a})_{(c,a)\in\C^{d-1}}$ rather than
$\mathcal{P}_d^\cm$ itself, without loss of generality.

The dynamical Green function of $P_{c,a}$ is the continuous psh function $g_{c,a}:\C\to\R_+$ defined by
\[g_{c,a}(z):=\lim_{n\to\infty}d^{-n}\log^+|P_{c,a}^n(z)|, \ z\in\C,\]
where the limit is locally uniform in $(c,a,z)\in\C^d$. For any  $0\le j\le d-2$, the function $g_j(c,a):=g_{c,a}(c_j)$ is psh and continuous on $\C^{d-1}$ and, setting $T_j:=dd^cg_j$, we have $dd^cL=\sum_jT_j$ and $T_j\wedge T_j=0$.

In this family it is now classical to define the bifurcation measure on $\C^{d-1}$ as
\[\mu_\bif:=\bigwedge_{j=0}^{d-2}T_j=\frac{1}{(d-1)!}(dd^cL)^{d-1}.\]
This measure is a probability measure on $\C^{d-1}$. Moreover, its support is compact and coincides with the Shilov boundary of the connectedness locus $\mathcal{C}_d:=\{(c,a)\in \C^{d-1}\, ; \J_{P_{c,a}}$ is connected$\}=\{(c,a)\in\C^{d-1}\, ; \ \max_{j}g_j(c,a)=0\}$. For any $n\in\N^*$, we set
\[
D_n:=\sum_{k:k|n}\mu\left(\frac{n}{k}\right)d^k;
\]
$d^n=\sum_{k|n}D_k$ by M\"obius inversion, and
\begin{gather*}
D_n=\begin{cases}
     d_n-1 &\text{if }n=1,\\
     d_n & \text{if }n\ge 1.
    \end{cases}
\end{gather*}
For any $n\in\N^*$, the $n$-th \emph{dynatomic polynomial} of 
$P_{c,a}$ (rather than its lift) is defined as
\[
\Phi_n^*(P_{c,a},z):=\prod_{k:k|n}\left(P_{c,a}^k(z)-z\right)^{\mu(n/k)},
\]
and for any $0\leq j\leq d-1$ and any $n\in\N^*$, we set
\[
\Per_j(n):=\{(c,a)\in\C^{d-1}\, ; \ \Phi_{n}^*(P_{c,a},c_j)=0\}
\]
(cf.\ Subsection \ref{sec:genericity} for $\cM_d^\cm$).
The variety $\Per_j(n)$ is an algebraic hypersurface of $\C^{d-1}$ of degree $D_n$ (hence of degree $d_n$ for $n\geq2$) which is contained in $\{(c,a)\in\C^{d-1}\, ; \ g_{c,a}(c_j)=0\}$. Moreover, the following holds (see~\cite[Theorem 6.1]{favregauthier}).

\begin{theorem}\label{tm:transverspoly}
For any $\un=(n_0,\ldots,n_{d-2})\in(\N^*)^{d-1}$ satisfying $\min_jn_j>1$ and
any $(c,a)\in\bigcap_{j=0}^{d-1}\Per_j(n_j)$ such that $P_{c,a}$ 
has only simple critical points  in $\C$,
the $(d-1)$ hypersurfaces $\Per_j(n_j)$ are smooth and intersect transversely 
 at $(c,a)$.
\end{theorem}

Pick any $\un=(n_0,\ldots,n_{d-2})\in(\N^*)^{d-1}$.
We say a hyperbolic component $\mathcal{H}$ in $\C^{d-1}$ (or the family $(P_{c,a})_{(c,a)\in\C^{d-1}}$)
to be of (disjoint) type $\un$ if for every $(c,a)\in\mathcal{H}$,
$P_{c,a}$ admits $d-1$ distinct attracting periodic orbits 
of respective exact periods $n_0, \ldots , n_{d-2}$  in $\C$.
Then all critical points of $P_{c,a}$ in $\C$
for $(c,a)\in \mathcal{H}$ are simple. For each $0\le i\le d-2$, 
we let $w_i(c,a) \in \D$ be the multiplier of the attracting cycle that 
has exact period $n_i$. In this way we get a holomorphic map
$\mathcal{W}=\mathcal{W}_{\mathcal{H}}:\mathcal{H}\to\D^{d-1}$ defined by
\begin{gather*}
\mathcal{W}(c,a)\pe (w_0(c,a),\ldots, w_{d-2}(c,a)),\quad(c,a)\in\mathcal{H}.
\end{gather*}
The following will also be useful in the sequel.

\begin{theorem}[{see~\cite[Theorem 6.8]{favregauthier}}]
The map $\mathcal{W}:\mathcal{H}\to\D^{d-1}$ is a biholomorphism.\label{tm:parametrizepoly}
\end{theorem}

\subsection{Counting hyperbolic components of disjoint type}
As in the case of rational maps, we denote by $N_\mathcal{P}(\un)$ the number of hyperbolic components of type $\un=(n_0,\ldots,n_{d-2})$ in the family $(P_{c,a})_{(c,a)\in\C^{d-1}}$. When $n_j\neq n_\ell$ for all $j\neq \ell$ and $n_j\geq1$ for all $j$, we have $N_\mathcal{P}(\un)=(d-1)!\cdot d_{|\un|}$. This result is an immediate consequence of Theorem~\ref{tm:parametrizepoly}. Indeed, all such components contain one postcritically finite parameter, counted with multiplicity, and all of them are contained in $\mathcal{C}_d$. The result follows from B\'ezout's Theorem and the fact that
$\deg(\Per_j(n_j))=d_{n_j}$.

Our aim here is to give a good generalization of the above statement, 
{including} the case when $n_j=n_\ell$ for all $j,\ell$.

The first observation is that  any hyperbolic component $\mathcal{H}$ in $\C^{d-1}$
of type $\un$ is contained in the compact set $\mathcal{C}_d$. We rely on the following lemma, which is an immediate adaptation of Lemma~\ref{5.5} (hence we omit the proof).

\begin{lemma}\label{lm:masscomppoly}
For any $\underline{\rho} \in ]0,1[^{d-1}$ and any $\un=(n_0,\ldots,n_{d-2})\in(\N^*)^{d-1}$
with $\min_jn_j\geq2$, the measure $T_{\un}^{d-1}(\underline{\rho})$ has full mass on the union of all hyperbolic components $\Omega\subset\mathcal{C}_d$ such that for all $(c,a)\in \Omega$, $P_{c,a}$ has $d-1$ distinct attracting cycles in $\C$ of respective exact periods $n_0,\ldots,n_{d-2}$. Furthermore, it gives mass $\# \mathrm{Stab}(\un)/d_{|\un|}$ to each of those components. 
\end{lemma}

Here is the precise statement.

\begin{theorem}\label{tm:countgoodpoly}
There exists a constant $C\geq1$ depending only on $d$, such that 
for any $\un=(n_0,\ldots,n_{d-2})\in(\N^*)^{d-1}$ with $\min_jn_j\geq 2$, we have
\[
0\leq 1-\frac{\# \mathrm{Stab}(\un)\cdot N_\mathcal{P}(\un)}{(d-1)!\cdot d_{|\un|}}\leq C\max_{0 \leq j \leq d-2}\frac{\sigma_2(n_j)}{d^{n_j}}.
\]
\end{theorem}

\begin{proof}
Pick any $\un=(n_0,\ldots,n_{d-2})\in(\N^*)^{d-1}$ with $\min_jn_j\geq2$. 
Set $\underline{\rho}:=(1/2,\ldots,1/2)$,
and pick 
a smooth cut-off function $\Psi$ on $\C^{d-1}$
such that $\Psi=1$ on $\mathcal{C}_d$.
Applying Theorem~\ref{tm:vitessemoyennesalgebraic} {yields}
\[
\left|\langle T_{\un}^{d-1}(\underline{\rho}),\Psi\rangle-\langle(dd^cL)^{d-1},\Psi\rangle\right|\leq C\|\Psi\|_{\mathrm{DSH}}^*\max_{0\leq j \leq d-2}\frac{\sigma_2(n_j)}{d^{n_j}},
\]
{where $C>0$ only depends on $\supp(\Psi)$ and $d$}.
As seen in the previous Subsection, we have $\langle(dd^cL)^{d-1},\Psi\rangle=(d-1)!$,
and by Lemma~\ref{lm:masscomppoly} and $\supp(T_{\un}^{d-1}(\underline{\rho}))\subset\mathcal{C}_d$, 
\[
\langle T_{\un}^{d-1}(\underline{\rho}),\Psi\rangle
=\frac{\# \mathrm{Stab}(\un)}{d_{|\un|}}\cdot N_\mathcal{P}(\un).
\]
Now the proof is complete also 
by $N_\mathcal{P}(\un)\leq d_{|\un|}/\# \mathrm{Stab}(\un)$.
\end{proof}

\begin{remark} \normalfont
This result is coherent with the above remark concerning the case $n_j\neq n_\ell$ for all $j\neq \ell$, since in that case, $\mathrm{Stab}(\un)=\{\mathrm{id}\}$.
The above statement can also be interpreted as follows:

\emph{The number of postcritically finite parameters for which all critical points are periodic with prescribed exact periods $n_0,\ldots,n_{d-2}\geq2$ and at least $2$ critical points lie in the same super-attracting cycle, counted with multiplicity of intersection of the $\Per_{\tau(j)}(n_j)$ for all $\tau\in\mathfrak{S}_{d-1}$, is bounded from above 
by $C\max_{j\leq d-2}(\sigma_2(n_j)/d^{n_j})\cdot d^{|\un|}$.}

This is a much better estimate than the one we can obtain by naive arguments. Indeed, 
without taking the multiplicity into account we can naively get a bound from above of the form $C d^{|\un|-\min_jn_j/2}\max_{j\leq d-2}n_j$. Let us give the argument in the case $d=3$ and $n_0=n_1=n$ for simplicity: in that both critical points lie in the same $n$-cycle and one of them is sent to the other after at most $n/2$ iterates. The estimate follows from B\'ezout's Theorem.
\end{remark}

An immediate application of this Theorem is the following:

\begin{corollary}\label{cor:sameperiod}
For any integer $n\geq2$, we have
\[
0\leq 1-\frac{N_\mathcal{P}(n,\ldots,n)}{(d_{n})^{d-1}}\leq C\left(\frac{\sigma_2(n)}{d^{n}}\right).
\]
In particular, we have
\[\frac{N_\mathcal{P}(n,\ldots,n)}{(d_n)^{d-1}}=1+O\left(\frac{\sigma_2(n)}{d^{n}}\right), \ \text{as} \ n\to\infty.\]
\end{corollary}

\begin{proof}
In the present case, we have $\#\mathrm{Stab}(n,\ldots,n)=(d-1)!$ and $d_{|(n,\ldots,n)|}={(d_n)^{d-1}}$. Since $\mu_\bif$ is a probability measure, the result follows from Theorem~\ref{tm:countgoodpoly} above.
\end{proof}

\begin{remark} \normalfont \normalfont\normalfont
In fact, we have proved that, counted with multiplicity, the number of intersection points of the $\Per_j(n)$ for which at least two critical points lie in the same periodic orbit is bounded from above by a constant times $\sigma_2(n)d^{(d-2)n}$.
\end{remark}

\subsection{Distribution of polynomials with $(d-1)$ attracting cycles}
Pick $\un=(n_0,\ldots,n_{d-2})\in(\N^*)^{d-1}$ with $\min_jn_j\geq2$ and 
{$\uw:=(w_0,\ldots,w_{d-2})\in\D^{d-1}$}. As in the case of rational maps, we let $C_{\un,\uw}$ be the (finite) set of parameters $(c,a)\in\C^{d-1}$ such that $P_{c,a}$ has $d-1$ distinct attracting cycles in $\C$ of respective exact periods $n_0,\ldots,n_{d-2}$ and multipliers $w_0,\ldots,w_{d-2}$. We also let
\[
\nu_{\un,\uw}:=\frac{\#\mathrm{Stab}(\un,\uw)}{(d-1)!\cdot d_{|\un|}}\sum_{(c,a)\in C_{\un,\uw}}\delta_{(c,a)}.
\]
The only modification from the case of rational maps is the multiplication by $1/(d-1)!$. From the normalization $\mu_\bif=(dd^cL)^{d-1}/(d-1)!$, we see easily that this factor should also appear in the definition of $\nu_{\un,\uw}$.
A similar argument to that in the proof of Theorem~\ref{tm:vitessecenters} gives the following.
\begin{theorem}\label{tm:vitessepoly}
There exists a constant $C>0$ depending only on $d$ such that
\begin{enumerate}
\item for any $\Psi\in\mathcal{C}^2_c(\C^{d-1})$ and 
any $\un=(n_0,\ldots,n_{d-2})\in(\N^*)^{d-1}$ with $\min_jn_j\geq2$, 
\[
\left|\int_{\C^{d-1}}\Psi\, \nu_{\un,\underline{0}}-\int_{\C^{d-1}}\Psi\, \mu_\bif\right|
\leq C\max_{0\leq j\leq d-2}\left(\frac{\sigma_2(n_j)}{d^{n_j}}\right)\|\Psi\|_{\mathcal{C}^2}.
\]
\item for any $\Psi\in\mathcal{C}^1_c(\C^{d-1})$, any $\uw=(w_0,\ldots,w_{d-2})\in\D^{d-1}$ and any $\un=(n_0,\ldots,n_{d-2})\in(\N^*)^{d-1}$ with 
$\min_jn_j\geq2$,
\[
\left|\int_{\C^{d-1}}\Psi\, \nu_{\un,\uw}-\int_{\C^{d-1}}\Psi\, \mu_\bif\right|
\leq C
\max_{0\leq j \leq d-2}\left( \left(\frac{-1}{d^{n_j}\log|w_j|}\right)^{1/2},
\left(\frac{\sigma_2(n_j)}{d^{n_j}}\right)^{1/2}\right)
\|\Psi\|_{\mathcal{C}^1}.
\]
\end{enumerate}
\end{theorem}

\begin{remark} \normalfont
The key difference with the case of the moduli space $\mathcal{M}_d$ of degree $d$ rational maps is the existence of a \emph{universal} constant $C>0$. This is a consequence of the fact that $C_{\un,\uw}\cup\mathrm{supp}(\mu_\bif)\subset\mathcal{C}_d$, which is compact in $\C^{d-1}$, for all $\un$ and all $\uw$. This compactness property implies the existence of a universal constant $C_1>0$ in the conclusion of Theorem~\ref{tm:vitessemoyennesalgebraic} in the family $(P_{c,a})_{(c,a)\in\C^{d-1}}$. 
\end{remark}

We now come to our last result in the spirit of Theorem B of \cite{distribGV}: 
for any $n\in\N^*$, we want to prove the measure equidistributed on parameters $(c,a)\in\C^{d-1}$ satisfying $c_j\in\Fix^*(P_{c,a}^n)$ for any $0\le j\le d-2$ converges towards the bifurcation measure, with an exponential speed of convergence.

\begin{corollary}
There exists a constant $C>0$ depending only on $d$ such that 
for any integer $n\geq2$ and any $\Psi\in\mathcal{C}_c^2(\C^{d-1})$,
we have
\[\left|\frac{1}{(d_n)^{d-1}}\int_{\C^{d-1}}\Psi\bigwedge_{j=0}^{d-2}[\Per_j(n)]-\int_{\C^{d-1}}\Psi\, \mu_\bif\right|\leq C\cdot \frac{\sigma_2(n)}{d^n}\cdot\|\Psi\|_{\mathcal{C}^2}.\]
\end{corollary}

\begin{proof}
For any integer $n\geq2$, we have
$\bigwedge_{j=0}^{d-2}[\Per_{j}(n)]=(d_{n})^{d-1}\nu_{n,(0,\ldots,0)}$, so that we can directly apply Theorem~\ref{tm:vitessepoly}.
\end{proof}

\bibliographystyle{short}
\bibliography{biblio}
\end{document}